\documentclass{imsart}

\RequirePackage{amsthm,amsmath,amsfonts,amssymb}
\RequirePackage[authoryear]{natbib}
\RequirePackage[colorlinks,citecolor=blue,urlcolor=blue]{hyperref}
\RequirePackage{graphicx}

\startlocaldefs
\usepackage{enumitem}
\numberwithin{equation}{section}
\theoremstyle{plain}
\newtheorem{theo}{Theorem}[section]
\newtheorem{prop}[theo]{Proposition}
\newtheorem{lem}[theo]{Lemma}

\theoremstyle{definition}

\newtheorem{exam}[theo]{Example}
\newtheorem{remark}[theo]{Remark}

\newcommand{\e}{\mathbb{E}}
\newcommand{\E}{\mathbb{E}}
\newcommand{\va}{\mathrm{Var}}
\newcommand{\p}{\mathbb{P}}
\newcommand{\N}{\mathbb{N}}
\newcommand{\R}{\mathbb{R}}
\newcommand{\JJ}{\mathcal{J}}
\newcommand{\DD}{\mathcal{D}}
\newcommand{\A}{\mathcal{A}}

\newcommand{\II}{\mathcal{I}} 

\newcommand{\Aut}{\mathrm{Aut}}

\endlocaldefs

\begin{document}

\begin{frontmatter}
\title{Refined Berry-Esseen bounds under local dependence}
\runtitle{Refined Berry-Esseen bounds under local dependence}

\begin{aug}
\author[A]{\fnms{Zhi-jun}~\snm{Cai}\ead[label=e1]{122312911@mail.sustech.edu.cn}},
\author[B]{\fnms{Qi-man}~\snm{Shao}\ead[label=e2]{shaoqm@sustech.edu.cn}}
\and
\author[B]{\fnms{Zhuo-Song}~\snm{Zhang}\ead[label=e3]{zhangzs3@sustech.edu.cn}}
\address[A]{Department of Statistics and Data Science, Southern University of Science and Technology, Shenzhen, China,\\\printead[presep={ \ }]{e1}}

\address[B]{Department of Statistics and Data Science \textit{and}
 Shenzhen International Center for Mathematics, Southern University of Science and Technology, Shenzhen, China\\\printead[presep={\ }]{e2,e3}}
\end{aug}

\begin{abstract}
	In this paper, we establish Berry--Esseen bounds for both self-normalized and non-self-normalized sums of locally dependent random variables. The proofs are based on Stein's method together with a concentration inequality approach. We develop a new class of concentration inequalities that extend classical results and achieve optimal convergence rates under more general dependence structures. 
As applications, we apply our main results to derive sharper Berry--Esseen bounds for graph dependency, distributed $U$-statistics, constrained $U$-statistics, and decorated injective homomorphism sums.
\end{abstract}



\end{frontmatter}

\section{Introduction}
The study of central limit theorems (CLTs) for dependent random variables has attracted significant attention in probability theory, both for its theoretical interest and practical applications. Several types of dependence have been studied, including martingales \citep{levy1935proprietes}, mixing sequences \citep{PNAS...42...43R,Bolthausen1982}, and locally dependent random variables \citep{Hoeffding1948,ChenandShao}. In general, a family of locally dependent random variables means that random variables within a subset are independent of random variables outside their neighborhood. In this paper, we focus on the CLT for the sum of  locally dependent random variables, and in particular, we establish refined Berry–Esseen type bounds under  the following local-dependence assumptions.
Let $[n]=\{1,2,\cdots,n\} $ be an index set and let  $\{X_i: i \in [n]\}$ be a family of mean-zero random variables. We assume
\begin{itemize}
    \item [(LD1)] For any $i\in[n]$, there exists $A_i\subset[n]$ such that
        $X_i$ is independent of $\{X_j: j\notin A_i\}$.
    \item [(LD2)] For any $i\in[n]$ and $j\in A_i$, there exists $A_{ij}\supset A_i$
        such that $\{X_i,X_j\}$ is independent of $\{X_k: k\notin A_{ij}\}$.
\end{itemize} \par 
The traditional approach for normal approximation relies on the characteristic function method. However, without independence assumptions, it is often difficult to obtain optimal bounds using this approach. In contrast, Stein's method is effective for dependent structures and often yields sharper bounds than traditional approaches.
 An important contribution in this direction is due to \cite{ChenandShao}, who first applied Stein’s method to normal approximation for locally dependent random variables.
Let $S_{n}=\sum_{i\in [n]} X_{i}$, they worked under (LD1) above together with the following condition:
\begin{itemize}
\item[(LD2$'$)] For any $i \in [n]$, there exists $B_i \supset A_i$ such that $\{X_k: k \in A_i\}$ is independent of $\{X_l: l \notin B_i\}$
\end{itemize}
and established a sharp Berry–Esseen bound for $S_{n}$. Building on this framework, later works have further extended these results. For example, \cite{liuandzhang2023cramer-typemoderate} obtained Carm\'{e}r-type moderate deviation results, and \cite{zhang2021berryesseen} derived  Berry–Esseen bounds for self-normalized sums of locally dependent random variables under (LD1) and (LD2$'$). Recently, under (LD1) and (LD2$'$) together with some additional  conditions, \cite{SuUlyanovWang2025PerturbStein} derived an upper bound on the total variation distance,  where the approximating random variable has a distribution that is a Poisson mixture with either a binomial or a negative binomial mixing distribution.\par
Although the dependence structures satisfying (LD1) and (LD2$'$) have found several applications, such as $m$-dependence, they are still not adequate for certain statistics. For example, under the framework of (LD1) and (LD2$'$), the Berry–Esseen bound for graph dependency does not match the corresponding Wasserstein-distance bound (see Remarks \ref{remark-graph} and \ref{remark-graph-2}). This mismatch can be traced back to the fact that the sets $B_i$ in (LD2) are too large. In this context, the previously imposed condition (LD2$'$)  can be replaced by the more flexible dependence structure (LD2) introduced above.

\cite{BARBOUR1989125} proved a Wasserstein-${1}$ bound for decomposable random variables, whose conditions are somewhat weaker than (LD1) and (LD2). Following this line of research, \cite{Raic2004} derived  a multivariate CLT for decomposable random vectors with finite second moments. From a different perspective, \cite{TemcinasNandaReinert2024} extended the notion of decomposability to a multivariate setting and established an abstract multivariate CLT via Stein’s method. 
The dependence structures (LD1) and (LD2) were introduced in \cite{10.1214/19-EJP301}. In that work, a bound for the Wasserstein-2 distance was derived under (LD1), (LD2), and (LD3), where (LD3) is an additional condition extending these two assumptions. \cite{Wasserstein-p(2023)} developed the Wasserstein-$p$ distance under stronger conditions.  For Berry–Esseen-type theorem, \cite{Fang2016Multivariate} established a multivariate Berry--Esseen bounds for bounded decomposable random vectors with the best known rate, whereas \cite{Peter-Kolmogorov-bounds-2023} provided a Berry–Esseen bound for decomposable random variables using the Stein–Tikhomirov method.   \par
A common issue in statistics is that the variance 
$\operatorname{Var}(S_{n})$ is unknown, which makes statistical inference infeasible. To address this, we also investigate central limit theorems for  self-normalized  sums of locally dependent random variables.
In this paper, we develop a Berry--Esseen bound for self-normalized sum of locally dependent random variables under (LD1) and (LD2). Compared to \cite{zhang2021berryesseen},  the size of $A_{ij}$ is usually much smaller than those used in \cite{zhang2021berryesseen}. \par
Our main contributions can be summarized as follows.
First, we establish two Berry--Esseen bounds for sums of locally dependent random variables. 
These bounds are sharper than those of \cite{Peter-Kolmogorov-bounds-2023} in terms of moment conditions and improve upon those of \cite{ChenandShao} and \cite{zhang2021berryesseen} in terms of $\kappa$.
Second, to prove these bounds, we develop two generalized randomized concentration inequalities under (LD1) and (LD2) using a recursive method, based on the ideas of \cite{ChenandShao} and \cite{Chen-recursive}.
Third, by applying main results, we derive Berry--Esseen bounds for some interesting applications, such as graph dependency, distributed $U$-statistics, constrained $U$-statistics, and decorated injective homomorphism sums  with a sharper convergence rate and weaker conditions. 
\par
 \par The proof of our main results is based on Stein's method and  refined concentration inequality approach. The concentration inequality approach is a powerful technique used to obtain the optimal Berry--Esseen bounds. There are many concentration inequalities under different conditions, see \cite{chenandshao2001non-uniform,ChenandShao,chenandshao2007normal}, \cite{shanandzhou2016}, \cite{shaoandzhang2022berry} and so on.  In this paper, we establish a new class of concentration inequalities that extend the classical results. These inequalities are valid in more general settings and lead to optimal convergence rates.\par 
Throughout this paper, we use the following notations. For  two real numbers $a,b$, let $a\wedge b=\min\{a,b\}$ and $a\vee b=\max\{a,b\}.$ For two random variables, \(X \perp\!\!\!\perp Y\) indicates that \(X\) and \(Y\) are independent. For a function \(f\), we write \(\|f\|_\infty = \sup_{x} |f(x)|\).  For any random variable $X$ and $p\geq 1$, let $\|X\|_{p}=(\e|X|^{p})^{1/p}$ be the $L_{p}$ norm of $X.$ For a set $A$, $|A|$ denotes its
cardinality. Finally, unless otherwise specified, the capital letter $C$ stands for an absolute constant whose value may take different values in different places.\par
The rest of this paper is organized as follows. In Section \ref{section-main-result}, we give our main results, including  Berry--Esseen bounds  for non-self-normalized and self-normalized sums of locally dependent  random variables.  Applications are presented in Section \ref{section-applications}. Some preliminary lemmas and randomized concentration inequalities are given in Section \ref{section-lemma} and the proof of main results are postponed to Section \ref{section-proof-of-main-results}.
\section{Main results}\label{section-main-result}
For the sake of clarity, we first revisit our setting.
Let $[n]=\{1,2,\cdots,n\} $ be an index set and let $\left\{X_i: i \in [n] \right\}$ be a field of real-valued random variables satisfying $\e\{ X_{i}\}=0$ for any $i\in [n]$. We consider the following local dependence structure.
\begin{itemize}
    \item [(LD1)] For any $i\in[n]$, there exists $A_i\subset[n]$ such that
        $X_i$ is independent of $\{X_j: j\notin A_i\}$.
    \item [(LD2)] For any $i\in[n]$ and $j\in A_i$, there exists $A_{ij}\supset A_i$
        such that $\{X_i,X_j\}$ is independent of $\{X_k: k\notin A_{ij}\}$.
\end{itemize}
\subsection{Berry--Esseen bounds for  sums of locally dependent random variables}
Let $$
S:=S_{n}=\sum_{i \in[n]} X_{i}\quad \text{and}\quad W_{1}=S/\sigma,$$ where $\sigma^{2}:=\sigma_{n}^{2}=\text{Var}(S)$. For any $i\in [n]$, let $$N_{i}=\{k\in [n]:i\in A_{k}\}\quad \text{and}\quad D_{i}=\{(k,l): l\in A_{k}, \{i\}\cap A_{kl}\neq \varnothing\}.$$
Let $\kappa$ and $\tau$ be two positive constants such that 
\begin{align}\label{eq-definition-of-kappa-tau}
\kappa=\max\big\{ \sup_{i\in [n]} |N_{i}|,\ \sup_{i,j} |A_{ij}| \big\}\quad\text{and}\quad \tau=\sup_{i\in [n]}|D_{i}|.
\end{align}
\begin{theo}\label{thm-main1-1} Under \textup{(LD1)} and \textup{(LD2)}, we have
	\begin{align}\label{eq-thm-11}
		\sup _{z \in \mathbb{R}}|\mathbb{P}(W_{1} \leq z)-\Phi(z)|\leq \frac{C\kappa^{2}}{\sigma^{3}}\sum_{i\in [n]}\|X_{i}\|_{4}^{3}+\frac{C\kappa^{1/2}(\kappa+\tau^{1/2})}{\sigma^{2}}\Big(\sum_{i\in [n]}\|X_{i}\|_{4}^{4}\Big)^{1/2}.
	\end{align}
\end{theo}

\subsection{Berry--Esseen bounds for self-normalized  sums of locally dependent random variables} In this subsection, we present  results for  self-normalized  sums of locally dependent random variables.
Let $$
 V=\sqrt{\Big(\sum_{i\in [n]} (X_{i}Y_{i}-\bar{X}\bar{Y})\Big)_{+}}\quad\text{and}\quad W_{2}=S/V,$$
where
$Y_{i}=\sum_{k\in A_{i}}X_{k}$, $\bar{X}=\sum_{i\in [n]}X_{i}/n, \bar{Y}=\sum_{i\in [n]}Y_{i}/n$ and $(x)_{+}=\max(x,0)$.  We remark that $\e X_{i}=0$ is not a necessary assumption, since one can directly replace   $ X_{i}$ by $X_{i}-\e X_{i}$. For $W_{2}$, we have the following theorem.
\begin{theo}\label{thm-main2}  Under \textup{(LD1)} and \textup{(LD2)}, we have
\begin{align}\label{eq-thm-03}
&\sup _{z \in \mathbb{R}}|\mathbb{P}(W_{2}\leq z)-\Phi(z)|\leq C \lambda\Big(\frac{\kappa^{2}}{\sigma^{3}}\sum_{i \in[n]}\|X_{i}\|_{4}^{3}+\frac{\kappa^{1 / 2}(\kappa+\tau^{1/2})}{\sigma^{2}}\big(\sum_{i \in[n]} \|X_i\|_{4}^4\big)^{1 / 2}\Big),
\end{align}
where
$ \lambda= \kappa \sum_{i \in[n]}\| X_i\|_{2}^{ 2}/\sigma^{2}.$
 \end{theo}
\begin{remark}
  We make some remarks on $\lambda$. In many examples, such as  $U$-statistics, we have
\[
\sigma^2\asymp\kappa \sum_{i \in [n]} \|X_i\|_2^2,
\]
under which \(\lambda\) is of order \(O(1)\).
Specifically, for sums of independent random variables,  $W_{2}$ is  the well-known t-statistic. In this case, $\kappa=\lambda=1$, then the right hand of  (\ref{eq-thm-03}) becomes
  \begin{align*}
     C \sum_{i \in[n]}\|X_{i}\|_{4}^{3}/\sigma^{3}+C\Big(\sum_{i \in[n]} \|X_i\|_{4}^{4}/\sigma^{4}\Big)^{1/2}.
  \end{align*}
  The above bound can  achieve the optimal rate of convergence, but the cost is the existence of finite fourth moment.
\end{remark}
\section{Applications}\label{section-applications}
\subsection{Graph dependency} Let $\{X_i : i \in \mathcal{V}\}$ be a collection of random variables. A graph $\mathcal{G} = (\mathcal{V}, \mathcal{E})$ is called a {\it dependency graph} if, for any two disjoint subsets $\Gamma_1, \Gamma_2 \subseteq \mathcal{V}$ such that no edge in $\mathcal{E}$ connects a vertex in $\Gamma_1$ with a vertex in $\Gamma_2$, the collections of random variables $\{X_i : i \in \Gamma_1\}$ and $\{X_i : i \in \Gamma_2\}$ are independent. This notion and the first quantitative CLT bounds in the Kolmogorov metric were developed by \cite{BaldiandRinott1989}. Let $A_{i}=\{j\in\mathcal{V}: \text{there is an edge connecting}\ i\ \text{and}\ j\}$.  Applying Theorems \ref{thm-main1-1} and \ref{thm-main2}, we have the following theorem.
\begin{theo}\label{thm-graph-dependency}
 For a field of random variables ${X_{i}, i\in \mathcal{V}}$ indexed by the vertices of a dependency graph $\mathcal{G}=(\mathcal{V},\mathcal{E})$ with $\e\{X_{i}\}=0$, let $S=\sum_{i\in \mathcal{V}}X_{i}$, $\sigma^{2}=\text{Var}(S)$, $Y_{i}=\sum_{j\in A_{i}}X_{i}$ and $V=\sqrt{\big(\sum_{i\in [n]} (X_{i}Y_{i}-\bar{X}\bar{Y})\big)_{+}}$. Put $$W_{1}=S/\sigma,\quad W_{2}=S/V,$$   then
 \begin{align}\label{eq-app2-01}
 \sup_{z\in\R}|\p(W_{1}\leq z)-\Phi(z)|\leq C d^{2}\sum_{i\in \mathcal{V}}\|X_{i}\|_{4}^{3}/\sigma^{3}+Cd^{3/2}\big(\sum_{i\in \mathcal{V}} \|X_{i}\|_{4}^{4}/\sigma^{2} \big)^{1/2}
 \end{align}
 and
 \begin{align}\label{eq-app2-02}
  \sup_{z\in\R}|\p(W_{2}\leq z)-\Phi(z)|
  \leq C\lambda_{1}\Big( d^{2}\sum_{i\in \mathcal{V}}\|X_{i}\|_{4}^{3}/\sigma^{3}+Cd^{3/2}\big(\sum_{i\in \mathcal{V}} \|X_{i}\|_{4}^{4}/\sigma^{4} \big)^{1/2}\Big),
 \end{align}
 where $\lambda_{1}=d\sum_{i\in\mathcal{V}}\e|X_{i}|^{2}/\sigma^{2}$, $C$ is an absolute constant and  $d$ is the maximal degree of $\mathcal{G}$.
\end{theo}
\begin{remark}\label{remark-graph}
  For Wasserstein distance, \citet[Theorem 3.6]{ross2011} proved that
  \begin{align}\label{eq-app2-03}
   \mathcal{W}_{1}(\mathcal{L}(W_{1}), N(0,1)) \leq C d^{2}\sum_{i\in \mathcal{V}}\e|X_{i}|^{3}+Cd^{3/2}\big(\sum_{i\in \mathcal{V}} \e|X_{i}|^{4} \big)^{1/2}.
  \end{align}
 Our bounds match the Wasserstein distance proved by \cite{ross2011}. 
  \end{remark}

  \begin{remark}\label{remark-graph-2}
For Berry--Esseen bounds, \cite{ChenandShao} showed that for any $p\in (2,3]$,
  \begin{align}\label{eq-app2-04}
     \sup_{z\in\R}|\p(W_{1}\leq z)-\Phi(z)|\leq Cd^{5(p-1)}\sum_{i\in \mathcal{V}} \e|X_{i}|^{p}/\sigma^{p}.
  \end{align}
Recently, \cite{JanischLehericy2024} established Berry--Esseen bounds with improved dependence on $d$,  under  a finite $\delta$-moment with $\delta\in(2,\infty]$.
When $\delta=\infty$, the bound is essentially optimal, namely
  \begin{align}\label{eq-app2-04.1}
\sup_{z\in\R}|\p(W_{1}\leq z)-\Phi(z)|
\le
\max\!\Big\{\,68.5(d+1)^2\,\sum_{i\in\mathcal{V}}\frac{\e|X_{i}|^{3}}{\sigma^3} ,\; 22.88\,\frac{L(d+1)}{\sigma}\Big\},
\end{align}
where $L=\max_{i\in\mathcal{V}}\|X\|_{\infty}.$
For the self-normalized case, \cite{zhang2021berryesseen} proved that
  \begin{align} \label{eq-app2-05}
  \begin{aligned}   
  	&\sup_{z\in \R}|\p(W_{2}\leq z)-\Phi(z)|\\
    &\quad\leq Cd^{9}\big[1+d^{6}n^{1/3}\big(\sum_{i \in \mathcal{V}}\e|X_{i}|^{3}/\sigma^{3}\big)^{2/3}\big]\cdot\big(n^{-1/2}+\sum_{i \in \mathcal{V}}\e|X_{i}|^{3}/\sigma^{3}\big).
  \end{aligned} 
  \end{align}
Compared with (\ref{eq-app2-04}) and (\ref{eq-app2-05}), our results are sharper in terms of  $d$, and compared with \eqref{eq-app2-04.1}, our moment condition is weaker.
  \end{remark}
\subsection{\texorpdfstring{Distributed $U$-statistics}{Distributed U-statistics}}
Let $X_1, \dots, X_N$ be independent random variables with an unknown distribution $P$, and let $h: \R^m \to \R$ be a measurable symmetric function such that
$
\e[h(X_1, \dots, X_m)] = \theta.
$
An unbiased estimator of $\theta$ is given by
\begin{align}\label{eq-ap-00}
    U_{N} = \binom{N}{m}^{-1} \sum_{1 \le i_1 < \cdots < i_m \le N} h(X_{i_1}, \ldots, X_{i_m}).
\end{align}
The statistic defined in (\ref{eq-ap-00}) is the $U$-statistic,  introduced by \cite{Hoeffding19481}, where $h$ is the kernel function and $m$ is the degree of the $U$-statistic. Many classical statistics are $U$-statistics, including the sample variance,  Kendall's $\tau$ for measuring correlation and Gini's mean difference.\par
However, computing the $U$-statistic in (\ref{eq-ap-00}) requires a computational cost that grows like $N^m$, 
which  become prohibitively expensive for large $N$.  To reduce the computational cost, and inspired by the ``divide-and-conquer" idea, \cite{LIN201016} proposed computing a distributed $U$-statistic as a replacement for the classical $U$-statistic.
The {\it distributed $U$-statistic} is defined as follows. 
First, the random sample is partitioned into $k$ subsets. 
The observations in the $i$-th subset are denoted by $\mathbb{S}_{i}:=\{X_{i1}, \ldots, X_{in_i}\}$, 
and $U_{N,i}$ denotes the $U$-statistic computed from this subset.
The distributed $U$-statistic is then defined as  the following weighted average of $U_{N,1},\cdots,U_{N,k}$:
\begin{align}\label{eq-ap-01}
    U_{d} := \frac{1}{N} \sum_{i=1}^{k} n_i U_{N,i}
    = \sum_{i=1}^{k} \frac{n_i}{N} \cdot \binom{n_i}{m}^{-1} \sum_{1 \le j_1 < \cdots < j_m \le n_i} h(X_{ij_1}, \dots, X_{ij_m}).
\end{align}
The time complexity of
computing the distributed $U$-statistics in (\ref{eq-ap-01}) is approximately $k(N/k)^{m}$. Compared with the $O(N^m)$ complexity of the classical $U$-statistic, the computational cost is significantly reduced. The remaining question is whether the distributed $U$-statistic is statistically equivalent to the classical $U$-statistic.
For the non-degenerate case, i.e., $\sigma_1^2 = \mathbb{E}[g^2(X_1)] > 0$, 
where $g(x) = \mathbb{E}[(h(X_1, \ldots, X_m) - \theta) \mid X_1 = x]$, 
\cite{LIN201016} proved that if $\mathbb{E}[h^2(X_1, \ldots, X_m)] < \infty$ and 
 $k = o(N)$, then
\[
W := \frac{\sqrt{N}}{m \sigma_1} (U_d - \theta) \xrightarrow{d} N(0,1) \quad \text{as } N \to \infty.
\]
\cite{chensongxi2021} studied generalized distributed symmetric statistics. 
Specifically, they considered symmetric statistics of the form
\[
T_N = \theta + N^{-1} \sum_{i=1}^N \alpha(X_i)
      + N^{-2} \sum_{1 \le i < j \le N} \beta(X_i, X_j) + R_N,
\]
where $\theta$ is the parameter of interest, 
$\alpha(x)$ and $\beta(x, y)$ are known functions, 
and $R_N$ is a remainder term. By Hoeffding’s decomposition, $U$-statistics is a  special case of $T_{N}.$ To demonstrate the feasibility of the distributed approach, they study some properties of distributed $U$-statistics with degree $2$ in the appendix. Specifically, suppose that $\e |g(X_{1})|^{3}<\infty$, $\e\left\{h\left(X_1, X_2\right)\right\}^2< \infty$,   $c_1 \leq \inf _{k_1, k_2} n_{k_1} / n_{k_2} \leq \sup _{k_1, k_2} n_{k_1} / n_{k_2} \leq c_2$ for some positive constant $c_{1}$ and $c_{2}$, $k=O(N^{a})$ for some positive constant  $a<1 / 2$, they showed that as $N \rightarrow \infty$,
$$
\begin{aligned}
& \sup _{x \in \R}\big|\p(W \leq x)-\p(N^{1 / 2}(U_N-\theta)/(2\sigma_{1}) \leq x)\big|= o(N^{-1 / 2}).
\end{aligned}
$$

 Applying Theorem \ref{thm-main1}, we obtain a Berry--Esseen bound for $W.$
\begin{theo}\label{thm-distributed-U-statistics} 
If $\e h^{4}(X_{1},\cdots, X_{m})<\infty$, $\sigma_{1}>0$ and  $n_{i}\geq m$, $1\leq i\leq k$, then 
	\begin{align}\label{eq-ap-02}
	\sup_{z\in\R}\Big|\p\Big(\frac{W}{\sqrt{\mathrm{Var}(W)}}\leq z\Big)-\Phi(z)\Big|\leq \frac{Cm}{\sqrt{N}}\cdot\frac{\|h\|_{4}^{3}}{\sigma_{1}^{3}},
	\end{align}
	where $\|h\|_{p}=\big(\e|h(X_{1},\cdots,X_{m})|^{p}\big)^{1/p}$.
\end{theo}
\begin{remark}
 Let $\sigma^{2}=\operatorname{Var}(h(X_{1}, X_{2},\cdots, X_{m}))$,  by \citet[Lemma 10.1]{chen2011normal}, we have 
    \[
\Big|\frac{n_{i}}{m^{2} \sigma_{1}^{2}}  \operatorname{Var}(U_{N,i})-1\Big|\leq \frac{(m-1)^{2}\sigma^{2}}{m(n_{i}-m+1)\sigma_{1}^{2}},
    \]
    which further implies that 
 \begin{align}\label{eq-ap-03}
    \big|\operatorname{Var}(W)-1\big|\leq \frac{1}{N} \sum_{i=1}^{k} n_{i}\cdot \Big|\frac{n_{i}}{m^{2} \sigma_{1}^{2}}  \operatorname{Var}(U_{N,i})-1\Big|\leq \frac{Cm\sigma^{2}}{N\sigma_{1}^{2}}\sum_{i=1}^{k}\frac{n_{i}}{n_{i}-m+1}.   
 \end{align}
Combine Theorem \ref{thm-distributed-U-statistics} and (\ref{eq-ap-03}), we can also obtain a Berry-Esseen bound for $W$: 
	\begin{align}\label{eq-ap-04}
	\sup_{z\in\R}\big|\p(W\leq z)-\Phi(z)\big|\leq \frac{Cm}{\sqrt{N}}\cdot\frac{\|h\|_{4}^{3}}{\sigma_{1}^{3}}+\frac{Cm\sigma^{2}}{N\sigma_{1}^{2}}\sum_{i=1}^{k}\frac{n_{i}}{n_{i}-m+1}.
	\end{align}
If $n_{i}\geq 2m$ and $k=O(\sqrt{N})$, the right-hand side of (\ref{eq-ap-04}) is of order $O(1/\sqrt{N})$.  
\end{remark}
\begin{remark} 
Compared with \cite{chensongxi2021}, our results remove condition $k=O(N^{a})$ for some positive constant  $a<1 / 2$  by normalising $W$, but the cost is the existence of finite forth moment. In the proofs of \cite{LIN201016} and \cite{chensongxi2021}, each $\sqrt{n_i} U_{N,i}/(m \sigma_1)$ is decomposed as $W_i + \Delta_i$, where $W_i$ is asymptotically normal and $\Delta_i \xrightarrow{p} 0$ as $n_i \to \infty$. Consequently, the final statistic $W$ can be written as a main term plus a remainder $\Delta = \sum_{i=1}^k \sqrt{n_i/N}\, \Delta_i$, which may not be small if the subset sizes $n_i$ are too small.  Intuitively, this approach adds up the errors from each block. So, this requires that the number of blocks is not too large.
However, the distributed $U$-statistics  can be regarded as a sum of certain dependent random variables.  Thus, by handling it as a whole, we can avoid the problem of error accumulation.
\end{remark}
\subsection{\texorpdfstring{Constrained $U$-statistics for $m$-dependent random variables}{Constrained U-statistics for m-dependent random variables}}
Let $X_{1}, \cdots, X_{n}$ be a sequence of stationary $m$-dependent random variables for some fixed integer $m\geq 0.$ Given a constraint $\mathcal{D}=(d_{1},\cdots,d_{l-1})\in \bar{\N}^{l}$, where $\bar{\N}=\N\cup\{\infty\}$, the constraint $U$-statistic is defined by
\begin{align}\label{eq-r1-01}
  U_{n}(f;\DD):=\sum_{\substack{ 1\leq i_{1}<\cdots<i_{l}\leq n\\ i_{j+1}-i_{j}\leq d_{j}}}f(X_{i_{1}},\cdots,X_{i_{l}}),
\end{align}
where $f: \R^{l}\to \R$ is a measurable function, not necessarily symmetric. Constrained $U$-statistics were first studied by \cite{janson_2023}, motivated by problems in pattern matching in random strings and permutations. To illustrate the significance of studying constrained $U$-statistics, we present two examples as follows.\par
\begin{exam}
  For a given finite alphabet $\A$, a random string from $\A$ of length $n$ can be represented as $\Xi_{n} = \xi_{1} \cdots \xi_{n}$, where $\xi_{1}, \dots, \xi_{n}$ are i.i.d. random variables taking values in $\A$. Give a word $\mathbf{w}=w_{1}\cdots w_{l}$ $(l\leq n)$ and $\mathcal{D}=(d_{1},\cdots,d_{l-1})\in \bar{\N}^{l-1}$,  an {\it occurrence} of $\mathbf{w}$ in $\Xi_{n}$ with constraint $\DD$ is a sequence of indices $i_{1}, i_{2},\cdots, i_{l}\in [n]$ satisfying that
\begin{align}\label{eq-r1-02}
 0<i_{j+1}-i_{j}\leq d_{j},\ 1\leq j<l\quad \text{and}\quad \xi_{i_{k}}=w_{k}\quad \text{for each}\ k\in [l].
\end{align}
Let $N_{n}(\mathbf{w}; \DD)$ denote the number of occurrences of $\mathbf{w}$ with constraint $\DD$ in $\Xi_{n}$. Then, one can check that
\begin{align}\label{eq-r1-03}
  N_{n}(\mathbf{w}; \DD)=U_{n}(f;\DD)
\end{align}
with
\begin{align}\label{eq-r1-04}
  f(x_{1},\cdots,x_{l})=\mathbf{1}(x_{1}x_{2}\cdots x_{l}=\mathbf{w})=\mathbf{1}(x_{i}=w_{i},\ \forall\ i\in [l]).
\end{align}
That is, $N_{n}(\mathbf{w}; \DD)$ is a special case of constrained $U$-statistic.
\end{exam}

There are many applications for string matching, such as the intrusion detection in the area of computer security,  gene searching and so on.
For more details, refer to  \cite{flajolet2006hidden, jacquet2015analytic, nicodeme2002motif}.  Another similar example for constrained $U$-statistic is pattern matching in random permutations.
\begin{exam}
  Let $\mathcal{G}_{n}$ denote the set consisting of all permutations of $[n].$ For $\pi=\pi_{1}\cdots\pi_{n}\in \mathcal{G}_{n}$ and $\tau=\tau_{1}\cdots\tau_{l}\in \mathcal{G}_{l}$, where $l\leq n$, an {\it occurrence} of $\tau$ in $\pi$ with constraint $\DD$  is a sequence of indices $i_{1}, i_{2},\cdots, i_{l}\in [n]$ satisfying that
\begin{align}\label{eq-r1-05}
  0<i_{j+1}-i_{j}\leq d_{j}\ \text{for}\ 1\leq j<l\quad\text{and}\quad \pi_{i_{k}}<\pi_{i_{l}} \Longleftrightarrow \tau_{k}<\tau_{l}\ \text{for}\ k\neq l\in [l].
\end{align}
That is, $\pi_{i_{1}}\cdots\pi_{i_{l}}$ has the same order relations as $\tau_{1}\cdots\tau_{l}.$ Let $N_{n}(\tau, \DD)$ denote the number of the occurrence of $\tau$ in $\pi$, where $\pi$ has uniformly distributed in $\mathcal{G}_{n}.$ Note that a random variable $\pi\in \mathcal{G}_{n}$ can be generated as follows. Generating a sequence of i.i.d. random variables $\{X_{i}\}_{i=1}^{n}$ with uniform distribution in $(0,1)$, and then replacing the values of $X_{1},\cdots, X_{n}$, in increasing order, by $1,\cdots,n$.  Then, $N_{n}(\tau, \DD)$ is a constrained $U$-statistic with
\[
f(x_{1},\cdots, x_{l}):=\prod_{1\le i<j\le l} \mathbf{1}\!\big\{(x_i-x_j)(\tau_i-\tau_j) > 0 \,\big\}.
\]
If $d_{1}=\cdots=d_{l-1}=\infty$, we denote $\DD$ by $\DD _{\infty}$, \cite{bona2007copies} proved that $N_{n}(\tau, \DD _{\infty})$ is asymptotically normal. For {\it vincular} $\DD\ (\text{i.e.}\  d_{i}\in \{1,\infty\})$, \cite{Hofer2017ACL} developed the asymptotic normality for $N_{n}(\tau, \DD)$.
\end{exam}
\cite{janson_2023} obtained a Berry--Esseen theorem for constrained $U$-statistics with bounded function $f$. For a given constrained $\DD $, define
\[
b:=b(\DD)=l-|\{k\in [l-1]: d_{k}<\infty\}|=1+|\{k\in [l-1]: d_{k}=\infty\}|.
\]
Let
\begin{align}\label{eq-r2-01}
W_{1}=\frac{U_{n}(f;\DD)-\e\big\{U_{n}(f;\DD)\big\}}{\sqrt{\text{Var}(U_{n}(f;\DD))}},
\end{align}
if $\lim_{n\to \infty}\text{Var}(U_{n}(f;\DD)/n^{2b-1}=\sigma_{f,\DD}^{2}>0$,
then
\begin{align*}
  \sup_{z\in\R}|\p(W_{1}\leq z)-\Phi(z)|=O(n^{-1/2}).
\end{align*}
In this paper, applying Theorems \ref{thm-main1-1} and \ref{thm-main2}, we obtain  Berry--Esseen bounds for non-self-normalized  and self-normalized constrained $U$-statistics. To define self-normalized constrained $U$-statistics, we first introduce some notation.
Let $\mathcal{I}\subset [n]^{l}$ denote the set of all indices $(i_{1},\cdots, i_{l})$ in the sum (\ref{eq-r1-01}), then  $U_{n}(f;\DD)-\e\{U_{n}(f;\DD)\}$ can be rewritten as  $$U_{n}(f;\DD)-\e\{U_{n}(f;\DD)\}=\sum_{i\in\II}Z_{i},$$ where $Z_{i}:=Z_{(i_{1},\cdots, i_{l})}=f(X_{i_{1}},\cdots, X_{i_{l}})-\e\{f(X_{i_{1}},\cdots, X_{i_{l}})\}.$ For any $i=(i_{1},\cdots,i_{l})\in \II$, note that $i_{1}<i_{2}<\cdots<i_{l}$,  let $s(i)$ denote the set consisting of the elements of each column of the vector $i,$ that is $
s(i)=\{i_{1},\cdots, i_{l}\}.$ In addition, for $i,j\in \II$, define
\begin{align*}
|i-j|=\min_{p\in s(i),q\in s(j)}|p-q|.
\end{align*}  Moreover, for any $i\in \II$, define
\begin{align*}
   A_{i1}&=\{k\in \II:  s(k)\cap s(i)\neq \varnothing\},\quad A_{i2}=\{k\in \II:  s(k)\cap s(i)=\varnothing\ \text{and}\ |i-k|\leq m\}.
\end{align*}
Let $A_{i}=A_{i1}\cup A_{i2}$ and $Y_{i}=\sum_{i\in A_{i}}Z_{i},$ then the self-normalized constrained  $U$-statistic is defined by
\begin{align}\label{eq-r1-09}
  W_{2}=\sum_{i\in \II}Z_{i}/V,\quad V=\sqrt{\Big(\sum_{i\in\II} Z_{i}Y_{i}-\bar{Z}\bar{Y}\Big)_{+}},
\end{align}
where $\bar{Z}=\sum_{i\in \II}Z_{i}/|\II|$ and $\bar{Y}=\sum_{i\in \II}Y_{i}/|\II|$. We have the following theorem.
\begin{theo}\label{thm-constrained-U-statistics}
For a fixed $\DD\in \bar{\N}^{l}$, let $W_{1}$ and $W_{2}$ be define by (\ref{eq-r2-01}) and (\ref{eq-r1-09}) respectively. For $i=(i_{1},\cdots,i_{l})\in \II$, denote $f_{i}$ by $f(X_{i_{1}},\cdots,X_{i_{l}})$.  Assume that $\e |f_{i}|^{4}<\infty$ and $\lim_{n\to \infty}\text{Var}(U_{n}(f;\DD)/n^{2b-1}=\sigma_{f,\DD}^{2}>0$, then 
\begin{align}\label{eq-r1-10}
	&\sup_{z\in \R} |\p(W_{1}\leq z)-\Phi(z)|\leq C_{m,l}\Big(\frac{n^{-b-1/2}}{\sigma_{f,\DD}^{3}} \sum_{i\in \II}\|f_{i}\|_{4}^{3}+\frac{n^{-b/2-1/2}}{\sigma_{f,\DD}^{2}} \big(\sum_{i\in \II}\|f_{i}\|_{4}^{4}\big)^{1/2}\Big)
\end{align}
and
\begin{align}\label{eq-r1-11}
\begin{aligned}
     &\sup_{z\in \R} |\p(W_{2}\leq z)-\Phi(z)|\\
 &\quad\leq C_{m,l}\Big(\frac{n^{-b-1/2}}{\sigma_{f,\DD}^{3}} \sum_{i\in \II}\|f_{i}\|_{4}^{3}+\frac{n^{-b/2-1/2}}{\sigma_{f,\DD}^{2}} \big(\sum_{i\in \II}\|f_{i}\|_{4}^{4}\big)^{1/2}\Big)\cdot\frac{n^{-b}}{\sigma_{f,\DD}^{2}} \sum_{i\in \II}\|f_{i}\|_{2}^{2},
\end{aligned}
\end{align}
where  $C_{m,l}$ is a positive constant depending on $m$, $l$, $f$ and $\DD,$ and $\sigma_{f,\DD}$ is a positive constant depending on $f$ and $\DD.$
\end{theo}
\begin{remark} If $f$ is bounded, note that $|\II|=O(n^{b})$, then by Theorem \ref{thm-constrained-U-statistics}, we have for $k=1$ or $2$,
\[
\sup_{z\in \R} |\p(W_{k}\leq z)-\Phi(z)|\leq C_{m,l}\max\big\{\|f\|_{\infty}^{5}/\sigma_{f,\DD}^{5},\ 1\big\}\cdot n^{-1/2},
\]
where $\|f\|_{\infty}=\sup_{(x_{1},\cdots,x_{l})\in \R^{l}}|f(x_{1},\cdots,x_{l})|.$
\end{remark}
\begin{remark} If for any $i\neq j\in\II$,
$\e |f_{i}|^{4}=\e |f_{j}|^{4}$, it follows from (\ref{eq-r1-10}) and (\ref{eq-r1-11}) that for $k=1$ or $2$,
\begin{align}\label{eq-r1-12}
  \sup_{z\in \R} |\p(W_{k}\leq z)-\Phi(z)|\leq C_{m,l}\max\big\{\|f_{i}\|_{4}^{5}/\sigma_{f,\DD}^{5},\ 1\big\}\cdot n^{-1/2}.
  \end{align}
\end{remark}
\begin{remark}
Replacing $i_{j+1}-i_{j}\leq d_{j}$ with $i_{j+1}-i_{j}=d_{j}$ in (\ref{eq-r1-01}) if $d_{j}<\infty$ yields another similar statistic in \cite{janson_2023}, namely {\it exactly constrained U-statistics}. Similarly, we can construct  non-self-normalized  and self-normalized  exactly constrained $U$-statistics, denoted $W_{3}$ and $W_{4}$. Applying Theorems \ref{thm-main1-1} and \ref{thm-main2}, we can obtain similar Berry–Esseen bounds for $W_{3}$ and $W_{4}$, with coefficients $C_{m,l}$ different from those for $W_{1}$ and $W_{2}$.
\end{remark}

\subsection{Decorated injective homomorphism sums} 
In this subsection, we are interested in central limit theorems for decorated random graphs models, which extend the Erd\H{o}s--R\'enyi random graph model.
 Decorated graphs  were introduced by \cite{lovasz2010limits} and subsequently developed in various directions, including Banach-space decorated graphons \citep{kunszentikovacs2022multigraph} and probability-valued graphons \citep{abraham2025probability}, as well as statistical inference for decorated graphons in multiplex networks \citep{dufour2024inference}. However, central limit theorems  for decorated subgraph statistics seem to have received much less attention in this generality so far.\par
Let $\mathcal{B}$ be an arbitrary measurable space.
An \emph{$\mathcal{B}$-decorated graph} is a pair $(G,g)$, where $G$ is a simple (loopless) graph
with vertex set $V(G)$ and edge set $E(G)$, and $g:E(G)\to \mathcal{B}$.
For fixed $n\ge 1$, let $\{X_{ij}:1\le i<j\le n\}$ be independent $\mathcal{B}$-valued random variables,
and let $G_n$ be the complete graph on the vertex set $[n]$.
We define a symmetric edge decoration $g$ on $G_n$ by
\[
g(i,j)=g(j,i):=X_{ij}\quad\text{if } i\neq j.
\]
 For a fixed $\mathcal{B}$-decorated graph $(F,f)$, a natural
 decorated subgraph counting statistic associated with $(F,f)$ in $(G_n,g)$
 is given by the \emph{decorated injective homomorphism sum}
\begin{align}\label{eq-definition-of-sh}
S_{n}=S_{n,h}:=\sum_{\substack{\varphi:V(F)\to[n]\\ \varphi\ \mathrm{injective}}}
\prod_{uv\in E(F)} h\big(f(uv), g(\varphi(u),\varphi(v))\big),
\end{align}
where $h:\mathcal{B}\times \mathcal{B}\to \R$ is a measurable function. 
There are many examples that can be written in the form of \eqref{eq-definition-of-sh}.
We illustrate this with the following two examples; for further examples, we refer the reader to \cite{lovasz2010limits}.

\begin{exam}[Erd\H{o}s--R\'enyi graph model] The Erd\H{o}s--R\'enyi random graph $G(n,p)$ is defined on the vertex set
$[n] = \{1,\dots,n\}$, where each pair of vertices $i$ and $j$ is independently
connected by an edge with probability $p$ and is absent with probability $1-p$.
For a fixed graph $F$ with at least one edge, let $N(F,G(n,p))$ denote the number of unlabeled  subgraphs of $G(n,p)$ that are isomorphic to $F$.
Let $\mathcal{B}=\{0,1\}$, let $f\equiv 1$, $h(x,y)=xy$, and let $\{X_{ij}:1\le i<j\le n\}$ be i.i.d.\
Bernoulli random variables with parameter $p$. Then,
\[
S_{n}=|\Aut(F)|\,N\!\left(F,G(n,p)\right),
\]
where $\Aut(F)$ denotes the automorphism group of $F$.
\end{exam}
\par
There is an extensive literature on central limit theorems for $N(F,G(n,p))$.
In particular, \cite{Rucinski1988} obtained necessary and sufficient conditions for a central limit theorem to hold for $N(F,G(n,p))$.
Moreover, several works have developed quantitative normal approximation bounds in various metrics; see, e.g., \cite{BARBOUR1989125,10.1214/19-EJP301,Peter-Kolmogorov-bounds-2023}.
We also refer to \cite{ShaoZhang2025} for central limit theorems for weighted triangle counts in inhomogeneous random graphs.
\par
\begin{exam}[Edge-colored complete graph]
Let $\mathcal{B}=K$, where $K=\{c_1,c_2,\dots,c_d\}$ is a color set, and take
\[
h(x,y)=\mathbf{1}_{\{x=y\}}.
\]
Assume that $\{X_{ij}\}$ are i.i.d.with common distribution $\mu$ on 
$K$.
In this case, $S_n(h)$ counts injective copies of $F$ in $G_n$ whose edge colors match those prescribed by $f$.
\end{exam} We shall apply Theorems \ref{thm-main1-1} and \ref{thm-main2}  to establish Berry–Esseen bounds for $S_{n}$, in both the non-self-normalized and self-normalized cases. To this end, we first rewrite $S_{n}$  as a sum of locally dependent random variables. 
Introduce the index set of all injective vertex maps
\[
\mathcal{I}
:= \big\{\varphi : V(F)\hookrightarrow [n]\big\}.
\]
For $\varphi\in\mathcal{I}$, define
\[
\xi_\varphi
:= \prod_{uv\in E(F)}
   h\big(f(uv), g(\varphi(u),\varphi(v))\big).
\]
It follows from the definition of $S_{n}$ that
\begin{align}\label{eq-pa-01}
    S_n = \sum_{\varphi\in\mathcal{I}} \xi_\varphi.
\end{align}
Each $\varphi\in\mathcal{I}$ gives an image subgraph $G_\varphi$ on vertex set $\varphi(V(F))$
with edge set $\{\{\varphi(u),\varphi(v)\} : uv\in E(F)\}$.
For each $\psi \in \mathcal{I}$, let
$$
A_{\varphi}=\left\{\psi \in \mathcal{J}: e\left(G_{\psi} \cap G_{\varphi}\right) \geq 1\right\}\quad \text{and}\quad Y_{\varphi}=\sum_{\psi\in A_{\varphi}} \xi_{\psi}-\e\{\xi_{\psi}\}.
$$
Define 
\begin{align}\label{eq-app4-01}
 W_{1}=\frac{S_{n}-\e\{S_{n}\}}{\operatorname{Var}(S_{n})}, \quad \text{and}\quad W_{2}=\frac{S_{n}-\e\{S_{n}\}}{V},
\end{align}
where $\eta_{\varphi}=\xi_{\varphi}-\e\{\xi_{\varphi}\},$ 
\begin{align*}
 V=\sqrt{\Big(\sum_{\varphi \in \mathcal{I}} \eta_{\varphi} Y_{\varphi}-\bar{\eta} \bar{Y}\Big)_{+}},\quad \bar{\eta}=\sum_{\varphi \in \mathcal{I}}\frac{\eta_{\varphi}}{|\mathcal{I}|}\quad \text { and } \quad\bar{Y}=\sum_{\varphi\in \mathcal{I}} \frac{Y_{\varphi}}{|\mathcal{I}|}. 
\end{align*}
Theorems \ref{thm-main1-1} and \ref{thm-main2} lead to the following theorem.
\begin{theo}\label{thm-app4} Let $W_{1}$ and $W_{2}$ be defined as in (\ref{eq-app4-01}) and $v=v(F)$, then
\begin{align}\label{eq-app9-01}
  \sup_{z\in\R}|\p(W_{1}\leq z)-\Phi(z)|\leq  Cn^{-2v-4}\sum_{i\in \mathcal{I}}\frac{\e|\eta_{i}|^{3}}{\sigma_{n}^{3}}+C\Big(n^{-3v-6}\sum_{i\in \mathcal{I}} \frac{\e|\eta_{i}|^{4}}{\sigma_{n}^{4}}\Big)^{1/2}
\end{align}
and 
\begin{align}\label{eq-app9-02}
  \sup_{z\in\R}|\p(W_{2}\leq z)-\Phi(z)|\leq  C\lambda_{2} n^{-2v-4}\sum_{i\in \mathcal{I}}\frac{\E|\eta_{i}|^{3}}{\sigma_{n}^{3}}+C\lambda_{2}\Big(n^{-3v-6}\sum_{i\in \mathcal{I}} \frac{\E|\eta_{i}|^{4}}{\sigma_{n}^{4}}\Big)^{1/2},
\end{align}
where $\sigma_{n}^{2}=\operatorname{Var}(S_{n})$,  $\lambda_{2}=n^{-v-2}\sum_{i\in \mathcal{I}}\E|\eta_{i}|^{2}/\sigma_{n}^{2}$ and $C$ is a positive constant only depending on $F.$
\end{theo}
\begin{remark} 
We now make a few remarks on $\operatorname{Var}(S_n)$. For subgraph counts in the Erd\H{o}s--R\'enyi random graph $G(n,p)$, by (3.7) of \cite{BARBOUR1989125} , $\va(S_{n})\geq C n^{2v-2}$ if $p$ is independent of $n$.  In this case, since $|\mathcal I| = O(n^{v})$, applying \eqref{eq-app9-01} and \eqref{eq-app9-02} yields, 
for $k=1,2$,
\[
\sup_{z\in\R}\bigl|\p(W_{k}\le z)-\Phi(z)\bigr|
   \le C n^{-1}.
\] 
To the best of our knowledge, this Berry--Esseen bound for $W_2$ is new.
\end{remark}

\section{Preliminary lemmas and propositions}\label{section-lemma}
In this section, we first give some preliminary lemmas which will be used in the proofs of  Theorems \ref{thm-main1-1}--\ref{thm-main2}, and then provide two generalized concentration inequalities under (LD1) and (LD2).  Throughout this section, for any subset $A\in [n]$, we denote 
\[
N_{A}=\{k\in[n]: A_{k}\cap A\neq \varnothing\},\quad N_{A}^{c}=[n]\backslash N_{A}\quad \text{and}\quad X_{A}=\{X_{i}: i\in A\}.
\]
Let \(\xi_A\) be an arbitrary non-negative measurable function of \(X_A\). Define 
\begin{align*}
	D_{A}&=\{(i,j): j\in A_{i}, A\cap A_{ij}\neq \varnothing\},\quad &&D_{A,1}=\{(i,j): j\in A_{i}, A\subset (A_{i}\cup A_{j})^{c}\},\nonumber\\
	D_{A,2}&=D_{A,1}\cap D_{A}^{c},\qquad 
	 &&D_{A,3}=D_{A,1}\cap D_{A}.
\end{align*}
In addition, for any subset $A\subset [n]$, let
\begin{align*}
	S_{n,A}=\sum_{i\in N_{A}^{c}}X_{i}.
\end{align*}
For simplicity of notation, in what follows we will omit the subscript $n$ whenever no confusion arises, and simply write
$
S_A := S_{n,A}.
$
\subsection{Some preliminary lemmas}
The following lemma provides  the upper bounds for $\e|\sum_{i\in [n]}\sum_{j\in  A_{i}} (X_{i}X_{j}-\e X_{i}X_{j})|^{2}$ and $\e\big\{\xi_{A}^{p}\big|\sum_{i\in N_{A}^{c}}\sum_{j\in A_{i}\cap N_{A}^{c}} (X_{i}X_{j}-\e X_{i}X_{j}) \big|^{2}\big\}$ for any $p\geq 0$, which will play important roles in the proofs of Theorems \ref{thm-main1-1}--\ref{thm-main2}.
\begin{lem}\label{lem-XiYi-moment} Under \textup{(LD1)} and \textup{(LD2)}, for any $p\geq 0$, we have
	\begin{align}\label{eq-t-02}
		\e\Big\{\xi_{A}^{p}\Big|\sum_{i\in N_{A}^{c}}\sum_{j\in A_{i}\cap N_{A}^{c}} (X_{i}X_{j}-\e X_{i}X_{j}) \Big|^{2}\Big\}\leq 4\|\xi_{A}\|_{p}^{p}\cdot (\gamma_{A}^{2}+4\gamma),
	\end{align} 
    	where 
	\begin{align*}
		\gamma_{A}&=\sum_{(i,j)\in  D_{A}} \|X_{i}\|_{4}\|X_{j}\|_{4},\quad
		\gamma=\sum_{i \in [n]}\sum_{j\in A_{i}}\sum_{k\in A_{ij}}\sum_{l\in N_{k}\cup A_{k}}\|X_i\|_{4}\|X_j\|_{4}\|X_{k}\|_{4}\|X_{l}\|_{4}.
	\end{align*}
	As a consequence,
	\begin{align}\label{eq-t-01}
		\e\Big|\sum_{i\in [n]}\sum_{j\in A_{i}} (X_{i}X_{j}-\e X_{i}X_{j}) \Big|^{2}\leq 16\gamma.
	\end{align}
\end{lem}
\begin{proof} 
	Observe that
	\begin{align}\label{eq-t-09}
		\e\Big\{ \xi_{A}^{p}\Big|\sum_{i\in N_{A}^{c}}\sum_{j\in A_{i}\cap N_{A}^{c}} (X_{i}X_{j}-\e X_{i}X_{j}) \Big|^{2}\Big\}= H_{1}+H_{2}+H_{3},
	\end{align}
	where
	\begin{align*}
		H_{1}&= \sum_{i \in N_{A}^{c}}\sum_{j\in A_{i}\cap N_{A}^{c}}\sum_{i'\in N_{A}^{c}\cap A_{ij}^{c}}\sum_{j'\in N_{A}^{c}\cap A_{i'}\cap A_{ij}^{c}}\e\big\{ \xi_{A}^{p}(X_i X_j-\e X_{i}X_{j})(X_{i'} X_{j'}-\e X_{i'} X_{j'}) \big\},\\
		H_{2}&= \sum_{i \in N_{A}^{c}}\sum_{j\in A_{i}\cap N_{A}^{c}}\sum_{i'\in N_{A}^{c}\cap A_{ij}^{c}}\sum_{j'\in N_{A}^{c}\cap A_{i'}\cap A_{ij}}\e\big\{ \xi_{A}^{p}(X_i X_j-\e X_{i}X_{j})(X_{i'} X_{j'}-\e X_{i'} X_{j'}) \big\},\\
		H_{3}&= \sum_{i \in N_{A}^{c}}\sum_{j\in A_{i}\cap N_{A}^{c}}\sum_{i'\in N_{A}^{c}\cap A_{ij}}\sum_{j'\in N_{A}^{c}\cap A_{i'}}\e\big\{ \xi_{A}^{p}(X_i X_j-\e X_{i}X_{j})(X_{i'} X_{j'}-\e X_{i'} X_{j'}) \big\}.
	\end{align*}
	For $H_{1}$, for the sake of convenience, let 
	\begin{align}\label{eq-definition-of-Iijk}
		I_{i,j,k}=\mathbf{1}(i\in A_{jk}^{c}).
	\end{align}
 Note that by the definition of $N_{A}$, we have 
	\[
	i\in N_{A}^{c} \Longleftrightarrow  A\cap A_{i}=\varnothing,
	\]
	then $\{(i,j);  i\in N_{A}^{c}, j\in A_{i}\cap N_{A}^{c}\}=D_{A,1}$. Using the fact that $D_{A,1}=D_{A,2}\cup D_{A,3}$, we have 
	\begin{align}\label{eq-t-11}
    \begin{aligned}
        H_{1}&= \sum_{(i,j)\in D_{A,1}}\sum_{(i',j')\in D_{A,1}}\e\big\{ \xi_{A}^{p}(X_i X_j-\e X_{i}X_{j})(X_{i'} X_{j'}-\e X_{i'} X_{j'}) \big\}I_{i',i,j}I_{j',i,j}\\
		&:=H_{11}+H_{12},
    \end{aligned}	
	\end{align}
    where 
    \begin{align*}
        H_{11}&=\sum_{(i,j)\in D_{A,2}}\sum_{(i',j')\in D_{A,1}}\e\big\{ \xi_{A}^{p}(X_i X_j-\e X_{i}X_{j})(X_{i'} X_{j'}-\e X_{i'} X_{j'}) \big\}I_{i',i,j}I_{j',i,j},\nonumber\\
        H_{12}&=\sum_{(i,j)\in D_{A,3}}\sum_{(i',j')\in D_{A,1}}\e\big\{ \xi_{A}^{p}(X_i X_j-\e X_{i}X_{j})(X_{i'} X_{j'}-\e X_{i'} X_{j'})\big\}I_{i',i,j}I_{j',i,j}.
    \end{align*}
	For $H_{11}$, note that if $i'\in A_{ij}^{c}, j'\in A_{ij}^{c}, A\in A_{ij}^{c}$, then $(X_{i}, X_{j})$ is independent of $(\xi_{A},X_{i'}, X_{j'})$, which further implies
	\begin{align}\label{eq-t-12}
		H_{11}=0.
	\end{align}
	For $H_{12}$, we rewrite $H_{12}$ as follows:
	\begin{align}\label{eq-t-13}
    \begin{aligned}
		H_{12}&=H_{121}+H_{122}+H_{123}+H_{124},
        \end{aligned}
	\end{align}
    where $M(i,j,i',j')=(X_i X_j-\e X_{i}X_{j})(X_{i'} X_{j'}-\e X_{i'} X_{j'})$ and
    \begin{align*}
        H_{121}&=\sum_{(i,j)\in  D_{A,3}}\sum_{(i',j')\in D_{A,2}}\e\big\{ \xi_{A}^{p}M(i,j,i',j')\big\}I_{i',i,j}I_{j',i,j}\cdot I_{i,i',j'}I_{j,i',j'},\\
        H_{122}&=\sum_{(i,j)\in  D_{A,3}}\sum_{(i',j')\in D_{A,3}}\e\big\{ \xi_{A}^{p}M(i,j,i',j')\big\}
    I_{i',i,j}I_{j',i,j}\cdot I_{i,i',j'}I_{j,i',j'},\\
        H_{123}&=\sum_{(i,j)\in  D_{A,3}}\sum_{(i',j')\in D_{A,1}}\e\big\{ \xi_{A}^{p}M(i,j,i',j')\big\}I_{i',i,j}I_{j',i,j}\cdot I_{i,i',j'}(1-I_{j,i',j'}),\\
        H_{124}&=\sum_{(i,j)\in  D_{A,3}}\sum_{(i',j')\in D_{A,1}}\e\big\{ \xi_{A}^{p}M(i,j,i',j')\big\}I_{i',i,j}I_{j',i,j}\cdot (1-I_{i,i',j'}).
    \end{align*}
	For $H_{121}$, note that if $i\in A_{i'j'}^{c},j\in A_{i'j'}^{c}, A\in A_{i'j'}^{c}$, then  $(X_{i'}, X_{j'})$ is independent of $(X_{i}, X_{j}, \xi_{A})$, which further implies 
	\begin{align}\label{eq-t-13.5}
		H_{121}=0.
	\end{align}
	For $H_{122}$, as $(i,j)\in D_{A,3}$ and $(i',j')\in D_{A,3}$, it follows that $A\in (A_{i}\cup A_{j}\cup A_{i'}\cup A_{j'})^{c}$. By H\"{o}lder's inequality  and Young's inequality, we have
	\begin{align}\label{eq-t-14}
    \begin{aligned}
        		|H_{122}|&\leq 4\sum_{(i,j)\in  D_{A,3}}\sum_{(i',j')\in D_{A,3} }\|\xi_{A}\|_{p}^{p}\cdot \|X_{i}\|_{4}\|X_{j}\|_{4}\|X_{i'}\|_{4}\|X_{j'}\|_{4} \\
		&\leq 4\|\xi_{A}\|_{p}^{p}\cdot\gamma_{A}^{2}.
    \end{aligned}
	\end{align}
	As for $H_{123}$, by H\"{o}lder's inequality and Young's inequality again, we have 
	\begin{align}\label{eq-t-14.1}
    \begin{aligned}
		|H_{123}|&\leq 4\sum_{i \in [n]}\sum_{j\in A_{i}}\sum_{i'\in [n]}\sum_{j'\in A_{i'}} \|\xi_{A}\|_{p}^{p}\cdot \|X_{i}\|_{4}\|X_{j}\|_{4}\|X_{i'}\|_{4}\|X_{j'}\|_{4}\mathbf{1}(j\in A_{i'j'})\\
		&\leq 4 \|\xi_{A}\|_{p}^{p}\cdot\sum_{i \in [n]}\sum_{j\in A_{i}}\sum_{k\in A_{ij}}\sum_{l\in N_{k}}\|X_i\|_{4}\|X_j\|_{4}\|X_{k}\|_{4}\|X_{l}\|_{4}.
        \end{aligned}
	\end{align}
	Similarly, we have 
	\begin{align}\label{eq-t-14.2}
		|H_{124}|\leq 4 \|\xi_{A}\|_{p}^{p}\cdot\sum_{i \in [n]}\sum_{j\in A_{i}}\sum_{k\in A_{ij}}\sum_{l\in A_{k}}\|X_i\|_{4}\|X_j\|_{4}\|X_{k}\|_{4}\|X_{l}\|_{4}.
	\end{align}
	Combining (\ref{eq-t-11})--(\ref{eq-t-14.2}) yields that
	\begin{align}\label{eq-t-19}
		|H_{1}|\leq 4\|\xi_{A}\|_{p}^{p}\cdot(\gamma_{A}^{2}+2\gamma).
	\end{align}
	For $H_{2}$ and $H_{3}$, with similar argument as that leading to (\ref{eq-t-14.1}), we have
	\begin{align}\label{eq-t-20}
		|H_{2}|\leq  4\gamma \quad\text{and}\quad |H_{3}|\leq  4\gamma. 
	\end{align}
	Then (\ref{eq-t-02}) follows from (\ref{eq-t-09}), (\ref{eq-t-19}) and (\ref{eq-t-20}).
\end{proof}
The following lemma  provides an  upper bound for  $\e \{\xi_{A}^{p}S_{A}^{2}\}$.
\noindent
\begin{lem}\label{lem-second-moment-for-S}
	Assume that  \textup{(LD1)} and \textup{(LD2)} hold, then for any $p\geq 0$, we have
	\begin{align}\label{eq-i-01}
		\e\{\xi_{A}^{p}S_{A}^{2}\}\leq \|\xi_{A}\|_{p}^{p}\cdot \Big(\e \{S_{A}^{2}\}+2\sum_{(i,j)\in D_{A}}\|X_{i}\|_{4}\|X_{j}\|_{4}\Big).
	\end{align}
\end{lem}
\begin{proof}
	Observe that
	\begin{align*}
        	\e\{\xi_{A}^{p}S_{A}^{2}\}&=\sum_{i\in N_{A}^{c}}\sum_{j\in N_{A}^{c}}\e\{\xi_{A}^{p}X_{i}X_{j}\}=
		\sum_{i\in N_{A}^{c}}\sum_{j\in N_{A}^{c}\cap A_{i}}\e\{\xi_{A}^{p}X_{i}X_{j}\}\\
		&:=R_{1}+R_{2},
	\end{align*}
    where 
    \[
R_{1}=\sum_{(i,j)\in D_{A,2}}\e\{\xi_{A}^{p}X_{i}X_{j}\},\quad R_{2}=\sum_{(i,j)\in D_{A,3}}\e\{\xi_{A}^{p}X_{i}X_{j}\}.
    \]
	For $R_{1}$, since $(i,j)\in D_{A,2}$, then $(X_{i}, X_{j})$ is independent of $\xi_{A}$, then 
	\begin{align*}
        		R_{1}&= \sum_{(i,j)\in D_{A,2}}\e\{\xi_{A}^{p}\}\e \{X_{i}X_{j}\}\\
		&=\e\{\xi_{A}^{p}\}\cdot\sum_{i\in N_{A}^{c}}\sum_{j\in N_{A}^{c}\cap A_{i}}\e \{X_{i}X_{j}\}-\e\{\xi_{A}^{p}\}\cdot\sum_{(i,j)\in D_{A,3}}\e\{X_{i}X_{j}\}\\
		&\leq \e\{\xi_{A}^{p}\}\cdot \Big(\e \{S_{A}^{2}\}+\sum_{(i,j)\in D_{A}}\|X_{i}\|_{4}\|X_{j}\|_{4}\Big).
	\end{align*}
	For $R_{2}$, note that $A\in A_{i}^{c}$ and $A\in A_{j}^{c}$, then $\xi_{A}\perp\!\!\!\perp X_{i}$ and $\xi_{A}\perp\!\!\!\perp X_{j}$, which further implies 
	\begin{align*}
        		|R_{2}|&\leq \e\{\xi_{A}^{p}\}\sum_{(i,j)\in D_{A,3}}\|X_{i}\|_{2}\|X_{j}\|_{2}
		\leq \|\xi_{A}\|_{p}^{p}\sum_{(i,j)\in D_{A}}\|X_{i}\|_{4}\|X_{j}\|_{4}.
	\end{align*}
	
\end{proof}
\noindent
 Let
\begin{align*}
  \bar{V}=\psi\Big(\sum_{i\in [n]}X_{i}Y_{i}\Big)\quad \text{and}\quad \bar{W}_{2}=\sum_{i\in [n]}X_{i}/\bar{V},
\end{align*}
where $\psi(x)=\big((x\vee (\sigma^{2}/4)) \wedge (2\sigma^{2})\big)^{1/2}.$ The following lemma is useful in the proof of Theorem \ref{thm-main2}.
\begin{lem}\label{lem-R4}
Assume that \textup{(LD1)} holds  and  \begin{align}\label{eq-t1-00}
     \kappa^{2}\sum_{i\in [n]}  \e|X_{i}|^{3}/\sigma^{3}\leq  1/500,
\end{align}
 then for any absolutely continuous function $f$ satisfying  $\|f\|_{\infty} \leq 1$ and $\left\|f^{\prime}\right\|_{\infty} \leq 1$, we have
\begin{align}\label{eq-t1-01}
 \sum_{i \in [n]}\Big|\mathbb{E}\Big\{\frac{X_{i}}{\bar{V}} f\Big(\bar{W}_{2}-\frac{Y_{i}}{\bar{V}} \Big)\Big\}\Big| \leq \frac{27\kappa^{2}}{\sigma^{3}}   \sum_{i\in [n]}\e|X_{i}|^{3}+\frac{11\kappa^{3}}{\sigma^{4}}   \sum_{i\in [n]}\e|X_{i}|^{4},
\end{align}
where $Y_{i}=\sum_{j\in A_{i}}X_{i}.$
\end{lem}
\begin{proof}
Let
$$
\bar{V}^{(i)}=\psi\Big(\sum_{k\in A_{i}^{c}}\sum_{l\in A_{i}^{c}\cap A_{k}}X_{k}X_{l}\Big),
$$
then $X_{i}$ is independent of $(\bar{V}^{(i)}, S-Y_{i})$, which further implies
\begin{align}\label{eq-t1-02}
\sum_{i \in [n]}\Big|\mathbb{E}\Big\{\frac{X_{i}}{\bar{V}} f\Big(\bar{W}_{2}-\frac{Y_{i}}{\bar{V}} \Big)\Big\}\Big|= & \sum_{i \in [n]}\Big|\mathbb{E}\Big\{\frac{X_{i}}{\bar{V}} f\Big(\frac{S -Y_{i}}{\bar{V}}\Big)\Big\}-\mathbb{E}\Big\{\frac{X_{i}}{\bar{V}^{(i)}} f\Big(\frac{S -Y_{i}}{\bar{V}^{(i)}}\Big)\Big\}\Big|.
\end{align}
Note that $\bar{V} \geqslant \sigma / 2$ and $\bar{V}^{(i)} \geqslant \sigma / 2$. Then, it follows that
\begin{align}\label{eq-t1-03}
\begin{aligned}
    \Big|\frac{1}{\bar{V}}-\frac{1}{\bar{V}^{(i)}}\Big| & \leq \frac{1}{\bar{V} \bar{V}^{(i)}\left(\bar{V}+\bar{V}^{(i)}\right)}\Big|\sum_{k \in A_i} \sum_{l \in A_k} X_{k} X_{l}+\sum_{k \in A_i^c} \sum_{l \in A_i \cap A_k} X_{k} X_{l}\Big| \\
& \leq \frac{4}{\sigma^3}\Big(\sum_{k \in A_i} \sum_{l \in A_k} |X_{k} X_{l}|+\sum_{k \in A_i^c} \sum_{l \in A_i \cap A_k} |X_{k} X_{l}|\Big).
\end{aligned}
\end{align}
By (\ref{eq-t1-02}), (\ref{eq-t1-03}) and the fact that $\|f\|_{\infty}\leq 1$, $\|f'\|_{\infty}\leq 1$, then
\begin{align}\label{eq-t1-05}
\begin{aligned}
    &\sum_{i \in [n]}\Big|\mathbb{E}\Big\{\frac{X_{i}}{\bar{V}} f\Big(\bar{W}_{2}-\frac{Y_{i}}{\bar{V}} \Big)\Big\}\Big|\\
    &\quad\leq  \sum_{i \in [n]} \mathbb{E}\Big\{\Big(1+\frac{2\left|S -Y_{i}\right|}{\sigma}\Big)\cdot|X_{i}|\cdot\Big|\frac{1}{\bar{V}}-\frac{1}{\bar{V}^{(i)}}\Big|\Big\} \\
&\quad\leq  \frac{4}{\sigma^3} \sum_{i \in [n]} \sum_{j \in A_i} \sum_{k \in A_j} \mathbb{E}\Big\{\Big(1+\frac{2\left|S -Y_{i}\right|}{\sigma}\Big)\left|X_{i} X_{j} X_{k}\right|\Big\} \\
&\quad\quad +\frac{4}{\sigma^3} \sum_{i \in [n]} \sum_{j \in A_i} \sum_{k\in N_{j}} \mathbb{E}\Big\{\Big(1+\frac{2\left|S -Y_{i}\right|}{\sigma}\Big)\left|X_{i} X_{j} X_{k}\right|\Big\}\\
&\quad\leq 8\kappa^{2}\sum_{i\in[n]}\e|X_{i}|^{3}/\sigma^{3}+ I_{1}+I_{2}+I_{3}+I_{4},
\end{aligned}
\end{align}
where 
\begin{align*}
   I_{1}&=\frac{16 \kappa^2}{3 \sigma^4} \sum_{i \in [n]} \mathbb{E}\big\{|S -Y_{i}|\cdot\left|X_{i}^3\right|\big\},
&&I_{2}=\frac{16\kappa}{3 \sigma^4} \sum_{i \in [n]} \sum_{j\in A_{i}} \mathbb{E}\big\{|S -Y_{i}|\cdot\left|X_{j}^3\right|\big\},\nonumber \\
I_{3}&= \frac{8\kappa}{3\sigma^4} \sum_{i \in [n]} \sum_{j\in A_{i}}\sum_{k\in A_{j}}  \mathbb{E}\big\{|S -Y_{i}|\cdot\left|X_{k}^3\right|\big\},&&I_{4}=\frac{8\kappa}{3\sigma^4} \sum_{i \in [n]} \sum_{j\in A_{i}}\sum_{k\in N_{j}}  \mathbb{E}\big\{|S -Y_{i}|\cdot\left|X_{k}^3\right|\big\}.
\end{align*}
For $I_{1}$,  recall that $|A_{i}|\leq \kappa$, then by  Minkowski's inequality, H\"{o}lder's inequality and (\ref{eq-t1-00}), we have
\begin{align}\label{eq-t1-04}
  \|S-Y_{i}\|_{2}\leq \|S\|_{2}+\|Y_{i}\|_{2}\leq \sigma+ \big( \kappa^{2}\sum_{i\in [n]} \e|X_{i}|^{3} \big)^{1/3}\leq 1.126\sigma.
\end{align} By (\ref{eq-t1-04}), we have 
\begin{align}\label{eq-t1-5.1}
    I_{1}&\leq 6.01\kappa^{2}   \sum_{i\in [n]}\e|X_{i}|^{3}/\sigma^{3}.
\end{align}
For $I_{2}$, by H\"{o}lder's inequality, we have
\begin{align*}
\Big\|\sum_{A_{i}\cup A_{j}}X_{k}\Big\|_{2}\leq \Big(\e \Big|\sum_{A_{i}\cup A_{j}}X_{k}\Big|^{3}\Big)^{1/3}\leq \Big(4\kappa^{2}\sum_{k\in[n]}\e|X_{k}|^{3}\Big)^{1/3}\leq 0.2\sigma,
\end{align*}
which further implies
\begin{align}\label{eq-t1-07}
\begin{aligned}
    	\e\{ |S-Y_{i}|\cdot |X_{j}|^{3}\}&=\e\Big\{\Big|\sum_{k\in [n]\backslash (A_{i}\cup A_{j})}X_{k}\Big|\cdot |X_{j}|^{3}\Big\}+\sum_{k\in A_{j}\backslash A_{i}}\e|X_{k}X_{j}^{3}|\\
	&\leq 1.2\sigma\e|X_{j}|^{3}+\sum_{k\in A_{j}}\e|X_{k}X_{j}^{3}|.
\end{aligned}
\end{align}
It follows from (\ref{eq-t1-07}) that 
\begin{align}\label{eq-t1-08}
I_{2}&\leq  6.4\kappa^{2}   \sum_{i\in [n]}\e|X_{i}|^{3}/\sigma^{3}+5.34\kappa^{3}   \sum_{i\in [n]}\e|X_{i}|^{4}/\sigma^{4}.
\end{align}
Similarly, by (\ref{eq-t1-07}) again, we have
\begin{align}\label{eq-t1-09}
\max\big\{I_{3},I_{4}\big\}&\leq  3.2\kappa^{2}   \sum_{i\in [n]}\e|X_{i}|^{3}/\sigma^{3}+2.67\kappa^{3}  \sum_{i\in [n]}\e|X_{i}|^{4}/\sigma^{4}.
\end{align}
Combine (\ref{eq-t1-05}), (\ref{eq-t1-5.1}), (\ref{eq-t1-08}) and (\ref{eq-t1-09}), we complete the proof  of (\ref{eq-t1-01}).
\end{proof}
The following lemma  provides the   upper bounds for  $\e \{\xi_{A}^{p}S_{A}^{4}\}$, $\e\{S^{4}\}$ and $\e \big(\sum_{i\in [n]} Y_{i}\big)^{4}.$
\begin{lem}\label{lem-fourth-moment-for-S}
Assume that  \textup{(LD1)} and \textup{(LD2)} hold and 
	\begin{align}\label{eq-t2-01}
		\frac{|A|^{2}\kappa^{2}}{\sigma^{3}}\sum_{i\in [n]}  \|X_{i}\|_{4}^{3}\leq  \frac{1}{500}\quad\text{and}\quad \frac{|A|^{1/2}\kappa^{1/2}(\kappa+\tau^{1/2})}{\sigma^{2}}\Big(\sum_{i\in [n]} \|X_{i}\|_{4}^{4}\Big)^{1/2}\leq \frac{1}{500}.
	\end{align}
	Then for any $p\geq 0$, we have
		\begin{align}\label{eq-t2-03}
		\e \{\xi_{A}^{p}S_{A}^{4}\}\leq 13\lambda\sigma^{4}\e\{\xi_{A}^{p}\}
	\end{align}
	and
			\begin{align}\label{eq-t2-02}
		 	\e\{S^{4}\}\leq 13\lambda\sigma^{4},\qquad \e \Big(\sum_{i\in [n]} Y_{i}\Big)^{4}\leq 13\kappa^{4}\lambda\sigma^{4},
		\end{align}
where 
\[
S=\sum_{i\in [n]} X_{i}\quad\textup{and} \quad\lambda=\kappa\sum_{i\in [n]}\|X_{i}\|_{2}^{2}/\sigma^{2}.
\]
\end{lem}
\begin{proof} We first prove (\ref{eq-t2-03}). Let $Y_{i,A}=\sum_{l \in A_i \cap N_{A}^{c}} X_{l}$, observe that
	\begin{align}\label{eq-t2-11}
		\e\{\xi_{A}^{p} S_{A}^{4}\}&=\sum_{i\in N_{A}^{c}}\e \big\{\xi_{A}^{p}X_{i}\big(S_{A}^{3}-(S_{A}-Y_{i, A})^{3}\big)\}=H_{1}-H_{2}+H_{3},
	\end{align}
    where 
    \begin{align*}
        H_{1}&=3\sum_{i\in N_{A}^{c}}\e \big\{\xi_{A}^{p}X_{i}Y_{i, A}S_{A}^{2}\big\},\ H_{2}=3\sum_{i\in N_{A}^{c}}\e \big\{ \xi_{A}^{p}X_{i}Y_{i, A}^{2}S_{A}\big\},\ H_{3}=\sum_{i\in N_{A}^{c}}\e \big\{\xi_{A}^{p}X_{i}Y_{i, A}^{3}\big\}.
    \end{align*}
 For $H_{1}$, let $Y_{ij}=\sum_{k\in A_{ij}\cap N_{A}^{c}}X_{k}$. Using the basic inequality $(a+b)^{2}\leq  2a^{2}+2b^{2}$, we obtain 
 	\begin{align*}
 	|H_{1}|&=3\Big|\sum_{i\in N_{A}^{c}}\sum_{j\in N_{A}^{c}\cap A_{i}} \e\big\{\xi_{A}^{p}X_{i}X_{j}S_{A}^{2}\big\}\Big|\leq H_{11}+H_{12},
 \end{align*}
 where 
 \[
 	H_{11}=6\sum_{i\in N_{A}^{c}}\sum_{j\in N_{A}^{c}\cap A_{i}} \e |\xi_{A}^{p}X_{i}X_{j}Y_{ij}^{2}|,\quad H_{12}=6\sum_{i\in N_{A}^{c}}\sum_{j\in N_{A}^{c}\cap A_{i}} \e |\xi_{A}^{p}X_{i}X_{j}(S_{A}-Y_{ij})^{2}|.
 \]
  For $H_{11}$, by Young's inequality, we have
 \begin{align}\label{eq-t2-15}
 \begin{aligned}
      	|H_{11}|&\leq 6\kappa\sum_{i\in N_{A}^{c}}\sum_{j\in N_{A}^{c}\cap A_{i}}\sum_{k\in N_{A}^{c}\cap A_{ij}} \e |\xi_{A}^{p}X_{i}X_{j}X_{k}^{2}|\\
 	&\leq 6\kappa\sum_{i\in N_{A}^{c}}\sum_{j\in N_{A}^{c}\cap A_{i}}\sum_{k\in N_{A}^{c}\cap A_{ij}}  \Big(\frac{\e\{\xi_{A}^{p}X_{i}^{4}\}}{4}+\frac{\e\{ \xi_{A}^{p}X_{j}^{4}\}}{4}+\frac{\e\{ \xi_{A}^{p}X_{k}^{4}\}}{2}\Big)\\
 	&\leq 3\e\xi_{A}^{p}\cdot(\kappa^{3}+\kappa\tau)\sum_{i\in [n]}\e|X_{i}|^{4}\\
 	&\leq 0.1\sigma^{4}\e\{\xi_{A}^{p}\}.
 \end{aligned}
 \end{align}
 For $H_{12}$, noting that  $D_{A,1}=D_{A,2}\cup D_{A,3}$, so we can further decompose 
$H_{12}$ as follows:
 \begin{align*}
  H_{12}&=6H_{121}+6H_{122},
 \end{align*}
 where 
 \begin{align*}
     H_{121}=\sum_{(i,j)\in D_{A,2}} \e \big|\xi_{A}^{p}X_{i}X_{j}(S_{A}-Y_{ij})^{2}\big|,\qquad H_{122}=\sum_{(i,j)\in D_{A,3}} \e \big|\xi_{A}^{p}X_{i}X_{j}(S_{A}-Y_{ij})^{2}\big|.
 \end{align*}
 For $H_{121}$, with a similar argument as in the proof of  (\ref{eq-i-01}), we have
 \begin{align}\label{eq-tt-09}
 \begin{aligned}
      	&\e\big|\xi_{A}^{p}(S_{A}-Y_{ij})^{2}\big|\\
 	&\quad=\sum_{(k,l)\in D_{A,1}}\e\big\{ \xi_{A}^{p}X_{k}X_{l}\mathbf{1}(k\in A_{ij}^{c})\mathbf{1}(l\in A_{ij}^{c})\big\}\\
 	&\quad=\sum_{(k,l)\in D_{A,2}}\e\big\{ \xi_{A}^{p}X_{k}X_{l}\mathbf{1}(k\in A_{ij}^{c})\mathbf{1}(l\in A_{ij}^{c})\big\}\\
 	&\quad\quad+\sum_{(k,l)\in D_{A,3}}\e\big\{ \xi_{A}^{p}X_{k}X_{l}\mathbf{1}(k\in A_{ij}^{c})\mathbf{1}(l\in A_{ij}^{c})\big\}\\
 	&\quad\leq \e\{\xi_{A}^{p}\}\cdot \Big(\e\{S_{A}^{2}\}+2 \sum_{k\in A_{ij}}\sum_{l\in A_{k}\cup N_{k}}\e|X_{k}X_{l}|+\sum_{(k,l)\in D_{A}} \|X_{k}\|_{4}\|X_{l}\|_{4}\Big).
 \end{aligned}
 \end{align}
 Noting that $|N_{A}|\leq A\kappa$ and $|D_{A}|\leq |A|\tau$, then by (\ref{eq-t2-01}), we have
  \begin{align}\label{eq-tt-9.1}
 	\|S-S_{A}\|_{2}&\leq \sum_{k\in N_{A}}\|X_{k}\|_{2}\leq \Big(|A|^{2}\kappa^{2}\sum_{i\in [n]} \e |X_{i}|^{3}\Big)^{1/3}\leq 0.126\sigma
 \end{align}
and
 \begin{align}\label{eq-tt-9.2}
 \begin{aligned}
     	\sum_{(k,l)\in D_{A}} \|X_{k}\|_{4}\|X_{l}\|_{4}&\leq \frac{1}{2}\Big(\sum_{(k,l)\in D_{A}} (\|X_{k}\|_{4}^{2}+\|X_{l}\|_{4}^{2})^{2}   \Big)^{1/2}\cdot \Big(\sum_{(k,l)\in D_{A}} 1  \Big)^{1/2}\\
	&\leq \frac{\sqrt{2}}{2}\Big(|A|\tau\sum_{k\in [n]}\sum_{l\in A_{k}} (\|X_{k}\|_{4}^{4}+\|X_{l}\|_{4}^{4}  )\Big)^{1/2}\\
	&\leq |A|^{1/2}\tau^{1/2}\kappa^{1/2}\Big(\sum_{k\in [n]} \|X_{k}\|_{4}^{4}  \Big)^{1/2}\leq 0.01\sigma^{2}.
 \end{aligned}
\end{align}
 By (\ref{eq-t2-01}), (\ref{eq-tt-09}), (\ref{eq-tt-9.1}) and (\ref{eq-tt-9.2}), we have
 \begin{align*}
       H_{121}&=\sum_{(i,j)\in D_{A,2}} \e|X_{i}X_{j}|\cdot\e \big|\xi_{A}^{p}(S_{A}-Y_{ij})^{2}\big|\\
  &\leq 1.5\lambda\sigma^{4}\e\{ \xi_{A}^{p}\}+2\e\xi_{A}^{p}\sum_{i\in [n]}\sum_{j\in A_{i}}\sum_{k\in A_{ij}}\sum_{l\in A_{k}\cup N_{k}}\e |X_{i}X_{j}|\e |X_{k}X_{l}|\\
  &\leq 1.6\lambda\sigma^{4}\e\{\xi_{A}^{p}\},
 \end{align*}
 where 
 $
 \lambda=\kappa\sum_{i\in [n]}\|X_{i}\|_{2}^{2}/\sigma^{2}
$
and the last inequality follows from 
\begin{align*}
   &2\sum_{i\in [n]}\sum_{j\in A_{i}}\sum_{k\in A_{ij}}\sum_{l\in A_{k}\cup N_{k}}\e |X_{i}X_{j}|\e |X_{k}X_{l}|\nonumber\\
   &\quad\leq \frac{1}{2} \sum_{i\in [n]}\sum_{j\in A_{i}}\sum_{k\in A_{ij}}\sum_{l\in A_{k}\cup N_{k}} (\|X_{i}\|_{4}^{4}+\|X_{j}\|_{4}^{4}+\|X_{k}\|_{4}^{4}+\|X_{l}\|_{4}^{4})\nonumber\\
   &\quad\leq 2(\kappa^{3}+\kappa\tau)\sum_{i\in [n]} \|X_{i}\|_{4}^{4}\leq 0.1\sigma^{4}.
\end{align*}
 For $H_{122}$, we rewrite $H_{122}$ as follows:
 \begin{align}\label{eq-tt-11}
 	H_{122}&=\sum_{(i,j)\in D_{A,3}}\sum_{k\in N_{A}^{c}\cap A_{ij}^{c}}\sum_{l\in N_{A}^{c}\cap A_{ij}^{c}} \e\big\{ |\xi_{A}^{p}X_{i}X_{j}|\cdot X_{k}X_{l}\big\}=R_{1}+R_{2},
 \end{align}
 where 
 \begin{align*}
     R_{1}&=\sum_{(i,j)\in D_{A,3}}\sum_{k\in N_{A}^{c}\cap A_{ij}^{c}}\sum_{l\in N_{A}^{c}\cap A_{ij}^{c}\cap A_{k}^{c}} \e\big\{ |\xi_{A}^{p}X_{i}X_{j}|\cdot X_{k}X_{l}\big\},\nonumber\\
     R_{2}&=\sum_{(i,j)\in D_{A,3}}\sum_{k\in N_{A}^{c}\cap A_{ij}^{c}}\sum_{l\in N_{A}^{c}\cap A_{ij}^{c}\cap A_{k}} \e\big\{ |\xi_{A}^{p}X_{i}X_{j}|\cdot X_{k}X_{l}\big\}.
 \end{align*}
 For $R_{1}$, if $i\in A_{k}^{c}$ and $j\in A_{k}^{c}$, then $X_{k}$ is independent of $(\xi_{A}, X_{i}, X_{j}, X_{l})$, which further implies 
 \begin{align*}
 	R_{1}=0.
 \end{align*}
 Based on the discussion above, we have
 \begin{align*}
 	R_{1}&=\sum_{(i,j)\in D_{A,3}}\sum_{k\in N_{A}^{c}\cap A_{ij}^{c}}\sum_{l\in N_{A}^{c}\cap A_{ij}^{c}\cap A_{k}^{c}} \e\big\{ |\xi_{A}^{p}X_{i}X_{j}|\cdot X_{k}X_{l}\cdot \mathbf{1}(i\in A_{k}\ \text{or}\ j\in A_{k})\big\}\\
        &=R_{11}+R_{12}.
 \end{align*}
 where 
 \begin{align*}
     R_{11}&=\sum_{(i,j)\in D_{A,3}}\sum_{l\in N_{A}^{c}\cap A_{ij}^{c}}\sum_{k\in N_{A}^{c}\cap A_{ij}^{c}\cap N_{l}^{c}\cap A_{l}^{c}} \e\big\{ |\xi_{A}^{p}X_{i}X_{j}|\cdot X_{k}X_{l}\big\}\cdot \mathbf{1}(i\in A_{k}\ \text{or}\ j\in A_{k}),\\
     R_{12}&=\sum_{(i,j)\in D_{A,3}}\sum_{l\in N_{A}^{c}\cap A_{ij}^{c}}\sum_{k\in N_{A}^{c}\cap A_{ij}^{c}\cap N_{l}^{c}\cap A_{l}} \e\big\{ |\xi_{A}^{p}X_{i}X_{j}|\cdot X_{k}X_{l}\big\}\cdot \mathbf{1}(i\in A_{k}\ \text{or}\ j\in A_{k}).
 \end{align*}
 For $R_{11}$, if $i\in A_{l}^{c}$ and $j\in A_{l}^{c}$, then $X_{l}$ is independent of $(\xi_{A}, X_{i},X_{j},X_{k})$, which implies that $R_{11}=0$. Hence 
 \begin{align}\label{eq-tt-11.3}
 \begin{aligned}
     R_{11}
     &=\sum_{(i,j)\in D_{A,3}}\sum_{l\in N_{A}^{c}\cap A_{ij}^{c}}\sum_{k\in N_{A}^{c}\cap A_{ij}^{c}\cap N_{l}^{c}\cap A_{l}^{c}} \e\big\{ |\xi_{A}^{p}X_{i}X_{j}|\cdot X_{k}X_{l}\big\}\\
     &\qquad\qquad\qquad\qquad\qquad\times \mathbf{1}(i\in A_{k}\ \text{or}\ j\in A_{k})\cdot \mathbf{1}(i\in A_{l}\ \text{or}\ j\in A_{l})\\
 	&\leq \|\xi_{A}\|_{p}^{p}\sum_{(i,j)\in D_{A,3}}\sum_{l\in [n]}\sum_{k\in [n]}  \|X_{i}\|_{4}\|X_{j}\|\|X_{k}\|_{4}\|X_{l}\|_{4}\\
     &\qquad\qquad\qquad\qquad\qquad\times \mathbf{1}(i\in A_{k}\ \text{or}\ j\in A_{k})\cdot \mathbf{1}(i\in A_{l}\ \text{or}\ j\in A_{l})\\
 	&\leq 4\|\xi_{A}\|_{p}^{p}\sum_{(i,j)\in D_{A}}\sum_{l\in N_{i}\cup N_{j}}\sum_{k\in N_{i}\cup N_{j}}  \|X_{i}\|_{4}\|X_{j}\|_{4}\|X_{k}\|_{4}\|X_{l}\|_{4}\\ 
 	&\leq 0.05\sigma^{4}\|\xi_{A}\|_{p}^{p},
 \end{aligned}
 \end{align}
 where the last inequality follows from (\ref{eq-tt-9.2}) and the fact that  
 \[
 \sum_{k\in N_{i}\cup N_{j}}\|X_{k}\|_{4}\leq \big(4\kappa^{2}\sum_{ k\in [n]}\|X_{k}\|_{4}^{3}\big)^{1/3}\leq 0.2\sigma.
 \] For $R_{12}$, as $l\in A_{k}^{c}$, it follows from (\ref{eq-tt-9.2}) that
 \begin{align}\label{eq-tt-11.4}
 	R_{12}&\leq \|\xi_{A}\|_{p}^{p}\sum_{(i,j)\in D_{A}}\sum_{l\in [n]}\sum_{k\in  A_{l}}\|X_{i}\|_{4}\|X_{j}\|_{4}\|X_{k}\|_{2}\|X_{l}\|_{2}
 	\leq 0.05\lambda\sigma^{4}\|\xi_{A}\|_{p}^{p}.
 \end{align}
Combine (\ref{eq-tt-11.3}) and (\ref{eq-tt-11.4}), we have
\begin{align}\label{eq-tt-11.5}
	R_{1}\leq 0.1\lambda\sigma^{4}\|\xi_{A}\|_{p}^{p}.
\end{align}
For $R_{2}$, with $I_{i,j,k}$ defined by (\ref{eq-definition-of-Iijk}), we have
\begin{align}\label{eq-tt-11.6}
\begin{aligned}
    	R_{2}&\leq \sum_{(i,j)\in D_{A,3}}\sum_{(k,l)\in D_{A,1}} \e\big\{ |\xi_{A}^{p}X_{i}X_{j}|\cdot |X_{k}X_{l}| \big\}I_{k,i,j}I_{l,i,j}\\
    &=R_{21}+R_{22}+R_{23}+R_{24},
\end{aligned}
\end{align}
where 
\begin{align*}
	R_{21}&=\sum_{(i,j)\in D_{A,3}}\sum_{(k,l)\in D_{A,2}} \e\big\{ |\xi_{A}^{p}X_{i}X_{j}|\cdot |X_{k}X_{l}|\big\} \cdot I_{i,k,l}I_{j,k,l}I_{k,i,j}I_{l,i,j},\nonumber\\
	R_{22}&=\sum_{(i,j)\in D_{A,3}}\sum_{(k,l)\in D_{A,3}} \e\big\{ |\xi_{A}^{p}X_{i}X_{j}|\cdot |X_{k}X_{l}|\big\}I_{k,i,j}I_{l,i,j}\cdot I_{i,k,l}I_{j,k,l},\nonumber\\
	R_{23}&=\sum_{(i,j)\in D_{A,3}}\sum_{(k,l)\in D_{A,1}} \e\big\{ |\xi_{A}^{p}X_{i}X_{j}|\cdot |X_{k}X_{l}| \big\}I_{k,i,j}I_{l,i,j}\cdot I_{i,k,l}(1-I_{j,k,l}),\nonumber\\
	R_{24}&=\sum_{(i,j)\in D_{A,3}}\sum_{(k,l)\in D_{A,1}} \e\big\{ |\xi_{A}^{p}X_{i}X_{j}|\cdot |X_{k}X_{l}|\big\} I_{k,i,j}I_{l,i,j}\cdot (1-I_{i,k,l}).
\end{align*}
For $R_{21}$, as $A\in A_{kl}^{c}$, $i\in A_{kl}^{c}$ and $j\in A_{kl}^{c}$, then $(X_{k},X_{l})\perp\!\!\!\perp (\xi_{A}, X_{i},X_{l})$, which further implies that
\begin{align}\label{eq-tt-11.61}
	|R_{21}|&\leq \lambda\sigma^{2}\|\xi_{A}\|_{p}^{p}\sum_{(i,j)\in D_{A}}\|X_{i}\|_{4}\|X_{j}\|_{4}\leq 0.05\lambda\sigma^{4}\|\xi_{A}\|_{p}^{p}.
\end{align}
For $R_{22}$, using the fact that $D_{A,3}\subset D_{A}$ and (\ref{eq-tt-9.2}) yields 
\begin{align}\label{eq-tt-11.7}
	|R_{21}|&\leq \|\xi_{A}\|_{p}^{p}\Big(\sum_{(i,j)\in D_{A}}\|X_{i}\|_{4}\|X_{j}\|_{4}\Big)^{2}\leq 0.05\sigma^{4}\|\xi_{A}\|_{p}^{p}.
\end{align}
For $R_{23}$ and $R_{24}$, by the definition of $D_{A}$, if $l\in A_{k}$ and $j\in A_{kl}$, then $(k,l)\in D_{j}$. So by (\ref{eq-tt-9.2}), we have
\begin{align}\label{eq-tt-11.8}
\begin{aligned}
    	&R_{23}+R_{24}\\
        &\quad\leq \|\xi_{A}\|_{p}^{p}\sum_{(i,j)\in D_{A}}\|X_{i}\|_{4}\|X_{j}\|_{4}\cdot \Big( \sum_{(k,l)\in D_{i}}\|X_{k}\|_{4}\|X_{l}\|_{4}+\sum_{(k,l)\in D_{j}}\|X_{k}\|_{4}\|X_{l}\|_{4}\Big)\\
    &\quad\leq 0.1\sigma^{4}.
\end{aligned}
\end{align}
  It follows from (\ref{eq-tt-11}) and (\ref{eq-tt-11.5})  together with (\ref{eq-tt-11.6})--(\ref{eq-tt-11.8}) that
   \begin{align}\label{eq-tt-12}
  	|H_{12}|\leq 12\lambda\sigma^{4}\e\{ \xi_{A}^{p}\}.
  \end{align}
  Combining (\ref{eq-t2-15}) and (\ref{eq-tt-12}) yields
  \begin{align}\label{eq-t2-16}
  	|H_{1}|\leq 12.1\lambda\sigma^{4}\e \{\xi_{A}^{p}\}.
  \end{align}
  As for $H_{2}$ and $H_{3}$, with similar argument as that leading to  (\ref{eq-t2-16}), we have
  \begin{align}\label{eq-t2-116}
  	H_{2}+H_{3}\leq 0.9\sigma^{4}\lambda\e \{\xi_{A}^{p}\}.
  	\end{align}
  Combine (\ref{eq-t2-11}), (\ref{eq-t2-16}) and (\ref{eq-t2-116}), we obtain  (\ref{eq-t2-03}). The first inequality of (\ref{eq-t2-02}) can be obtained  by a  similar argument as above.   	For the second inequality of (\ref{eq-t2-02}), note that $$\sum_{i\in [n]}Y_{i}=\sum_{i\in [n]}\sum_{j\in  A_{i}} X_{i}=\sum_{i\in [n]}|N_{i}| X_{i}$$
  and $|N_{i}|\leq \kappa.$  With similar argument as that in the proof of (\ref{eq-t2-03}), we have
  \begin{align*}
  	\e \Big(\sum_{i\in [n]} Y_{i}\Big)^{4}\leq 13\kappa^{4}\lambda^{2}\sigma^{4}.
  \end{align*}
  \end{proof}
\subsection{Generalized  concentration inequalities}
In this subsection, we develop two refined randomized concentration inequalities under (LD1) and (LD2), which plays an important role in the proof of  main results. Throughout this subsection, for any $A\subset [n]$, we denote 
\[
S_{A}=S_{n,A}:=\sum_{i\in N_{A}^{c}} X_{i}, \quad  Y_{k,A}=\sum_{l \in A_k \cap N_{A}^{c}} X_{l},
\]
\begin{prop}\label{prop:1} Let $A \subset[n]$ and $B \subset[n]$ be two arbitrary nonempty index sets. Let $\xi_A=h\left(X_A\right) \geqslant 0, \eta_{B}=a-c \sum_{m \in B}\left|X_m\right|/\sigma$ and $\zeta_{B}=b+c \sum_{m \in B}\left|X_m\right|/\sigma$, where $a, b, c \in \mathbb{R}$ satisfy $a \leq b$ and $c \geqslant 1$. Then,
	\begin{align*}
		\mathbb{E}\left\{\xi_A \mathbf{1}\left(\eta_B \leq S_{A}/\sigma \leq \zeta_B\right)\right\} \leq 156\left\|\xi_A\right\|_{4/3}\cdot\sum_{i=0}^{7}\delta_{i},
	\end{align*}
	where
	\begin{align*}
		\delta_{0}&=\frac{b-a}{100},\quad\delta_{1}=\frac{c}{\sigma}\sum_{i\in N_{A}}\|X_{i}\|_{4}, \quad
		\delta_2 = \frac{c}{\sigma}\sum_{m \in B}\left\|X_m\right\|_4,\\
		\delta_3 &=\frac{c}{\sigma^{2}} \sum_{m \in B} \sum_{k \in  N_m}\left\|X_k\right\|_4\left\|X_m\right\|_4,\quad 	\delta_{4}=\frac{c}{\sigma^{2}} \sum_{(i,j)\in  D_{A}} \|X_{i}\|_{4}\|X_{j}\|_{4},\\
		\delta_5  &=\frac{c}{\sigma^{3}}  \sum_{i\in [n]} |A_{i}|^2\left\|X_i\right\|_{4}^3+\frac{c}{\sigma^{3}}\sum_{i\in [n]}\sum_{j\in A_{i}} |A_{i}|\left\|X_j\right\|_{4}^3,\nonumber\\
		\delta_{6}^{2} &=\frac{c^{2}}{\sigma^{4}}\sum_{i \in [n]}\sum_{j\in A_{i}}\sum_{k\in A_{ij}}\sum_{l\in A_{k}\cup N_{k}}\|X_i\|_{4}\|X_j\|_{4}\|X_{k}\|_{4}\|X_{l}\|_{4}\nonumber\\
		&\quad+\frac{c^{2}}{\sigma^{4}} \sum_{i\in [n]}\sum_{j\in A_{i}\cup N_{i}} |A_{i}|^2\left\|X_i\right\|_{4}^3\|X_{j}\|_{4}+\frac{c^{2}}{\sigma^{4}}\sum_{i\in [n]}\sum_{j\in A_{i}}\sum_{k\in A_{i}\cup N_{j}} |A_{i}|\left\|X_j\right\|_{4}^3\|X_{k}\|_{4},\nonumber\\
		\delta_{7}^{2}&=\frac{c^{2}}{\sigma^{5}} \sum_{i\in [n]} |A_{i}|^2\sum_{j\in A_{i}\cup N_{i}}\sum_{k\in N_{j}}\left\|X_i\right\|_{4}^3\|X_{j}\|_{4}\|X_{k}\|_{4}\\
        &\quad+\frac{c^{2}}{\sigma^{5}}\sum_{i\in [n]}\sum_{j\in A_{i}}\sum_{k\in A_{i}\cup N_{j}}\sum_{l\in N_{k}} |A_{i}|\left\|X_j\right\|_{4}^3\|X_{k}\|_{4}\|X_{l}\|_{4}\nonumber\\
		&\quad+\frac{c^{2}}{\sigma^{5}} \sum_{i\in [n]}\sum_{(j,k)\in D_{i}} |A_{i}|^2\left\|X_i\right\|_{4}^3\|X_{j}\|_{4}\|X_{k}\|_{4}\\
        &\quad+\frac{c^{2}}{\sigma^{5}}\sum_{i\in [n]}\sum_{j\in A_{i}}\sum_{(k,l)\in D_{j}} |A_{i}|\|X_j\|_{4}^3\|X_{k}\|_{4}\|X_{l}\|_{4}.
	\end{align*}
\end{prop}
\begin{proof} 
We prove this proposition by a recursive argument, in the spirit of \cite{ChenandShao}, but we apply the recursion to a different object. To this end, let $K$ be defined as
	\begin{align*}
		K=\sup \frac{\e\{\xi_{A}\mathbf{1}(\eta_{B}\leq S_{A}/\sigma\leq \zeta_{B})\}}{\left\|\xi_A\right\|_{4/3}\sum_{i=0}^{7}\delta_{i}},
	\end{align*}
	where the supremum runs over $A,B,a,b,c$ as defined above. Without loss of generality, we assume that $\max_{0\leq i\leq 7}\delta_{i}<1/100$, otherwise the inequality is trivial.
  For each $k\in N_{A}^{c}$, let
    \begin{align}\label{eq-definition-of-M}
\begin{aligned}
    	\widehat{M}_k(t)&=X_k \{\mathbf{1}(- Y_{k,A} \leq t \leq
	0)-\mathbf{1}(0\leq t<-Y_{k,A})\},\\
	\widehat{M}(t)&=\sum_{k \in  N_A^c} \widehat{M}_k(t),\quad  M(t)=\e\big\{\widehat{M}(t)\big\}.
\end{aligned}
\end{align}
	To prove the concentration inequality, we first introduce a function. For any $\eta<\zeta$ and  $\alpha>0$, let
\begin{align}\label{eq-fabw}
	f_{\eta, \zeta,\alpha}(w)= \begin{cases}-\frac{\zeta-\eta+\alpha}{2}, & \text { for } w \leqslant \eta-\alpha, \\ \frac{1}{2 \alpha}(w-\eta+\alpha)^2-\frac{\zeta-\eta+\alpha}{2}, & \text { for } \eta-\alpha<w \leqslant \eta, \\ w-\frac{\eta+\zeta}{2}, & \text { for } \eta<w \leqslant \zeta, \\ -\frac{1}{2 \alpha}(w-\zeta-\alpha)^2+\frac{\zeta-\eta+\alpha}{2}, & \text { for } \zeta<w \leqslant \zeta+\alpha, \\ \frac{\zeta-\eta+\alpha}{2}, & \text { for } w>\zeta+\alpha.\end{cases}
\end{align}
Then, $f_{\eta, \zeta,\alpha}((\eta+\zeta) / 2)=0$ and the derivative $f_{\eta, \zeta,\alpha}^{\prime}$ is continuous and can be expressed as
$$
f_{\eta, \zeta,\alpha}^{\prime}(w)= \begin{cases}1, & \text { for } \eta \leqslant w \leqslant \zeta, \\ 0, & \text { for } w \leqslant \eta-\alpha \text { or } w \geqslant \zeta+\alpha, \\ \text {linear},\ & \text { for } \eta-\alpha \leqslant w \leqslant \eta \text { or } \zeta \leqslant w \leqslant \zeta+\alpha .\end{cases}
$$
    Let $f_{\eta_{B}, \zeta_{B},\alpha}(w)$ be defined by (\ref{eq-fabw}) with
	$$
	\alpha=100(\delta_5+\delta_6+\delta_{7}).
	$$
    For simplicity of notation, in what follows we omit the dependence on $\alpha$ and write
$f_{\eta_B, \zeta_B}(w) := f_{\eta_B, \zeta_B, \alpha}(w).$
	For any $k \in N_A^c$, define
	$$
	S_{A}^{(k)}=S_{A}-Y_{k,A}, \quad\eta_B^{(k)}=a-\frac{c}{\sigma}\sum_{m \in B \backslash A_k}|X_m|, \quad \zeta_B^{(k)}=b+\frac{c}{\sigma}\sum_{m \in B \backslash A_k}|X_m| .
	$$	Then, it follows that $X_k$ is independent of $(\xi_A, \eta_B^{(k)}, \zeta_B^{(k)}, S_{A}^{(k)})$ if $k\in N_{A}^{c}$, which further implies
	\begin{align}\label{eq-p-02}
        		&\frac{1}{\sigma}\mathbb{E}\big\{\xi_A S_{A} f_{\eta_B, \zeta_B}(S_{A}/\sigma)\big\}\nonumber\\ &\quad= \sum_{k \in  N_A^c} \frac{1}{\sigma}\mathbb{E}\big\{\xi_A X_k f_{\eta_B, \zeta_B}(S_{A}/\sigma)\big\} \nonumber\\
		 &\quad= \sum_{k \in  N_A^c} \frac{1}{\sigma}\mathbb{E}\big\{\xi_A X_k\big(f_{\eta_{B}, \zeta_B}(S_{A}/\sigma)-f_{\eta_B, \zeta_B}(S_{A}^{(k)}/\sigma)\big)\big\}\\
		&\qquad +\sum_{k \in  N_A^c} \frac{1}{\sigma}\mathbb{E}\big\{\xi_A X_k\big(f_{\eta_B, \zeta_B}(S_{A}^{(k)}/\sigma)-f_{\eta_{B}^{(k)}, \zeta_B^{(k)}}(S_{A}^{(k)}/\sigma)\big)\big\} \nonumber\\
		 &\quad= H_1+H_2+H_3+H_4,\nonumber
	\end{align}
	where
	\begin{align*}
		H_1&=\frac{1}{\sigma^{2}}\mathbb{E}\Big\{\int_{\mathbb{R}} \xi_A f_{\eta_B, \zeta_B}^{\prime}(S_{A}/\sigma) M(t) d t\Big\}, \\
		H_2&=\sum_{k \in  N_A^c}\frac{1}{\sigma^{2}}\mathbb{E}\Big\{\int_{\mathbb{R}} \xi_A[f_{\eta_B, \zeta_B}^{\prime}((S_{A}+t)/\sigma)-f_{\eta_B, \zeta_B}^{\prime}(S_{A}/\sigma)] \widehat{M} _{k}(t) d t\Big\}, \\
		H_3&=\frac{1}{\sigma^{2}}\mathbb{E}\Big\{\xi_A f_{\eta_B, \zeta_B}^{\prime}(S_{A}/\sigma) \int_{\mathbb{R}}[\widehat{M}(t)-M(t)] d t\Big\}, \\
		H_4&=\sum_{k \in  N_A^c} \frac{1}{\sigma}\mathbb{E}\big\{\xi_A X_k[f_{\eta_B, \zeta_B}(S_{A}^{(k)}/\sigma)-f_{\eta_B^{(k)}, \zeta_B^{(k)}}(S_{A}^{(k)}/\sigma)]\big\} .
	\end{align*}
	Next, we will provide the upper bounds of $|\mathbb{E}\{\xi_A S_{A} f_{\eta_B, \zeta_B}(S_{A}/\sigma)\}|/\sigma, H_2, H_3$ and $H_4$, and the lower bound of $H_1$. We remark that we  use the recursive argument when we prove the upper bound of $H_2$. \\ \hspace*{\fill} \\
	\noindent
	{\it (i) Upper bound of $|\mathbb{E}\{\xi_A S_{A} f_{\eta_B, \zeta_B}(S_{A}/\sigma)\}/\sigma|$.} Note that $\|f_{\eta_B, \zeta_B}\|_{\infty}\leq (\zeta_B-\eta_B+\alpha)/2\leq (b-a+\alpha)/2+\sum_{m\in B}|X_{m}|/\sigma$, then
	\begin{align}\label{eq-p-03}
    \begin{aligned}
        		&\frac{1}{\sigma}|\mathbb{E}\{\xi_A S_{A} f_{\eta_B, \zeta_B}(S_{A}/\sigma)\}|\\
		&\quad\leq 0.5(b-a+\alpha)\sigma^{-1}\e|\xi_{A} S_{A}|+c\sigma^{-2}\sum_{m\in B}\e|\xi_{A}S_{A} X_{m}|\\
		&\quad\leq 0.5(b-a+\alpha)\sigma^{-1}\e|\xi_{A} S_{A}|+c\sigma^{-2}\sum_{m\in B}\|\xi_{A}S_{A}\|_{4/3} \|X_{m}\|_{4}.
    \end{aligned}
	\end{align}
	Noting that 
	$
	\|S-S_{A}\|_{2}\leq \sum_{i\in N_{A}}\|X_{i}\|_{2}\leq  \sigma\cdot \delta_{1}\leq  0.01\sigma,
	$
	then
	\begin{align}\label{eq-p-04}
		0.99\sigma\leq \|S_{A}\|_{2}\leq 1.01\sigma.
	\end{align}
By (\ref{eq-p-04}), Lemma \ref{lem-second-moment-for-S}, and the  condition $\delta_{4}\leq 0.01$, we have for any $p\geq 0$, $\e\{\xi_{A}^{p}S_{A}^{2}\}\leq \e\{\xi_{A}^{p}\}\cdot(\|S_{A}\|_{2}^{2}+2\sigma^{2}\delta_{4})\leq 1.1\e\{\xi_{A}^{p}\}\sigma^{2}$, and then
	\begin{align}\label{eq-p-06}
		\|\xi_{A} S_{A}\|_{4/3}\leq \big(\e\{\xi_{A}^{4/3}\}\big)^{1/4}\big( \e\{\xi_{A}^{4/3}S_{A}^{2}\}  \big)^{1/2}\leq 1.1\sigma\|\xi_{A}\|_{4/3}.
	\end{align}
	Combining (\ref{eq-p-03}) and (\ref{eq-p-06}) yields
	\begin{align}\label{eq-p-07}
		\frac{1}{\sigma}\big|\mathbb{E}\{\xi_A S_{A} f_{\eta_B, \zeta_B}(S_{A}/\sigma)\}\big|\leq 0.55(b-a+\alpha)\|\xi_{A}\|_{4/3}+1.1\|\xi_{A}\|_{4/3}\delta_{2}.
	\end{align} 
	\noindent
	{\it (ii) Lower bound of $H_{1}$.} It follows from the definition of $M(t)$ that
	\begin{align}\label{eq-p-7.1}
		\int_{\mathbb{R}} M(t) d t & = \sum_{k\in N_{A}^{c}}\sum_{l\in A_{k}\cap N_{A}^{c}}\e \{X_{k}X_{l}\}=\e\{S_{A}^{2}\}.
	\end{align}
	Note that $f'_{\eta_{B},\zeta_{B}}(S_{A}/\sigma)\geq \mathbf{1}(\eta_{B}\leq S_{A}/\sigma\leq \zeta_{B})$, together with (\ref{eq-p-04})  yields 
	\begin{align}\label{eq-p-08}
		H_{1}\geq 0.98\e \{\xi_{A}\mathbf{1}(\eta_{B}\leq S_{A}/\sigma\leq \zeta_{B})\}.
	\end{align}
	{\it (iii) Upper bound of $|H_{3}|$.}  For $H_3$, using the fact $|f_{\xi_{B},\zeta_{B}}^{\prime}(S_{A}/\sigma)|\leq 1$ and  Lemma \ref{lem-XiYi-moment}, we have
	\begin{align}\label{eq-p-09}
        \left|H_3\right|
		& \leq \frac{1}{\sigma^{2}}\mathbb{E}\Big\{\xi_A \Big|\int_{\mathbb{R}}[\widehat{M}(t)-M(t)] d t\Big|\Big\} \nonumber\\
		& = \frac{1}{\sigma^{2}}\mathbb{E}\Big\{\xi_A\Big|\sum_{k \in  N_A^c} \sum_{l \in A_k \cap N_{A}^{c} }\left(X_k X_l-\mathbb{E}\{X_k X_l\}\right)\Big|\Big\} \nonumber\\
		&\leq \frac{1}{\sigma^{2}}\big(\e \{\xi_{A}^{p}\}\big)^{1/2}\cdot \Big(\mathbb{E}\Big\{\xi_A\Big|\sum_{k \in  N_A^c} \sum_{l \in A_k \cap N_{A}^{c} }\left(X_k X_l-\mathbb{E}\{X_k X_l\}\right)\Big|^{2}\Big\}\Big)^{1/2}\\
		&\leq \frac{2}{\sigma^{2}}\|\xi_{A}\|_{4/3}\sum_{(i,j)\in  D_{A}} \|X_{i}\|_{4}\|X_{j}\|_{4}\nonumber\\
		&\quad+\frac{4}{\sigma^{2}}\|\xi_{A}\|_{4/3}
		\Big(\sum_{i \in [n]}\sum_{j\in A_{i}}\sum_{k\in A_{ij}}\sum_{l\in N_{k}\cup A_{k}}\|X_i\|_{4}\|X_j\|_{4}\|X_{k}\|_{4}\|X_{l}\|_{4}\Big)^{1/2}\nonumber\\
	&\leq 2\|\xi_{A}\|_{4/3}\delta_{4}+4\|\xi_{A}\|_{4/3}\delta_{6}.	\nonumber
	\end{align}
	{\it (iv) Upper bound of $H_{4}.$} For $H_{4}$, observe that for any $\eta, \eta', \zeta,\zeta'\in \R$,
	\begin{align}\label{eq-p-10}
		\sup_{w\in \R}|f_{\eta,\zeta}(w)-f_{\eta',\zeta'}(w)|\leq \frac{|\eta- \eta'|}{2}+\frac{|\zeta-\zeta'|}{2},
	\end{align}
	which further implies
	\begin{align}\label{eq-p-11}
    	\begin{aligned}
		|H_{4}|&\leq \frac{c}{\sigma^{2}}\sum_{k\in N_{A}^{c}}\sum_{m\in B\cap A_{k}}\e |\xi_{A}X_{k}X_{m}|\\
		&\leq \frac{c}{\sigma^{2}}\sum_{k\in N_{A}^{c}}\sum_{m\in B\cap A_{k}}\|\xi_{A}X_{k}\|_{4/3}\|X_{m}\|_{4}\\
		&\leq \frac{c}{\sigma^{2}}\|\xi_{A}\|_{4/3}\cdot \sum_{m\in B}\sum_{k\in N_{m}}\|X_{m}\|_{4}\|X_{k}\|_{4}\\
		&\leq \|\xi_{A}\|_{4/3}\cdot \delta_{3}.
        \end{aligned}
	\end{align}
	\noindent
	{\it (v) Upper bound of $H_{2}$.} Now, it suffices to bound $H_2$, for which we use the recursive argument. Note that
	$$
	\begin{aligned}
		\left|H_2\right| &\leq  \frac{1}{\sigma^{2}}\sum_{k \in  N_A^c} \mathbb{E}\Big\{\int_{\mathbb{R}} \xi_A|f_{\eta_B, \zeta_B}^{\prime}((S_{A}+t)/\sigma)-f_{\eta_B, \zeta_B}^{\prime}(S_{A}/\sigma)|\cdot| \widehat{M}_k(t)| d t\Big\} \\
		&:=H_{21}+H_{22},
	\end{aligned}
	$$
	where
	$$
	\begin{aligned}
		& H_{21}=\frac{1}{\sigma^{3}}\sum_{k \in  N_A^c} \mathbb{E}\left\{\frac{\xi_A}{\alpha} \int_{\mathbb{R}}\int_{t\wedge 0}^{t\vee 0} \mathbf{1}\left(\eta_B-\alpha \leq (S_{A}+s)/\sigma \leq \eta_B\right) \cdot |\widehat{M}_{k}(t)| d s d t\right\}, \\
		& H_{22}=\frac{1}{\sigma^{3}}\sum_{k \in  N_A^c} \mathbb{E}\left\{\frac{\xi_A}{\alpha} \int_{\mathbb{R}} \int_{t\wedge 0}^{t\vee 0}  \mathbf{1}\left(\zeta_B \leq (S_{A}+s)/\sigma \leq \zeta_B+\alpha\right)\cdot |\widehat{M}_k(t)| d s d t\right\}.
	\end{aligned}
	$$
	For the indicator function in $H_{21}$, as $|s|\leq |t| \leq\left|Y_{k,A}\right| \leq \sum_{l \in  A_k\cap N_{A}^{c} }\left|X_l\right|$,
	then
	\begin{align}\label{eq-p-12}
    \begin{aligned}
        		& \mathbf{1}(\eta_B-\alpha \leq (S_{A}+s)/\sigma \leq \eta_B) \\
		& \leq \mathbf{1}\Big(a-\frac{c}{\sigma} \sum_{m \in B}|X_{m}|-\frac{1}{\sigma}\sum_{l \in  A_k}\left|X_l\right|-\alpha \leq \frac{S_{A}}{\sigma} \leq a+\frac{c}{\sigma} \sum_{m \in B}\left|X_m\right|+\frac{1}{\sigma}\sum_{l \in  A_k}\left|X_l\right|\Big) \\
		&\leq \mathbf{1}(a-\alpha-U_k \leq W \leq a+U_k),
    \end{aligned}
	\end{align}
	where
	$$
	U_k=\frac{c+1}{\sigma}\sum_{l \in B\cup A_{k}}\left|X_l\right|.
	$$
	By (\ref{eq-p-12}), we have
	\begin{align}\label{eq-p-13}
        	\left|H_{21}\right|& \leq \frac{1}{\alpha\sigma^{3}} \sum_{k \in  N_A^c} \mathbb{E}\left\{\xi_A\mathbf{1}\left(a-\alpha-U_k \leq S_{A}/\sigma \leq a+U_k\right)|X_k Y_{k,A}^{2}|\right\}\nonumber\\
		&\leq \frac{1}{\alpha\sigma^{3}} \sum_{k \in  N_A^c}|A_{k}|\sum_{l\in A_{k}\cap N_{A}^{c}} \mathbb{E}\left\{\xi_A |X_{k}X_{l}^{2}|\mathbf{1}\left(a-\alpha-U_k \leq S_{A}/\sigma \leq a+U_k\right)\right\}\\
		&\leq H_{211}+H_{212},\nonumber
	\end{align}
    where 
    \begin{align*}
        H_{211}&=\frac{1}{3\alpha\sigma^{3}} \sum_{k \in  N_A^c} |A_{k}|^{2} \mathbb{E}\big\{\xi_A |X_{k}|^{3}\mathbf{1}\left(a-\alpha-U_k \leq S_{A}/\sigma \leq a+U_k\right)\big\},\\
        H_{212}&=\frac{2}{3\alpha\sigma^{3}} \sum_{k \in  N_A^c}|A_{k}|\sum_{l\in A_{k}\cap N_{A}^{c}} \mathbb{E}\big\{\xi_A |X_{l}|^{3}\mathbf{1}\left(a-\alpha-U_k \leq S_{A}/\sigma \leq a+U_k\right)\big\}.
    \end{align*}
	For $H_{211}$, define
	\[
	A_{1,k}=A\cup \{k\},\quad S_{A_{1,k}}=S-\sum_{i\in N_{A_{1,k}}}X_{i} \quad \text{and}\quad B_{k}=B\cup A_{k}\cup N_{k}.
	\]
	Note that $N_{A}\subset N_{A_{1,k}}\subset N_{A}\cup N_{k}$, then $$|S_{A}-S_{A_{1,k}}|=\Big|\sum_{i\in N_{A_{1,k}}}X_{i}-\sum_{i\in N_{A}}X_{i}\Big|\leq \sum_{i\in N_{k}}|X_{i}|,$$ which further implies
	\begin{align}\label{eq-p-14}
		\mathbf{1}\big(a-\alpha-U_k \leq S_{A}/\sigma \leq a+U_k\big)\leq  \mathbf{1}\big(a-\alpha-U'_k \leq S_{A_{1,k}}/\sigma \leq a+U'_k\big),
	\end{align}
	where
	\[
	U'_{k}=\frac{c+2}{\sigma}\sum_{l \in B_{k}}\left|X_l\right|.
	\]
	By the definition of $K$, we have
	\begin{align}\label{eq-p-15}
        &\mathbb{E}\{\xi_{A}|X_{k}|^{3}\mathbf{1}(a-\alpha-U'_k \leq S_{A_{1,k}}/\sigma \leq a+U'_k)\}\nonumber\\
		&\quad\leq K\|\xi_{A} X_{k}^{3}\|_{4/3}\cdot\Big\{\frac{\alpha}{100}+\frac{c+2}{\sigma}\sum_{m\in N_{A}}\|X_{m}\|_{4}+\frac{c+2}{\sigma}\sum_{m\in N_{k}}\|X_{m}\|_{4}\nonumber\\
		&\quad\quad +\frac{c+2}{\sigma}\sum_{m\in B}\|X_{m}\|_{4}+\frac{c+2}{\sigma}\sum_{m\in  A_{k}}\|X_{m}\|_{4}+\frac{c+2}{\sigma}\sum_{m\in N_{k}}\|X_{m}\|_{4}\\
		&\quad\quad +\frac{c+2}{\sigma^{2}}\sum_{m\in  A_{k}}\sum_{l\in  N_{m}}\|X_{m}\|_{4}\|X_{l}\|_{4}+
		\frac{c+2}{\sigma^{2}}\sum_{m\in B}\sum_{l\in  N_{m}}\|X_{m}\|_{4}\|X_{l}\|_{4}\nonumber\\
		&\quad\quad+\frac{c+2}{\sigma^{2}}\sum_{m\in  N_{k}}\sum_{l\in  N_{m}}\|X_{m}\|_{4}\|X_{l}\|_{4}+\frac{c+2}{\sigma^{2}}\sum_{(i,j)\in D_{A}}\|X_{i}\|_{4}\|X_{j}\|_{4} \nonumber\\
		&\quad\quad+\frac{c+2}{\sigma^{2}}\sum_{(i,j)\in D_{k}}\|X_{i}\|_{4}\|X_{j}\|_{4}+\frac{c+2}{c}(\delta_{5}+\delta_{6}+\delta_{7})
		\Big\}.\nonumber
	\end{align}
	By the definitions of $\delta_{6}$ and $\delta_{7}$, we have
	\begin{align}\label{eq-p-16}
	\frac{c}{\sigma^{4}}\sum_{k\in [n]} |A_{k}|^{2} \|X_{k}\|_{4}^{3}\sum_{m\in A_{k}\cup N_{k}}\|X_{m}\|_{4} \leq \delta_{6}^{2},
	\end{align}
	and 
		\begin{align}
		\frac{c}{\sigma^{5}}\sum_{k\in [n]} |A_{k}|^{2} \|X_{k}\|_{4}^{3}\sum_{(i,j)\in D_{k}}\|X_{i}\|_{4}\|X_{j}\|_{4} &\leq \delta_{7}^{2}, \label{eq-p-17}\\
		\frac{c}{\sigma^{5}}\sum_{k\in [n]} |A_{k}|^{2} \|X_{k}\|_{4}^{3}\sum_{m\in  A_{k}\cup N_{k}}\sum_{l\in  N_{m}}\|X_{m}\|_{4}\|X_{l}\|_{4} &\leq \delta_{7}^{2}. \label{eq-p-18}
	\end{align}
	Combining (\ref{eq-p-13})--(\ref{eq-p-18}) yields
	\begin{align*}
		|H_{211}|&\leq \frac{K\|\xi_{A}\|_{4/3}\delta_{5}}{c\alpha}\cdot\Big(0.01\alpha+\delta_{1}+\delta_{2}+\delta_{3}+\delta_{4}+\delta_{5}+\delta_{6}+\delta_{7} \Big)\nonumber\\
        &\quad+\frac{3K\|\xi_{A}\|_{4/3}(\delta_{6}^{2}+\delta_{7}^{2})}{\alpha}.
	\end{align*}
	Recalling that
	\[
	\alpha=100(\delta_{5}+\delta_{6}+\delta_{7})\quad \text{and}\quad c\geq 1,
	\] then
	\begin{align}\label{eq-p-20}
		|H_{211}|&\leq 0.1\|\xi_{A}\|_{4/3} K(\delta_{1}+\delta_{2}+\delta_{3}+\delta_{4}+\delta_{5}+\delta_{6}+\delta_{7}).
	\end{align}
	For $H_{212}$, note that
	\begin{align*}
		\mathbf{1}\left(a-\alpha-U_k \leq S_{A}/\sigma \leq a+U_k\right)\big\}\leq  \mathbf{1}\left(a-\alpha-U_{k,l} \leq S_{A_{1,l}}/\sigma \leq a+U_{k,l}\right)\big\},
	\end{align*}
	where
	\[
	U_{k,l}=\frac{c+2}{\sigma}\sum_{l \in B\cup A_{k}\cup N_{l}}\left|X_l\right|.
	\]
	With similar arguments as that leading to (\ref{eq-p-20}), we have
	\begin{align}\label{eq-p-22}
		|H_{212}|\leq 0.2\|\xi_{A}\|_{4/3} K(\delta_{1}+\delta_{2}+\delta_{3}+\delta_{4}+\delta_{5}+\delta_{6}+\delta_{7}).
	\end{align}
	Combining (\ref{eq-p-20}) and (\ref{eq-p-22}) yields
	\begin{align*}
		|H_{21}|\leq 0.3\|\xi_{A}\|_{4/3} K(\delta_{1}+\delta_{2}+\delta_{3}+\delta_{4}+\delta_{5}+\delta_{6}+\delta_{7}).
	\end{align*}
	We can obtain the same upper bound for $H_{22}$ with similar argument as before, then
	\begin{align}\label{eq-p-24}
		|H_{2}|\leq 0.6\|\xi_{A}\|_{4/3} K(\delta_{1}+\delta_{2}+\delta_{3}+\delta_{4}+\delta_{5}+\delta_{6}+\delta_{7}).
	\end{align}
	It follows from (\ref{eq-p-02}), (\ref{eq-p-07}), (\ref{eq-p-08}), (\ref{eq-p-09}), (\ref{eq-p-11}) and  (\ref{eq-p-24}) that
	\begin{align*}
        &0.98\mathbb{E}\{\xi_A \mathbf{1}(\eta_B \leq S_{A}/\sigma \leq \zeta_B)\}\\
		& \quad \leq 0.55 (b-a)\left\|\xi_A\right\|_{4/3}+55\left\|\xi_A\right\|_{4/3}(\delta_{5}+\delta_{6}+\delta_{7})\\
        &\qquad+\|\xi_A\|_{4/3}(2\delta_{2}+\delta_{3}+2\delta_{4}+4\delta_{6}) \\
		& \qquad+0.6\|\xi_{A}\|_{4/3} K(\delta_{1}+\delta_{2}+\delta_{3}+\delta_{4}+\delta_{5}+\delta_{6}+\delta_{7})\\
		&\quad\leq 59\left\|\xi_A\right\|_{4/3}\big( 0.01(b-a)+\delta_{1}+\delta_{2}+\delta_{3}+\delta_{4}+\delta_{5}+\delta_{6}+\delta_{7}\big)\\
		& \qquad+0.6\|\xi_{A}\|_{4/3} K(\delta_{1}+\delta_{2}+\delta_{3}+\delta_{4}+\delta_{5}+\delta_{6}+\delta_{7}).
	\end{align*}
	By the definition of $K$, we have
	$$
	0.98K \leq 59+0.6 K,
	$$
	and hence
	$$
	K \leq 156.
	$$
	This completes the proof.
\end{proof}
In what follows, we provide a concentration inequality for self-normalized  sums of locally dependent random variables. 

\begin{prop}\label{prop:2} Let $A \subset[n]$ and $B \subset[n]$ be two arbitrary nonempty index sets. Moreover, let $ \xi_A=h\left(X_A\right) \geqslant 0$ and
	\begin{align*}
		\bar{V}_{A}&=\psi\Big(\sum_{k\in N_{A}^{c}}\sum_{l\in N_{A}^{c}\cap A_{k}} X_{k}X_{l}\Big),\quad && S_{A}=\sum_{k\in N_{A}^{c}}X_{k},\nonumber\\
		\eta_{A,B}&=a-c\sum_{m \in B}\left|X_m\right|/\sigma-c|S_{A}|\cdot Q_{A}/\sigma,\quad
		&&\zeta_{A,B}=b+c\sum_{m \in B}\left|X_m\right|/\sigma+c|S_{A}|\cdot Q_{A}/\sigma,
	\end{align*} where
	\begin{align*}
		&\psi(x)=((x\vee (0.25\sigma^{2}))\wedge (2\sigma^{2}))^{1/2},\quad T_{A}^{2}=\frac{1}{\sigma^{2}}\sum_{k\in N_{A}}\sum_{l\in A_{k}}|X_{k}X_{l}|+\frac{1}{\sigma^{2}}\sum_{k\in N_{A}}\sum_{l\in N_{k}}|X_{k}X_{l}|,\nonumber\\
		& Q_{A}=\min \{1, T_{A}\}
	\end{align*}
	and $a, b \in \mathbb{R}$ satisfy  $a \leq b$ and $c\geq 1$. Then,
	\begin{align*}
		\mathbb{E}\{\xi_A \mathbf{1}\big(\eta_{A,B} \leq S_{A}/\bar{V}_{A} \leq \zeta_{A,B}\big)\} \leq 8755\left\|\xi_A\right\|_{4/3}\Big\{\frac{(b-a)}{1500}+\delta_1+\delta_{2}+\delta_{3}+\delta_{4}\Big\},
	\end{align*}
	where
	$$
	\begin{aligned}
		\lambda&=\kappa\sum_{k\in [n]}\|X_{k}\|_{2}^{2}/\sigma^{2},\qquad  \delta_1 =\frac{c}{\sigma} \sum_{m \in B}\left\|X_m\right\|_4,  \\
		\delta_{2}  &=\frac{c\lambda\kappa^2|A|^2 }{\sigma^3} \sum_{i=1}^n\left\|X_i\right\|_{4}^3, \\
		\delta_{3} & =\frac{c\lambda\kappa^{1/2}(\kappa+\tau^{1/2})|A|^{1 / 2} }{\sigma^2}\big(\sum_{i=1}^n\left\|X_i\right\|_4^4\big)^{1/2},\\
		\delta_{4}^{2} &=\frac{\lambda^{2}c^{2}}{\sigma^{2}}\sum_{k\in N_{A}}\sum_{l\in N_{k}\cup A_{k}}\|X_{k}\|_{4}\|X_{l}\|_{4}.
	\end{aligned}
	$$
\end{prop}
\begin{proof}
 We  prove this proposition by a recursive argument. To this end, let $K$ be defined as
	$$
	K=\sup \frac{\e\{\xi_{A}\mathbf{1}(\eta_{A,B}\leq S_{A}/\bar{V}_{A}\leq \zeta_{A,B})\}}{ (b-a)/1500+\delta_1+\delta_{2}+\delta_{3}+\delta_{4}}
	$$
	where the supremum is taken over all  $A, B,a,b,c$   as in Proposition \ref{prop:2}.  Without loss of generality, we assume that 
\begin{align}\label{eq-assume-prop-2}
	\delta_{2} \leq  \frac{1}{500}\quad\text{and}\quad \delta_{3}\leq \frac{1}{500}, 
	\end{align}
    otherwise the inequality is trivial.
	Let $f_{\eta_{A,B}, \zeta_{A,B},\alpha}(w)$ be defined by (\ref{eq-fabw}) with
	$$
	\alpha=1500\left(\delta_{2}+\delta_{3}\right).
	$$
    Again,  for simplicity of notation, in what follows we omit the dependence on $\alpha$ and write
$f_{\eta_{A,B}, \zeta_{A,B}}(w) := f_{\eta_{A,B}, \zeta_{A,B},\alpha}(w).$
	%
	Further, for any $k\in N_{A}^{c}$, define
	\begin{align*}
		\bar{V}_{A}^{(k)}&=\psi\Big(\sum_{i\in A_{k}^{c}\cap N_{A}^{c}}\sum_{j\in N_{A}^{c}\cap A_{i}\cap A_{k}^{c}} X_{i}X_{j}\Big),\\
		\eta_{A,B}^{(k)}&=a-\frac{c}{\sigma}\sum_{m \in B\cap A_{k}^{c}}\left|X_m\right|-\frac{c}{\sigma}|S_{A}^{(k)}|\cdot Q_{A}^{(k)},\\
		\zeta_{A,B}^{(k)}&=b+\frac{c}{\sigma}\sum_{m \in B\cap A_{k}^{c}}\left|X_m\right|+\frac{c}{\sigma}|S_{A}^{(k)}|\cdot Q_{A}^{(k)},\\
		T_{A}^{(k)2}&=\frac{1}{\sigma^{2}}\sum_{i\in N_{A}\cap A_{k}^{c}}\sum_{l\in A_{i}\cap A_{k}^{c}}|X_{i}X_{j}|+\frac{1}{\sigma^{2}}\sum_{i\in N_{A}\cap A_{k}^{c}}\sum_{j\in N_{i}\cap A_{k}^{c}}|X_{i}X_{j}|,\nonumber\\
		S_{A}^{(k)}&=\sum_{i\in N_{A}^{c}\cap A_{k}^{c}}X_{i},\quad Q_{A}^{(k)}=\min\{1, T_{A}^{(k)}\}.
	\end{align*}
	Then, it follows that $X_k$ is independent of $\big(\xi_A, \eta_{A,B}^{(k)}, \zeta_{A,B}^{(k)}, S_{A}^{(k)}, \bar{V}_{A}^{(k)}\big)$ if $k\in N_{A}^{c}$. Therefore, with $W_{A}=S_{A}/\bar{V}_{A}$, we have
	\begin{align*}
        		&\mathbb{E}\big\{\xi_A W_{A} f_{\eta_{A,B}, \zeta_{A,B}}(W_{A})\big\}\\
		&=  \sum_{k \in  N_A^c} \mathbb{E}\Big\{\xi_A\cdot\Big( \frac{X_k}{\bar{V}_{A}}\Big) f_{\eta_{A,B}, \zeta_{A,B}}(W_{A})\Big\} \\
		&=  \sum_{k \in  N_A^c} \mathbb{E}\Big\{\xi_A \cdot\Big( \frac{X_k}{\bar{V}_{A}}\Big)\big(f_{\eta_{A,B}, \zeta_{A,B}}(W_{A})-f_{\eta_{A,B}, \zeta_{A,B}}((S_{A}-Y_{k,A})/\bar{V}_{A})\big)\Big\} \\
		&\quad +\sum_{k \in  N_A^c} \mathbb{E}\Big\{\xi_A \Big(\Big( \frac{X_k}{\bar{V}_{A}}\Big)\cdot f_{\eta_{A,B}, \zeta_{A,B}}\Big(\frac{S_{A}-Y_{k,A}}{\bar{V}_{A}}\Big)-\Big( \frac{X_k}{\bar{V}_{A}^{(k)}}\Big)\cdot f_{\eta_{A,B}^{(k)}, \zeta_{A,B}^{(k)}}\Big(\frac{S_{A}-Y_{k,A}}{\bar{V}_{A}^{(k)}}\Big)\Big)\Big\} \\
		&= H_1+H_2+H_3+H_{4},
	\end{align*}
	where
	\begin{align*}
		H_1&= \mathbb{E}\Big\{\int_{\R}\frac{\xi_{A}}{\bar{V}_{A}^{2}}\cdot f_{\eta_{A,B}, \zeta_{A,B}}^{\prime}(W_{A}) M(t) \mathrm{d} t\Big\}, \\
		H_2&=\sum_{k \in  N_A^c} \mathbb{E}\Big\{\int_{\R}\frac{\xi_{A}}{\bar{V}_{A}^{2}}\cdot \Big(f_{\eta_{A,B}, \zeta_{A,B}}^{\prime}\Big(\frac{S_{A}+t}{\bar{V}_{A}}\Big)-f_{\eta_{A,B}, \zeta_{A,B}}^{\prime}\Big( \frac{S_{A}}{\bar{V}_{A}} \Big)\Big) \widehat{M} _{k}(t) \mathrm{d} t\Big\},\\
		H_{3}&=\mathbb{E}\Big\{\int_{\R}\frac{\xi_{A}}{\bar{V}_{A}^{2}}\cdot f_{\eta_{A,B}, \zeta_{A,B}}^{\prime}(W_{A}) (\widehat{M}(t)-M(t)) \mathrm{d} t\Big\}, \\
		H_4&=\sum_{k \in  N_A^c} \mathbb{E}\Big\{\xi_A \Big[\Big( \frac{X_k}{\bar{V}_{A}}\Big)\cdot f_{\eta_{A,B}, \zeta_{A,B}}\Big(\frac{S_{A}-Y_{k,A}}{\bar{V}_{A}}\Big)-\Big( \frac{X_k}{\bar{V}_{A}^{(k)}}\Big)\cdot f_{\eta_{A,B}^{(k)}, \zeta_{A,B}^{(k)}}\Big(\frac{S_{A}-Y_{k, A}}{\bar{V}_{A}^{(k)}}\Big)\Big]\Big\}
	\end{align*}
    and $M(t),\widehat{M}_{k}(t), \widehat{M}(t)$ are defined as in (\ref{eq-definition-of-M}).
	Next, we will provide the upper bounds of $|\mathbb{E}\{\xi_A W_{A} f_{\eta_{A,B}, \zeta_{A,B}}(W_{A})\}|, H_2$, $H_{3}$ and $H_4$, and the lower bound of $H_1$. We remark that we  use the recursive argument when we prove the upper bound of $H_2$. Now, we prove the bounds one by one.
	\\ \hspace*{\fill} \\
	{\it(i) Upper bound of $|\mathbb{E}\{\xi_A W_{A} f_{\eta_{A,B}, \zeta_{A,B}}(W_{A})\}|$.} Recall that $|\eta_{A,B}-\zeta_{A,B}|\leq (b-a)+2c\sum_{m \in B}\left|X_m\right|/\sigma+2c|S_{A}|\cdot Q_{A}/\sigma$, it follows from the definition of $f_{\eta_{A,B},\zeta_{A,B}}(\cdot)$ and  H\"{o}lder's inequality that
	\begin{align}\label{eq-pp-02}
        		&\big|\mathbb{E}\big\{\xi_A W_{A} f_{\eta_{A,B}, \zeta_{A,B}}(W_{A})\big\}\big|\nonumber\\
		&\quad\leq  \mathbb{E}\Big|\xi_A W_{A} \frac{\zeta_{A,B}-\eta_{A,B}+\alpha}{2}\Big| \nonumber\\
		&\quad\leq  \frac{(b-a+\alpha)}{2}\e|\xi_AW_{A}|  +\frac{c}{\sigma} \sum_{m \in B} \mathbb{E}|X_m \xi_A W_{A}|+\frac{c}{\sigma}\e |\xi_{A}W_{A}S_{A}Q_{A}| \\
		&\quad\leq  \frac{b-a+\alpha}{\sigma}\e|\xi_A S_{A}|  +\frac{2c}{\sigma^{2}} \sum_{m \in B} \mathbb{E}|X_m \xi_A S_{A}|+\frac{2c}{\sigma^{2}}\e |\xi_{A}S_{A}^{2}Q_{A}|.\nonumber
	\end{align}
 By (\ref{eq-tt-9.1}), we have
	\begin{align}\label{eq-pp-2.2}
		0.874\sigma\leq \|S_{A}\|\leq 1.126\sigma.
	\end{align}
Combining (\ref{eq-tt-9.2}),  (\ref{eq-p-06}), (\ref{eq-pp-2.2}) and Lemma \ref{lem-second-moment-for-S}, for the first two terms of (\ref{eq-pp-02}), we have 
	\begin{align}\label{eq-pp-03}
    \begin{aligned}
         &\frac{b-a+\alpha}{\sigma}\e|\xi_A S_{A}|  +\frac{2c}{\sigma^{2}} \sum_{m \in B} \mathbb{E}|X_m \xi_A S_{A}|\\
		 &\quad\leq 1.2(b-a+\alpha)\|\xi_{A}\|_{4/3}+\|\xi_{A}\|_{4/3}\cdot \frac{3c}{\sigma}\sum_{m\in B}\|X_{m}\|_{4}\\
		 &\quad\leq 1.2(b-a+\alpha)\|\xi_{A}\|_{4/3}+3\|\xi_{A}\|_{4/3}\cdot \delta_{1}.
    \end{aligned}	
	\end{align}
	As for $\e |\xi_{A}S_{A}^{2}Q_{A}|$, recall that $Q_{A}=\min\{1, T_{A}\}$, by H\"{o}lder's inequality, we have
	\begin{align}\label{eq-pp-04}
		\e |\xi_{A}S_{A}^{2}Q_{A}|\leq \e |\xi_{A}S_{A}^{2}T_{A}|\leq \|\xi_{A}S_{A}^{2}\|_{4/3} \cdot\|T_{A}\|_{4}.
	\end{align}
For the second term of (\ref{eq-pp-04}), we observe that
	\begin{align*}
        		\e\{T_{A}^{4}\}&\leq \frac{1}{\sigma^{4}}\e \Big(\sum_{k\in N_{A}}\sum_{l\in A_{k}}|X_{k}X_{l}|+\sum_{k\in N_{A}}\sum_{l\in N_{k}}|X_{k}X_{l}|\Big)^{2}\\
		&\leq \frac{2}{\sigma^{4}}\e \Big(\sum_{k\in N_{A}}\sum_{l\in A_{k}}|X_{k}X_{l}|\Big)^{2}+\frac{2}{\sigma^{4}}\e \Big(\sum_{k\in N_{A}}\sum_{l\in N_{k}}|X_{k}X_{l}|\Big)^{2}\\
		&\leq \frac{2}{\sigma^{4}} \Big(\sum_{k\in N_{A}}\sum_{l\in A_{k}}\|X_{k}\|_{4}\|X_{l}\|_{4} \Big)^{2}+\frac{2}{\sigma^{4}}\Big(\sum_{k\in N_{A}}\sum_{l\in N_{k}}\|X_{k}\|_{4}\|X_{l}\|_{4} \Big)^{2}\\
		&\leq  \frac{4}{\sigma^{4}} \Big(\sum_{k\in N_{A}}\sum_{l\in A_{k}\cup N_{k}}\|X_{k}\|_{4}\|X_{l}\|_{4} \Big)^{2},
	\end{align*}
	which implies that
	\begin{align}\label{eq-pp-06}
		\|T_{A}\|_{4}\leq \frac{\sqrt{2}}{\sigma} \Big(\sum_{k\in N_{A}}\sum_{l\in N_{k}\cup A_{k}}\|X_{k}\|_{4}\|X_{l}\|_{4} \Big)^{1/2}.
	\end{align}
    For the first term of (\ref{eq-pp-04}), by H\"{o}lder's inequality and  Lemma \ref{lem-fourth-moment-for-S}, we have 
    \begin{align}\label{eq-pp-6.1}
       \|\xi_{A}S_{A}^{2}\|_{4/3}&\leq \big(\e\{\xi_{A}^{4/3}\}\big)^{1/4}\cdot \big(\e\{\xi_{A}^{4/3}S_{A}^{4}\}\big)^{1/2}\leq \sqrt{13\lambda}\sigma^{2}\|\xi_{A}\|_{4/3}.
    \end{align}
    Combining (\ref{eq-pp-04}), (\ref{eq-pp-06}) and (\ref{eq-pp-6.1}), we have
	\begin{align}\label{eq-pp-07}
		\frac{2c}{\sigma^{2}}\e |\xi_{A}S_{A}^{2}T_{A}|\leq \frac{2c}{\sigma^{2}} \|\xi_{A}S_{A}^{2}\|_{4/3}\|T_{A}\|_{4}\leq 11\|\xi_{A}\|_{4/3}\delta_{4}.
	\end{align}
	Combining (\ref{eq-pp-02}), (\ref{eq-pp-03}) and (\ref{eq-pp-07}) yields 
	\begin{align}\label{eq-pp-08}
    \begin{aligned}
        		&\big|\mathbb{E}\big\{\xi_A W_{A} f_{\eta_{A,B}, \zeta_{A,B}}(W_{A})\big\}\big|\\
        &\quad
		\leq 1.2(b-a+\alpha)\|\xi_{A}\|_{4/3}+3\|\xi_{A}\|_{4/3}\cdot \delta_{1} +11\|\xi_{A}\|_{4/3}\delta_{4}.
    \end{aligned}
	\end{align}
	
	\noindent
	{\it (ii) Lower bound of $H_{1}.$} By (\ref{eq-p-7.1}) and (\ref{eq-pp-2.2}), we have
	\begin{align*}
		\int_{\R} M(t)\mathrm{d}t&=\e\{S_{A}^{2}\}\geq 0.76.
	\end{align*}
	Noting that $\bar{V}_{A}^{2}\leq 2\sigma^{2}$ and using the inequality $f_{\eta_{A,B}, \zeta_{A,B}}^{\prime}(W_{A})\geq \mathbf{1}(\eta_{A,B}\leq W_{A}\leq \zeta_{A,B})$, we obtain 
	\begin{align}\label{eq-pp-10}
		H_{1}&\geq  0.76\mathbb{E}\Big\{\frac{\sigma^{2}\xi_{A}\mathbf{1}(\eta_{A,B}\leq W_{A}\leq \zeta_{A,B})}{\bar{V}_{A}^{2}}\Big\}\geq 0.38 \e\{ \xi_{A}\mathbf{1}(\eta_{A,B}\leq W_{A}\leq \zeta_{A,B}) \}.
	\end{align}
	\noindent
	{\it (iii) Upper bound of $H_{3}$.} For $H_{3}$, by the fact that $\|f_{\eta_{A,B}, \zeta_{A,B}}^{\prime}\|_{\infty}\leq 1$,  Lemma \ref{lem-XiYi-moment} and (\ref{eq-tt-9.2}), we have
	\begin{align}\label{eq-pp-11}
    \begin{aligned}
        		|H_{3}|&\leq \frac{4}{\sigma^{2}}\mathbb{E}\Big\{\xi_{A}\cdot \Big|\sum_{k\in N_{A}^{c}}\sum_{l\in N_{A}^{c}\cap A_{k}}(X_{k}X_{l}-\e X_{k}X_{l})\Big|\Big\}\\
		&\leq 16\sqrt{2}(\e\{\xi_{A}\})\cdot  \frac{|A|^{1/2}(\kappa^{3/2}+\kappa^{1/2}\tau^{1/2})}{\sigma^{2}}\Big(\sum_{i\in [n]}\|X_{i}\|_{4}^{4}\Big)^{1/2}\\
		&\leq 22.7\|\xi_{A}\|_{4/3}\delta_{3}.
    \end{aligned}
	\end{align}
	\noindent
	{\it (iv) Upper bound of $H_{4}$.}  For $H_{4}$, we have
	\begin{align}\label{eq-kk-01}
    \begin{aligned}
        		H_{4}&=\sum_{k \in  N_A^c} \mathbb{E}\Big\{\xi_A \Big( \frac{X_k}{\bar{V}_{A}}- \frac{X_k}{\bar{V}_{A}^{(k)}}\Big)\cdot f_{\eta_{A,B}, \zeta_{A,B}}\Big(\frac{S_{A}-Y_{k}}{\bar{V}_{A}}\Big)\Big\}\\
		&\quad+\sum_{k \in  N_A^c} \mathbb{E}\Big\{\xi_A\cdot \Big( \frac{X_k}{\bar{V}_{A}^{(k)}}\Big)\cdot\Big( f_{\eta_{A,B}, \zeta_{A,B}}\Big(\frac{S_{A}-Y_{k}}{\bar{V}_{A}}\Big)- f_{\eta_{A,B}^{(k)}, \zeta_{A,B}^{(k)}}\Big(\frac{S_{A}-Y_{k}}{\bar{V}_{A}^{(k)}}\Big)\Big)\Big\}\\
		&:= H_{41}+H_{42}+H_{43},
    \end{aligned}
	\end{align}
	where
	\begin{align*}
		H_{41}&=\sum_{k \in  N_A^c} \mathbb{E}\Big\{\xi_A \Big( \frac{X_k}{\bar{V}_{A}}- \frac{X_k}{\bar{V}_{A}^{(k)}}\Big)\cdot f_{\eta_{A,B}, \zeta_{A,B}}\Big(\frac{S_{A}-Y_{k}}{\bar{V}_{A}}\Big)\Big\},\nonumber\\
		H_{42}&=\sum_{k \in  N_A^c} \mathbb{E}\Big\{\xi_A\cdot \Big( \frac{X_k}{\bar{V}_{A}^{(k)}}\Big)\cdot\Big( f_{\eta_{A,B}, \zeta_{A,B}}\Big(\frac{S_{A}-Y_{k}}{\bar{V}_{A}}\Big)- f_{\eta_{A,B}^{(k)}, \zeta_{A,B}^{(k)}}\Big(\frac{S_{A}-Y_{k}}{\bar{V}_{A}}\Big)\Big)\Big\},\nonumber\\
		H_{43}&=\sum_{k \in  N_A^c} \mathbb{E}\Big\{\xi_A\cdot \Big( \frac{X_k}{\bar{V}_{A}^{(k)}}\Big)\cdot\Big( f_{\eta_{A,B}^{(k)}, \zeta_{A,B}^{(k)}}\Big(\frac{S_{A}-Y_{k}}{\bar{V}_{A}}\Big)- f_{\eta_{A,B}^{(k)}, \zeta_{A,B}^{(k)}}\Big(\frac{S_{A}-Y_{k}}{\bar{V}_{A}^{(k)}}\Big)\Big)\Big\}.
	\end{align*}
	In what follows, we provide the upper bounds  of $H_{41}, H_{42}$ and $H_{43}$, seperately. For $H_{41}$, 
    Recall that $\bar{V}_{A}\geq \sigma/2$ and $\bar{V}_{A}^{(k)}\geq \sigma/2$, then 
	\begin{align}\label{eq-kk-04}
    \begin{aligned}
        		\Big|\frac{1}{\bar{V}_{A}}-\frac{1}{\bar{V}_{A}^{(k)}}  \Big|&\leq  \frac{|\bar{V}_{A}^{2}-\bar{V}_{A}^{(k)2}|}{\bar{V}_{A}\bar{V}_{A}^{(k)}(\bar{V}_{A}+\bar{V}_{A}^{(k)})}
    \\
        &\leq \frac{4}{\sigma^{3}}\Big(\sum_{i\in N_{A}^{c}\cap A_{k}}\sum_{j\in N_{A}^{c}\cap A_{i}}|X_{i}X_{j}|+
	 		\sum_{i\in N_{A}^{c}\cap A_{k}}\sum_{j\in N_{A}^{c}\cap N_{i}}|X_{i}X_{j}|\Big)\\&:=\frac{4}{\sigma^{3}} G_{k}^{2}. 
    \end{aligned}
	\end{align}
	By the fact that $\|f_{\eta_{A,B}, \zeta_{A,B}}\|_{\infty}\leq (\zeta_{A,B}-\eta_{A,B}+\alpha)/2$ and (\ref{eq-kk-04}), we have
	\begin{align}\label{eq-kk-06}
    \begin{aligned}
        		H_{41} 
		&\leq H_{411}+H_{412}+H_{413},
    \end{aligned}
	\end{align}
    where 
    \begin{align*}
        H_{411}&=\frac{2(b-a+\alpha)}{\sigma^{3}}\sum_{k \in  N_A^c} \mathbb{E}\big\{\xi_A|X_{k}| G_{k}^{2}\big\},\quad H_{412}=\frac{4c}{\sigma^{4}}\sum_{k \in  N_A^c}\sum_{m\in B} \mathbb{E}\big\{\xi_A|X_{k}X_{m}|G_{k}^{2}\big\},\\
	H_{413}&=\frac{4c}{\sigma^4}\sum_{k \in  N_A^c}\mathbb{E}\big\{\xi_A|X_{k}S_{A}|Q_{A}G_{k}^{2}\big\}.
    \end{align*}
	For $H_{411}$, note that by H\"{o}lder's inequality, we have 
	\begin{align}\label{eq-kk-07}
		\frac{1}{\sigma^{3}}\sum_{k \in  N_A^c} \mathbb{E}\big\{\xi_A|X_{k}| G_{k}^{2}\big\}
		&\leq 	\frac{1}{\sigma^{3}}\sum_{k \in  N_A^c} \|\xi_{A}\|_{1}^{1/2}\|X_{k}\|_{2}\cdot \big(\mathbb{E}\big\{\xi_{A}G_{k}^{4}\big\}\big)^{1/2}.
	\end{align}
 For any $p\geq 0$, by H\"{o}lder's inequality,  we have for any $i,j,i',j'\in N_{A}^{c}$,
 \begin{align*}
     \e\{\xi_{A}^{p}|X_{i}X_{j}X_{i'}X_{j'}|\}&\leq \|\xi_{A}^{p/4}X_{i}\|_{4}\|\xi_{A}^{p/4}X_{j}\|_{4}\|\xi_{A}^{p/4}X_{i'}\|_{4}\|\xi_{A}^{p/4}X_{j'}\|_{4}\nonumber\\
     &=\e\{\xi_{A}^{p}\}\cdot \|X_{i}\|_{4}\|X_{j}\|_{4}\|X_{i'}\|_{4}\|X_{j'}\|_{4},
 \end{align*}
 then
 \begin{align}\label{eq-kk-08}
 \begin{aligned}
   &\mathbb{E}\big\{\xi_{A}^{p}G_{k}^{4}\big\}\\
		&\leq 2\sum_{i\in N_{A}^{c}\cap A_{k}}\sum_{j\in N_{A}^{c}\cap A_{i}}\sum_{i'\in N_{A}^{c}\cap A_{k}}\sum_{j'\in N_{A}^{c}\cap A_{i'}}\e\{\xi_{A}^{p}|X_{i}X_{j}X_{i'}X_{j'}|\}\\
		&\quad+
		2\sum_{i\in N_{A}^{c}\cap A_{k}}\sum_{j\in N_{A}^{c}\cap N_{i}}\sum_{i'\in N_{A}^{c}\cap A_{k}}\sum_{j'\in N_{A}^{c}\cap N_{i'}}\e\{\xi_{A}^{p}|X_{i}X_{j}X_{i'}X_{j'}|\}\\
		&\leq 2(\e\{\xi_{A}^{p}\}) \cdot \Big(\Big(\sum_{i\in  A_{k}}\sum_{j\in  A_{i}}\|X_{i}\|_{4}\|X_{j}\|_{4} \Big)^{2}+ \Big(\sum_{i\in  A_{k}}\sum_{j\in  N_{i}}\|X_{i}\|_{4}\|X_{j}\|_{4} \Big)^{2}\Big)\\
        &\leq 4(\e\{\xi_{A}^{p}\}) \cdot \Big(\sum_{i\in  A_{k}}\sum_{j\in  A_{i}\cup N_{i}}\|X_{i}\|_{4}\|X_{j}\|_{4} \Big)^{2}.  
 \end{aligned}	
    \end{align}
	Combining (\ref{eq-assume-prop-2}), (\ref{eq-kk-07}) and (\ref{eq-kk-08}), we have
	\begin{align}\label{eq-kk-09}
    \begin{aligned}
        H_{411}&\leq 4(b-a+\alpha)\|\xi_{A}\|_{1}\cdot \frac{1}{\sigma^{3}}\sum_{k \in  N_A^c}\sum_{i\in  A_{k}}\sum_{j\in  A_{i}\cup N_{i}}\|X_{k}\|_{2}\|X_{i}\|_{4}\|X_{j}\|_{4}\\
		&\leq 8(b-a+\alpha)\|\xi_{A}\|_{1}\cdot \frac{\kappa^{2}}{\sigma^{3}}\sum_{k \in  [n]}\|X_{k}\|_{4}^{3}\\
		&\leq 0.02(b-a+\alpha)\|\xi_{A}\|_{4/3}.
    \end{aligned}	
	\end{align}
	For $H_{412}$, by (\ref{eq-assume-prop-2}), (\ref{eq-kk-08}) and H\"{o}lder's inequality,  we have
	\begin{align}\label{eq-kk-10}
        H_{412}&\leq \frac{4c}{\sigma^4}\cdot\sum_{k \in  N_A^c}\sum_{m\in B} \|\xi_{A}X_{k}G_{k}^{2}\|_{4/3} \|X_{m}\|_{4}\nonumber\\
		&\leq \frac{4c}{\sigma^4}\cdot\sum_{k \in  N_A^c}\sum_{m\in B} \big(\e\{\xi_{A}^{4/3}X_{k}^{4}\}\big)^{1/4}\big(\e\{\xi_{A}^{4/3}G_{k}^{4}\}\big)^{1/2} \|X_{m}\|_{4}\nonumber\\
		&\leq \frac{8c}{\sigma^4}\|\xi_{A}\|_{4/3}\cdot\sum_{k \in  N_A^c}\sum_{i\in  A_{k}}\sum_{j\in  A_{i}\cup N_{i}}\|X_{i}\|_{4}\|X_{j}\|_{4}\|X_{k}\|_{4}\cdot\sum_{m\in B}\|X_{m}\|_{4}\\
		&\leq \frac{16c}{\sigma}\|\xi_{A}\|_{4/3}\sum_{m\in B}\|X_{m}\|_{4}\cdot \frac{\kappa^{2}}{\sigma^{3}}\sum_{k\in[n]}\|X_{k}\|_{4}^{3}\nonumber\\
		&\leq \frac{0.1c}{\sigma}\|\xi_{A}\|_{4/3}\sum_{m\in B}\|X_{m}\|_{4}\nonumber\\
		&\leq 0.1\|\xi_{A}\|_{4/3}\delta_{1}.	\nonumber
	\end{align}
	For $H_{413}$, recall that $Q_{A}\leq 1$, and hence, 
	\begin{align}\label{eq-kk-11}
    \begin{aligned}
        		H_{413}&\leq \frac{4c}{\sigma^4}\sum_{k \in  N_A^c}\e\big\{\xi_{A}|X_{k}S_{A}G_{k}^{2}| \big\}\\
		&\leq \frac{4c}{\sigma^4}\sum_{k \in  N_A^c}\big(\e\{\xi_{A}X_{k}^{2}S_{A}^{2}\}\big)^{1/2}\cdot \big(\e\{\xi_{A}G_{k}^{4}\}\big)^{1/2}.
    \end{aligned}
	\end{align}
	Now, we  provide an upper bound for $\e\{\xi_{A}X_{k}^{2}S_{A}^{2}\}$. By (\ref{eq-tt-9.2}), (\ref{eq-assume-prop-2}), (\ref{eq-pp-2.2}) and Lemma \ref{lem-second-moment-for-S}, we have
	\begin{align}\label{eq-kk-11.1}
		\e\{\xi_{A}^{p}(S_{A}-Y_{k,A})^{2}\}
		&\leq 2\e\{\xi_{A}^{p}S_{A}^{2}\}+2\e\{\xi_{A}^{p}\}\cdot \Big(\kappa^{3}\sum_{l\in [n]}\|X_{l}\|_{4}^{4}\Big)^{1/2}
        \leq 3\sigma^{2}\e\{\xi_{A}^{p}\},
	\end{align}
	which further implies
	\begin{align}\label{eq-kk-12}
    \begin{aligned}
        		\e\{\xi_{A}X_{k}^{2}S_{A}^{2}\}&\leq 2\e\{\xi_{A}X_{k}^{2}(S_{A}-Y_{k,A})^{2}\}+2\e\{\xi_{A}X_{k}^{2}Y_{k,A}^{2}\}\\
		&\leq 2\e\{X_{k}^{2}\} \e\{\xi_{A}(S_{A}-Y_{k,A})^{2}\}+2\kappa\sum_{l\in A_{k}}\e\{\xi_{A}X_{k}^{2}X_{l}^{2}\}\\
		&\leq 6\sigma^{2}\|\xi_{A}\|_{1}\e \{X_{k}^{2}\}+\kappa^{2}\|\xi_{A}\|_{1}\|X_{k}\|_{4}^{4}+\kappa\sum_{l\in A_{k}}\|\xi_{A}\|_{1}\|X_{l}\|_{4}^{4}.
    \end{aligned}
	\end{align}
	Substituting (\ref{eq-kk-08}) and (\ref{eq-kk-12}) into  (\ref{eq-kk-11}) yields
	\begin{align}\label{eq-kk-13}
        H_{413}&\leq \frac{20c\|\xi_{A}\|_{1}}{\sigma^3}\sum_{k \in  N_A^c}\sum_{i\in  A_{k}}\sum_{j\in  A_{i}\cup N_{i}}\|X_{k}\|_{2}\|X_{i}\|_{4}\|X_{j}\|_{4}\nonumber\\
		&\quad+ \frac{8c\kappa\|\xi_{A}\|_{1}}{\sigma^4}\sum_{k \in  N_A^c}\sum_{i\in  A_{k}}\sum_{j\in  A_{i}\cup N_{i}}\|X_{k}\|_{4}^{2}\|X_{i}\|_{4}\|X_{j}\|_{4}\nonumber\\
&\quad+\frac{8c\kappa^{1/2}\|\xi_{A}\|_{1}}{\sigma^4}\sum_{k \in  N_A^c}\Big(\sum_{l\in A_{k}}\|X_{l}\|_{4}^{4}\Big)^{1/2} \sum_{i\in  A_{k}}\sum_{j\in  A_{i}\cup N_{i}}\|X_{i}\|_{4}\|X_{j}\|_{4}\\
		&\leq \frac{40c\kappa^{2}\|\xi_{A}\|_{1}}{\sigma^3}\sum_{i\in[n]}\|X_{i}\|_{4}^{3}+\frac{32c\kappa^{3}\|\xi_{A}\|_{1}}{\sigma^4}\sum_{i\in[n]}\|X_{i}\|_{4}^{4}\nonumber\\
		&\leq 40\|\xi_{A}\|_{4/3}\delta_{2}+0.1\|\xi_{A}\|_{4/3}\delta_{3}.\nonumber
	\end{align}
	It follows from (\ref{eq-kk-06}), (\ref{eq-kk-09}), (\ref{eq-kk-10}) and (\ref{eq-kk-13}) that
	\begin{align}\label{eq-kk-14}
		H_{41}\leq 0.02(b-a+\alpha)\|\xi_{A}\|_{4/3}+0.1\|\xi_{A}\|_{4/3}\delta_{1}+40\|\xi_A\|_{4/3} \delta_{2}+0.1\|\xi_A\|_{4/3}\delta_{3}.
	\end{align}
	For $H_{42}$, combining $\bar{V}_{A}^{(k)}\geq \sigma/2$ and (\ref{eq-p-10}), we obtain  
	\begin{align}\label{eq-pp7}
		|H_{42}|&\leq \frac{1}{\sigma}\sum_{k \in  N_A^c} \e \big\{\xi_{A}|X_{k}|\cdot \big(\big|\eta_{A,B}-\eta_{A,B}^{(k)}\big|+\big|\zeta_{A,B}-\zeta_{A,B}^{(k)}\big|\big)\big\}.
	\end{align}
	By the definition of $\eta_{A,B}$, $\eta_{A,B}^{(k)}$, $\zeta_{A,B}$ and $\eta_{A,B}^{(k)}$, we have
	\begin{align}\label{eq-kk-15}
        &\max\big\{|\eta_{A,B}-\eta_{A,B}^{(k)}|,|\zeta_{A,B}-\zeta_{A,B}^{(k)}|\big\}\nonumber\\
    &\quad\leq c\sum_{m\in B\cap A_{k}}\frac{|X_{m}|}{\sigma}+\frac{c}{\sigma}|S_{A}Q_{A}-S_{A}^{(k)}Q_{A}^{(k)}|\\
		&\quad\leq c\sum_{m\in B\cap A_{k}}\frac{|X_{m}|}{\sigma}+\frac{c|Y_{k,A}|}{\sigma}\cdot Q_{A}+\frac{c|S_{A}-Y_{k,A}|}{\sigma}|Q_{A}-Q_{A}^{(k)}|.\nonumber
	\end{align}
	For the last term of (\ref{eq-kk-15}),
   using the inequality  $|\min\{1,x\}-\min\{1,y\}|\leq |x-y|$, we obtain 
	\begin{align}\label{eq-kk-16}
    \begin{aligned}
        |Q_{A}-Q_{A}^{(k)}|&\leq |T_{A}-T_{A}^{(k)}|\\
		&=\frac{1}{\sigma}\Bigg|\sqrt{\sum_{i\in N_{A}}\sum_{j\in A_{i}}|X_{i}X_{j}|+\sum_{i\in N_{A}}\sum_{j\in N_{i}}|X_{i}X_{j}|}\\
        &\quad-\sqrt{\sum_{i\in N_{A}\cap A_{k}^{c}}\sum_{j\in A_{i}\cap A_{k}^{c}}|X_{i}X_{j}|+\sum_{i\in N_{A}\cap A_{k}^{c}}\sum_{j\in N_{i}\cap A_{k}^{c}}|X_{i}X_{j}|} \Bigg|\\
        &\leq\frac{1}{\sigma}\Delta_{k},
    \end{aligned}
	\end{align}
    where 
    \begin{align*}
       \Delta_{k}^{2}&=\sum_{i\in N_{A}\cap A_{k}}\sum_{j\in A_{i}}|X_{i}X_{j}|+\sum_{i\in N_{A}}\sum_{j\in A_{i}\cap A_{k}}|X_{i}X_{j}|\\
       &\quad+\sum_{i\in N_{A}\cap A_{k}}\sum_{j\in N_{i}}|X_{i}X_{j}|+\sum_{i\in N_{A}}\sum_{j\in N_{i}\cap A_{k}}|X_{i}X_{j}|. 
    \end{align*}
	Combining (\ref{eq-pp7})--(\ref{eq-kk-16}) and by the fact that $Q_{A}\leq T_{A}$, we obtain
	\begin{align}\label{eq-kk-17}
    \begin{aligned}
        |H_{42}|
		&\leq \frac{2c}{\sigma^{2}}\sum_{k \in  N_A^c}\sum_{m\in B\cap A_{k}} \e \{\xi_{A}|X_{k}X_{m}|\}+\frac{2c}{\sigma^{2}}\sum_{k \in  N_A^c} \e \{\xi_{A}|X_{k}Y_{k,A}|T_{A}\}\\
		&\quad+\frac{2c}{\sigma^{3}}\sum_{k \in  N_A^c} \e\{\xi_{A}X_{k}|S_{A}-Y_{k,A}|\cdot|\Delta_{k}|\}\\
		&\leq H_{421}+H_{422}+H_{423},
    \end{aligned}	
	\end{align}
    where 
    \begin{align*}
H_{421}&=\frac{2c}{\sigma^{2}}\|\xi_{A}\|_{4/3}\sum_{k \in  N_A^c}\sum_{m\in B\cap A_{k}} \|X_{k}\|_{4/3} \|X_{m}\|_{4},\\
H_{422}&=\frac{2c}{\sigma^{2}}\|T_{A}\|_{4}\cdot\Big\|\sum_{k \in  N_A^c}\xi_{A}|X_{k}Y_{k,A}|\Big\|_{4/3},\\H_{423}&=\frac{2c}{\sigma^{3}}\sum_{k \in  N_A^c} \|X_{k}\|_{4/3}\|\xi_{A}(S_{A}-Y_{k,A})\|_{4/3}\cdot\|\Delta_{k}\|_{4}.
    \end{align*}
	For $H_{421}$, by (\ref{eq-tt-9.1}), we have
	\begin{align}\label{eq-kk-18}
		H_{421}\leq \frac{2c}{\sigma^{2}}\|\xi_{A}\|_{4/3}\sum_{m\in B}\sum_{k\in N_{m}} \|X_{m}\|_{4}\|X_{k}\|_{4}\leq 0.3\|\xi_{A}\|_{4/3}\delta_{1}.
	\end{align}
	To bound $H_{422}$, we need to obtain an upper bound for $\Big\|\xi_{A}\sum_{k \in  N_A^c}|X_{k}Y_{k,A}|\Big\|_{4/3}$. Note that for any $p\geq 0$, we have
	\begin{align}\label{eq-kk-19}
        &\e\Big\{\xi_{A}^{p}\Big(\sum_{k \in  N_A^c}|X_{k}Y_{k,A}|\Big)^{2}\Big\}\nonumber\\&\quad\leq \sum_{k\in N_{A}^{c}}\sum_{l\in A_{k}\cap N_{A}^{c}}\sum_{k'\in  N_{A}^{c}}\sum_{l'\in A_{k'}\cap N_{A}^{c}}\e\big\{ \xi_{A}^{p}|X_{k}X_{l}X_{k'}X_{l'}|\big\}\nonumber\\
		&\quad\leq \sum_{k\in  N_{A}^{c}}\sum_{l\in A_{k}\cap N_{A}^{c}}\sum_{k'\in A_{kl}^{c}\cap N_{A}^{c}}\sum_{l'\in A_{k'}\cap A_{kl}^{c}\cap N_{A}^{c}}\e \big\{\xi_{A}^{p}|X_{k}X_{l}X_{k'}X_{l'}|\big\}\\
		&\quad\quad+\sum_{k\in  N_{A}^{c}}\sum_{l\in A_{k}\cap N_{A}^{c}}\sum_{k'\in A_{kl}\cap N_{A}^{c}}\sum_{l'\in A_{k'}\cap N_{A}^{c}}\e \big\{\xi_{A}^{p}|X_{k}X_{l}X_{k'}X_{l'}|\big\}\nonumber\\
		&\quad\quad+\sum_{k\in  N_{A}^{c}}\sum_{l\in A_{k}\cap N_{A}^{c}}\sum_{k'\in  N_{A}^{c}}\sum_{l'\in A_{k'}\cap A_{kl}\cap N_{A}^{c}}\e \big\{\xi_{A}^{p}|X_{k}X_{l}X_{k'}X_{l'}|\big\}\nonumber\\
		&\quad\leq \sum_{(k,l)\in D_{A,1}}\sum_{(k',l')\in D_{A,1}}\e \big\{\xi_{A}^{p}|X_{k}X_{l}X_{k'}X_{l'}|\big\}\mathbf{1}(k'\in A_{kl}^{c})\mathbf{1}(l'\in A_{kl}^{c})\nonumber\\
&\quad\quad+\e\{\xi_{A}^{p}\}\cdot(\kappa^{3}+\kappa\tau)\sum_{k\in [n]}\|X_{k}\|_{4}^{4}.\nonumber
	\end{align}
    For the first term of (\ref{eq-kk-19}), recall that  (\ref{eq-tt-11.6})--(\ref{eq-tt-11.8}) implie that
	\begin{align}\label{eq-kk-19.1}
		 \sum_{(k,l)\in D_{A,3}}\sum_{(k',l')\in D_{A,1}}\e \big\{\xi_{A}^{p}|X_{k}X_{l}X_{k'}X_{l'}|\mathbf{1}(k'\in A_{kl}^{c})\mathbf{1}(l'\in A_{kl}^{c})\big\}\leq 0.2\lambda\sigma^{4}\e\{\xi_{A}^{p}\}.
	\end{align}
    Moreover, if $(k,l)\in D_{A,2}$, $k'\in A_{kl}^{c}$ and $l'\in A_{kl}^{c}$ then $(X_{k},X_{l})\perp\!\!\!\perp (\xi_{A},X_{k'},X_{l'})$, which implies
    \begin{align}\label{eq-kk-19.01}
    \begin{aligned}
                &\sum_{(k,l)\in D_{A,2}}\sum_{(k',l')\in D_{A,1}}\e \big\{\xi_{A}^{p}|X_{k}X_{l}X_{k'}X_{l'}|\mathbf{1}(k'\in A_{kl}^{c})\mathbf{1}(l'\in A_{kl}^{c})\big\}\\
        &\qquad\leq  \sum_{(k,l)\in D_{A,1}}\sum_{(k',l')\in D_{A,1}}\e\{|X_{k}X_{l}|\}\cdot \e\big\{\xi_{A}^{p}|X_{k'}X_{l'}|\big\}\leq \lambda^{2}\sigma^{4}\e\{\xi_{A}^{p}\}.
    \end{aligned}
    \end{align}
	Combining (\ref{eq-kk-19}), (\ref{eq-kk-19.1})  and  (\ref{eq-kk-19.01}) yields
	\begin{align*}
		\e\Big\{\xi_{A}^{p}\Big(\sum_{k \in  N_A^c}|X_{k}Y_{k}|\Big)^{2}\Big\}\leq 1.21\lambda^{2}\sigma^{4}\e\{\xi_{A}^{p}\}.
	\end{align*}
	Consequently,
	\begin{align}\label{eq-kk-20}
    \begin{aligned}
        		\Big\|\xi_{A}\sum_{k \in  N_A^c}|X_{k}Y_{k}|\Big\|_{4/3}&\leq \big(\e \{\xi_{A}^{4/3}\}\big)^{1/4}\cdot\Big(\e\Big\{\xi_{A}^{4/3}\Big(\sum_{k \in  N_A^c}|X_{k}Y_{k}|\Big)^{2}\Big\}\Big)^{1/2}\\
                &\leq 1.1\|\xi_{A}\|_{1}\sigma^{2}\lambda.
    \end{aligned}
	\end{align}
	Recall that, as shown in (\ref{eq-pp-06}), we have already established 
	\begin{align*}
		\|T_{A}\|_{4}\leq \frac{\sqrt{2}}{\sigma} \Big(\sum_{k\in N_{A}}\sum_{l\in N_{k}\cup A_{k}}\|X_{k}\|_{4}\|X_{l}\|_{4} \Big)^{1/2}.
	\end{align*}
	Combining this with (\ref{eq-kk-20}), we obtain
	\begin{align}\label{eq-kk-22}
		H_{422}\leq 3.2\|\xi_{A}\|_{4/3} \delta_{4}.
	\end{align}
	For $H_{423}$, by H\"{o}lder's inequality and (\ref{eq-kk-11.1}), we have 
	\begin{align}\label{eq-kk-23}
		\|\xi_{A}(S_{A}-Y_{k,A})\|_{4/3}\leq \big(\e \{\xi_{A}^{4/3}\}\big)^{1/4} \big( \e\{\xi_{A}^{4/3}(S_{A}-Y_{k,A})^{2}\}\big)^{1/2}\leq 1.8\sigma\|\xi_{A}\|_{4/3}.
	\end{align}
	As for $\|\Delta_{k}\|_{4}$, with similar arguments as that leading to (\ref{eq-pp-06}), we have
	\begin{align}\label{eq-kk-24}
    \begin{aligned}
        	\|\Delta_{k}\|_{4}&\leq 1.7 \Big(\sum_{i\in N_{A}\cap A_{k}}\sum_{j\in A_{i}\cup N_{i}}\|X_{i}\|_{4}\|X_{j}\|_{4}\Big)^{1/2}\\
            &\quad+1.7 \Big(\sum_{i\in N_{A}}\sum_{j\in (A_{i}\cup N_{i})\cap A_{k}}\|X_{i}\|_{4}\|X_{j}\|_{4}\Big)^{1/2}.
    \end{aligned}
	\end{align}
	By Cauchy's inequality, we have
	\begin{align}\label{eq-kk-25}
    \begin{aligned}
        		&\frac{c}{\sigma^{2}}\sum_{k\in N_{A}^{c}}\|X_{k}\|_{2}\Big(\sum_{i\in N_{A}\cap A_{k}}\sum_{j\in A_{i}\cup N_{i}}\|X_{i}\|_{4}\|X_{j}\|_{4}\Big)^{1/2}\\
		&\quad\leq \frac{c}{\sigma^{2}}\Big(\sum_{k\in N_{A}^{c}}\|X_{k}\|_{2}^{2}\Big)^{1/2}\cdot \Big(\sum_{k\in N_{A}^{c}}\sum_{i\in N_{A}\cap A_{k}}\sum_{j\in A_{i}\cup N_{i}}\|X_{i}\|_{4}\|X_{j}\|_{4}\Big)^{1/2}\\
		&\quad\leq \frac{c}{\sigma^{2}}\Big(\kappa\sum_{k\in N_{A}^{c}}\|X_{k}\|_{2}^{2}\Big)^{1/2}\cdot \Big(\sum_{i\in N_{A}}\sum_{j\in A_{i}\cup N_{i}}\|X_{i}\|_{4}\|X_{j}\|_{4}\Big)^{1/2}\\
		&\quad\leq \delta_{4}.
    \end{aligned}
	\end{align}
	Similarly, 
	\begin{align}\label{eq-kk-26}
	&\frac{c}{\sigma^{2}}\sum_{k\in N_{A}^{c}}\|X_{k}\|_{2}\Big(\sum_{i\in N_{A}}\sum_{j\in (A_{i}\cup N_{i})\cap A_{k}}\|X_{i}\|_{4}\|X_{j}\|_{4}\Big)^{1/2}\leq \delta_{4}.
	\end{align}
	It follows from (\ref{eq-kk-24}), (\ref{eq-kk-25}) and (\ref{eq-kk-26}) that
	\begin{align}\label{eq-kk-28}
		\frac{c}{\sigma^{2}}\sum_{k\in N_{A}^{c}}\|X_{k}\|_{2}\| \Delta_{k}\|_{4}\leq 3.4\delta_{4}.
	\end{align}
	By (\ref{eq-kk-23}) and (\ref{eq-kk-28}), we have
	\begin{align}\label{eq-kk-29}
		H_{423}\leq \frac{3.6c}{\sigma^{2}}\|\xi_{A}\|_{4/3}\sum_{k \in  N_A^c} \|X_{k}\|_{2}\|\Delta_{k}\|_{4}\leq 13\|\xi_{A}\|_{4/3} \delta_{4}.
	\end{align}
	Combining (\ref{eq-kk-17}), (\ref{eq-kk-18}), (\ref{eq-kk-22}) and (\ref{eq-kk-29}) yields
	\begin{align}\label{eq-kk-30}
		H_{42}\leq 0.3\|\xi_{A}\|_{4/3}\delta_{1}+17\|\xi_{A}\|_{4/3}\delta_{4}.
	\end{align}
	For $H_{43}$, using $\|f'_{\eta_{A,B}, \zeta_{A,B}}\|_{\infty}\leq 1$ together with  (\ref{eq-kk-04}), we have
	\begin{align*}
        	H_{43}&\leq \frac{2}{\sigma}\sum_{k\in N_{A}^{c}}\e\Big\{ \xi_{A}|X_{k}(S_{A}-Y_{k})|\cdot \Big|\frac{1}{\bar{V}_{A}}-\frac{1}{\bar{V}_{A}^{(k)}}\Big|\Big\}\\
		&\leq \frac{8}{\sigma^{4}}\sum_{k\in N_{A}^{c}}\e\Big\{ \xi_{A}|X_{k}(S_{A}-Y_{k,A})|G_{k}^{2}\Big\}\\
		&\leq \frac{8}{\sigma^4}\sum_{k \in  N_A^c}\big(\e\{\xi_{A}X_{k}^{2}(S_{A}-Y_{k,A})^{2}\}\big)^{1/2}\cdot \big(\e\{\xi_{A}G_{k}^{4}\}\big)^{1/2}.
	\end{align*}
    Combining this with   (\ref{eq-kk-08}) and (\ref{eq-kk-11.1}) yields 
    \begin{align}\label{eq-kk-31}
		H_{43}&\leq \frac{28\|\xi_{A}\|_{1}}{\sigma^3}\sum_{k \in  N_A^c}\sum_{i\in A_{k}}\sum_{j\in A_{i}\cup N_{i}}\|X_{k}\|_{4}\|X_{i}\|_{4}\|X_{i}\|_{4}\nonumber\\
		&\leq \frac{56\|\xi_{A}\|_{1}\kappa^{2}}{\sigma^3}\sum_{k \in [n]}\|X_{k}\|_{4}^{3}\\
		&\leq 56\|\xi_{A}\|_{4/3}\delta_{2}.\nonumber
    \end{align}
	Substituting (\ref{eq-kk-14}), (\ref{eq-kk-30}) and (\ref{eq-kk-31}) into (\ref{eq-kk-01}) yields 
	\begin{align}\label{eq-kk-32}
		|H_{4}|\leq \|\xi_{A}\|_{4/3}\big(0.02(b-a+\alpha)+0.4\delta_{1}+96\delta_{2}+0.1\delta_{3}+17\delta_{4}\big).
	\end{align}
	\noindent
	{\it (v) Upper bound of $H_{2}$.} Now, it suffices to bound $H_2$, for which we use the recursive argument.
	By the definition of $f_{\eta_{A,B},\zeta_{A,B}}$, we have
	\begin{align*}
    H_2
		&= \sum_{k \in  N_A^c} \mathbb{E}\Big\{\int_{\R}\frac{\xi_{A}}{\bar{V}_{A}^{2}}\cdot \Big(f_{\eta_{A,B}, \zeta_{A,B}}^{\prime}\Big(\frac{S_{A}+t}{\bar{V}_{A}}\Big)-f_{\eta_{A,B}, \zeta_{A,B}}^{\prime}\Big( \frac{S_{A}}{\bar{V}_{A}} \Big)\Big) \widehat{M} _{k}(t)\mathrm{d} t\Big\} \\
		&\leq H_{21}+H_{22},
	\end{align*}
    where 
    \begin{align*}
        H_{21}&=\sum_{k \in  N_A^c} \mathbb{E}\left\{\frac{\xi_A}{\alpha \bar{V}_{A}^{3}} \int_{\R} \int_{t\wedge 0}^{t\vee 0} \mathbf{1}\Big(\eta_{A,B}-\alpha \leq \frac{S_{A}+s}{\bar{V}_{A}} \leq \eta_{A,B}\Big) |\widehat{M}_k(t)| d s d t\right\}, \\
		H_{22}&=\sum_{k \in  N_A^c} \mathbb{E}\left\{\frac{\xi_A}{\alpha \bar{V}_{A}^{3}} \int_{\R} \int_{t\wedge 0}^{t\vee 0} \mathbf{1}\Big(\zeta_{A,B} \leq \frac{S_{A}+s}{\bar{V}_{A}} \leq \zeta_{A,B}+\alpha\Big) |\widehat{M}_k(t)| d s d t\right\}.
    \end{align*}
	For the indicator function in $H_{21}$, note that $$|s| \leq \left|Y_{k,A}\right| \leq \sum_{l \in N_{A}^{c} \cap A_k}\left|X_l\right|\leq \sum_{l \in A_k}|X_{l}|\quad \text{and}\quad \bar{V}_{A}\geq \sigma/2.$$ Hence, we have
	\begin{align}\label{eq-yy-01}
		& \mathbf{1}\Big(\eta_{A,B}-\alpha \leq \frac{S_{A}+s}{\bar{V}_{A}} \leq \eta_{A,B}\Big)
		\leq \mathbf{1}(a-\alpha-U'_k \leq S_{A}/\bar{V}_{A} \leq a+U'_k),
	\end{align}
	where
	$$
	U'_k=(c+2)\sum_{l \in B\cup A_{k}}\frac{|X_{l}|}{\sigma}+\frac{c |S_{A}|Q_{A}}{\sigma}.
	$$
	It follows from (\ref{eq-yy-01}) that
	\begin{align*}
        		\left|H_{21}\right| & \leq \sum_{k \in  N_A^c} \mathbb{E}\left\{\frac{8\xi_A}{\alpha \sigma^{3}} \int_{\R} \int_{t\wedge 0}^{t\vee 0} \mathbf{1}(a-\alpha-U'_k \leq S_{A}/\bar{V}_{A} \leq a+U'_k) |\widehat{M}_k(t)| d s d t\right\} \\
		& \leq \frac{4}{\sigma^{3}\alpha} \sum_{k \in  N_A^c} \mathbb{E}\big\{\xi_A\mathbf{1}\left(a-\alpha-U'_k \leq S_{A}/\bar{V}_{A} \leq a+U'_k\right)|X_k Y_{k,A}^2 |\big\}\\
		 & \leq \frac{4\kappa}{\sigma^{3}\alpha} \sum_{k \in  N_A^c}\sum_{l\in N_{A}^{c}\cap A_{k}} \mathbb{E}\big\{\xi_A|X_k X_l^2 |\mathbf{1}\left(a-\alpha-U'_k \leq S_{A}/\bar{V}_{A} \leq a+U'_k\right)\big\}\\
		&\leq H_{211}+H_{212},
	\end{align*}
    where 
    \begin{align*}
       H_{211}&= \frac{4\kappa^{2}}{3\sigma^{3}\alpha} \sum_{k \in  N_A^c} \mathbb{E}\big\{\xi_A|X_k^{3} |\mathbf{1}(a-\alpha-U'_k \leq S_{A}/\bar{V}_{A} \leq a+U'_k)\big\},\\
		H_{212}&=\frac{8\kappa^{2}}{3\sigma^{3}\alpha} \sum_{k \in  N_A^c}\sum_{l\in N_{A}^{c}\cap A_{k}} \mathbb{E}\big\{\xi_A|X_l^{3}|\mathbf{1}(a-\alpha-U'_k \leq S_{A}/\bar{V}_{A} \leq a+U'_k)\big\}.
    \end{align*}
	We use recursive arguments to bound $H_{211}$ and $H_{212}$. For $H_{211}$, let $A_{1,k}=A\cup \{k\}$, then by the definition of $S_{A}$ and $\bar{V}_{A}$, we have
	\begin{align}\label{eq-yy-03}
    \begin{aligned}
        		\Big|\frac{S_{A}}{\bar{V}_{A}}-\frac{S_{A_{1,k}}}{V_{A_{1,k}}}\Big|&\leq \Big|\frac{S_{A}}{\bar{V}_{A}}-\frac{S_{A_{1,k}}}{\bar{V}_{A}}\Big|+\Big|\frac{S_{A_{1,k}}}{\bar{V}_{A}}-\frac{S_{A_{1,k}}}{V_{A_{1,k}}}\Big|\\
		&\leq 2\sum_{m\in N_{k}}\frac{|X_{m}|}{\sigma}+|S_{A_{1,k}}|\cdot \Big|\frac{1}{\bar{V}_{A}}- \frac{1}{V_{A_{1,k}}} \Big|.
    \end{aligned}
	\end{align}
    By the inequality $|\sqrt{x}-\sqrt{y}|\leq \sqrt{|x-y|}$ for any $x,y\geq 0$, we have
 	\begin{align}\label{eq-l-08}
 		|\psi(x)-\psi(y)|=\Big|\sqrt{\psi^{2}(x)}-\sqrt{\psi^{2}(y)}\Big|\leq \sqrt{|\psi^{2}(x)-\psi^{2}(y)|}\leq \sqrt{|x-y|}.
 	\end{align}
	Note that by the definition of $N_{A}$, we have $N_{A}^{c}\backslash N_{A_{1,k}}^{c}\subset N_{k}$. Combining this with (\ref{eq-l-08}) yields 
	\begin{align}\label{eq-yy-04}
    \begin{aligned}
        		|\bar{V}_{A}-V_{A_{1,k}}|&\leq \sqrt{\Bigg|\sum_{i\in N_{A}^{c}}\sum_{j\in N_{A}^{c}\cap A_{i}} X_{i}X_{j}-\sum_{i\in N_{A_{1,k}}^{c}}\sum_{j\in N_{A_{1,k}}^{c}\cap A_{i}} X_{i}X_{j}\Bigg|}\\
		&\leq \sqrt{\sum_{i\in N_{k}}\sum_{j\in  A_{i}} |X_{i}X_{j}|+\sum_{i\in N_{k}}\sum_{j\in  N_{i}} |X_{i}X_{j}|}=\sigma T_{k}.
    \end{aligned}
	\end{align}
	By (\ref{eq-yy-04}) and the fact that $\bar{V}_{A}\geq \sigma/2$ and $\bar{V}_{A_{1,k}}\geq \sigma/2$, we have
	\begin{align}\label{eq-yy-05}
		|S_{A_{1,k}}|\cdot \Big|\frac{1}{\bar{V}_{A}}- \frac{1}{V_{A_{1,k}}} \Big|\leq |S_{A_{1,k}}|\min\Big\{\frac{4}{\sigma}, \frac{4T_{k}}{\sigma}\Big\}=\frac{4|S_{A_{1,k}}|}{\sigma}\cdot Q_{k}.
	\end{align}
	Moreover, note that $Q_{A}\leq 1$, then
	\begin{align}\label{eq-yy-06}
    \begin{aligned}
        		\frac{c|S_{A}|Q_{A}}{\sigma}&\leq \Big|\frac{c|S_{A}|Q_{A}}{\sigma}-\frac{c|S_{A_{1,k}}|Q_{A}}{\sigma} \Big|+ \frac{c|S_{A_{1,k}}|Q_{A}}{\sigma}\\
		&\leq c\sum_{m\in N_{k}}\frac{|X_{m}|}{\sigma}+ \frac{c|S_{A_{1,k}}|Q_{A}}{\sigma}.
    \end{aligned}
	\end{align}
	In addition, note that $\{k\}\subset A_{1,k}$ and $A\subset A_{1,k}$, then
	\begin{align}\label{eq-yy-006}
		T_{A}\leq  T_{A_{1,k}}\quad \text{and}\quad  T_{k}\leq  T_{A_{1,k}},
	\end{align}
	which further implies
	\begin{align}\label{eq-yy-07}
		Q_{A}=\min\{1, T_{A}\}\leq \min\{1, T_{A_{1,k}} \}=Q_{A_{1,k}}\quad\text{and}\quad  Q_{k}\leq Q_{A_{1,k}}.
	\end{align}
	Combining (\ref{eq-yy-03})--(\ref{eq-yy-07}) yields
	\begin{align}\label{eq-yy-08}
		\Big|\frac{S_{A}}{\bar{V}_{A}}-\frac{S_{A_{1,k}}}{V_{A_{1,k}}}\Big|\leq 2\sum_{m\in N_{k}}\frac{|X_{m}|}{\sigma}+4\frac{|S_{A_{1,k}}|Q_{A_{1,k}}}{\sigma}
	\end{align}
	and
	\begin{align}\label{eq-yy-09}
		\frac{c|S_{A}|Q_{A}}{\sigma}\leq c\sum_{m\in N_{k}}\frac{|X_{m}|}{\sigma}+\frac{c|S_{A_{1,k}}|Q_{A_{1,k}}}{\sigma}.
	\end{align}
	It follows from (\ref{eq-yy-08}) and (\ref{eq-yy-09}) that
	\begin{align*}
        		&\mathbf{1}(S_{A}/\bar{V}_{A} \leq a+U'_k)\\
		&=\mathbf{1}\Big( \frac{S_{A}}{\bar{V}_{A}} \leq a+(c+2)\sum_{l \in B\cup A_{k}}\frac{|X_{l}|}{\sigma}+\frac{c |S_{A}|Q_{A}}{\sigma}\Big)\\
		&\leq \mathbf{1}\Big( \frac{S_{A_{1,k}}}{V_{A_{1,k}}} \leq a+2\sum_{m\in N_{k}}\frac{|X_{m}|}{\sigma}+(c+2)\sum_{l \in B\cup A_{k}}\frac{|X_{l}|}{\sigma}+\frac{c |S_{A}|Q_{A}}{\sigma}+\frac{4|S_{A_{1,k}}| Q_{A_{1,k}}}{\sigma}\Big)\\
		&\leq \mathbf{1}\Big( \frac{S_{A_{1,k}}}{V_{A_{1,k}}} \leq a+(2c+4)\sum_{m \in B\cup A_{k}\cup N_{k}}\frac{|X_{m}|}{\sigma}+\frac{(c+4)|S_{A_{1,k}}| Q_{A_{1,k}}}{\sigma}\Big)\\
		&=\mathbf{1}(S_{A_{1,k}}/V_{A_{1,k}} \leq a+\tilde{U}_k),
	\end{align*}
	where
	\[
	\tilde{U}_k=(2c+4)\sum_{m \in B\cup A_{k}\cup N_{k}}\frac{|X_{m}|}{\sigma}+\frac{(2c+4)|S_{A_{1,k}}| Q_{A_{1,k}}}{\sigma}.
	\]
	Similarly,
	\begin{align*}
		\mathbf{1}(S_{A}/\bar{V}_{A} \geq a-\alpha-U'_k)\geq \mathbf{1}(S_{A_{1,k}}/V_{A_{1,k}} \geq a-\alpha-\tilde{U}_k),
	\end{align*}
	Combining the two inequalities above, we obtain
	\begin{align}\label{eq-yy-12}
		\mathbf{1}(a-\alpha-U'_k\leq  S_{A}/\bar{V}_{A} \leq a+U'_k)\leq \mathbf{1}(a-\alpha-\tilde{U}_k\leq S_{A_{1,k}}/V_{A_{1,k}} \leq a+\tilde{U}_k).
	\end{align}
	For $H_{211}$, it follows from  the definition of $K$ and (\ref{eq-yy-12}) that
	\begin{align}\label{eq-yy-13}
        H_{211}&=\frac{4\kappa^{2}}{3\sigma^{3}\alpha} \sum_{k \in  N_A^c} \mathbb{E}\big\{\xi_A|X_k^{3} |\mathbf{1}(a-\alpha-U_k \leq S_{A}/\bar{V}_{A} \leq a+U_k)\big\}\nonumber\\
		&\leq \frac{4\kappa^{2}}{3\sigma^{3}\alpha} \sum_{k \in  N_A^c} \mathbb{E}\big\{\xi_A|X_k^{3} |\mathbf{1}(a-\alpha-\tilde{U}_k \leq S_{A_{1,k}}/V_{A_{1,k}} \leq a+\tilde{U}_k)\big\}\nonumber\\
		&\leq \frac{4\kappa^{2}K}{3\sigma^{3}\alpha} \sum_{k \in  N_A^c}\|\xi_{A}X_{k}^{3}\|_{4/3}\Big\{\frac{\alpha}{1500}+ (2c+4)\sum_{m\in B\cup A_{k}\cup N_{k}}\frac{\|X_{m}\|_{4}}{\sigma}\\
&\quad+\frac{\lambda(2c+4)\kappa^2(|A|+1)^2 }{\sigma^3} \sum_{i=1}^n\left\|X_i\right\|_{4}^3 \nonumber\\
&\quad+\frac{\lambda(2c+4)\kappa^{1/2}(\kappa+\tau^{1/2})(|A|+1)^{1 / 2} }{\sigma^2}\big(\sum_{i=1}^n\left\|X_i\right\|_4^4\big)^{1/2}\nonumber\\		&\quad+\lambda(2c+4)\Big(\sum_{m\in N_{A}\cup N_{k}}\sum_{l\in A_{m}\cup N_{m}}\frac{\|X_{m}\|_{4}\|X_{l}\|_{4}}{\sigma^{2}} \Big)^{1/2} \Big\}.\nonumber
	\end{align}
	By Young's inequality, we have
	\begin{align}\label{eq-yy-14}
    \begin{aligned}
        		\frac{\kappa^{2}}{\sigma^{4}}\sum_{k\in N_{A}^{c}}\sum_{m\in A_{k}\cup N_{k}}\|X_{k}\|_{4}^{3}\|X_{m}\|_{4}&\leq \frac{\kappa^{2}}{\sigma^{4}}\sum_{k\in N_{A}^{c}}\sum_{m\in A_{k}\cup N_{k}}\big(\frac{3}{4}\|X_{k}\|_{4}^{4}+\frac{1}{4}\|X_{m}\|_{4}^{4}\big)\\
		&\leq \frac{2\kappa^{3}}{\sigma^{4}}\sum_{k\in [n]} \|X_{k}\|_{4}^{4}.
    \end{aligned}
	\end{align}
	Moreover, applying the basic inequality $ab\leq \kappa a^{2}/2+b^{2}/(2\kappa)$ with 
    \[
    a=\|X_{k}\|_{4}^{2}\quad\text{and}\quad
  b= \Big(\sum_{m\in N_{k}}\sum_{l\in A_{m}\cup N_{m}}\|X_{m}\|_{4}\|X_{l}\|_{4} \Big)^{1/2} \]
    yields that
	\begin{align}
  \label{eq-yy-16}
  \begin{aligned}
      		&\kappa^{2}\sum_{k\in N_{A}^{c}}\|X_{k}\|_{4}^{3}\cdot \Big(\sum_{m\in N_{k}}\sum_{l\in A_{m}\cup N_{m}}\|X_{m}\|_{4}\|X_{l}\|_{4} \Big)^{1/2}\\
		&\leq \frac{\kappa^{3}}{2}\sum_{k\in N_{A}^{c}}\|X_{k}\|_{4}^{4}+\frac{\kappa}{2}\sum_{k\in N_{A}^{c}}\sum_{m\in N_{k}}\sum_{l\in A_{m}\cup N_{m}}\|X_{k}\|_{4}^{2}\|X_{m}\|_{4}\|X_{l}\|_{4}\\
		&\leq \frac{\kappa^{3}}{2}\sum_{k\in N_{A}^{c}}\|X_{k}\|_{4}^{4}+\frac{\kappa}{2}\sum_{k\in N_{A}^{c}}\sum_{m\in N_{k}}\sum_{l\in A_{m}\cup N_{m}}\Big(\frac{1}{2}\|X_{k}\|_{4}^{4}+\frac{1}{4}\|X_{m}\|_{4}^{4}+\frac{1}{4}\|X_{l}\|_{4}^{4}\Big)\\
		&\leq 1.5\kappa^{3}\sum_{k\in [n]} \|X_{k}\|_{4}^{4}.
  \end{aligned}
	\end{align}
	It follows from (\ref{eq-yy-13})--(\ref{eq-yy-16}), we have
	\begin{align*}
        		H_{211}&\leq \frac{4\kappa^{2}K}{3\sigma^{3}\alpha} \sum_{k \in  N_A^c}\|\xi_{A}\|_{4/3}\|X_{k}\|_{4}^{3}\Big\{\frac{\alpha}{1500}+6\delta_{1}+24\delta_{2}+9\delta_{3}+6\delta_{4}\Big\}\\
		&\quad+\frac{28\lambda\kappa^{3}K}{\sigma^{4}\alpha}\|\xi_{A}\|_{4/3} \sum_{k \in  [n]}\|X_{k}\|_{4}^{4}.
	\end{align*}
	Recalling that
	\begin{align*}
\alpha=1500\left(\delta_{2}+\delta_{3}\right)\quad\text{and}\quad c\geq 1,\ |A|\geq 1,
	\end{align*}
	then
	\begin{align*}
        		H_{211}&\leq \frac{K \|\xi_{A}\|_{4/3} }{1125}\cdot \big( \delta_{2}+6\delta_{1}+24\delta_{2}+9\delta_{3}+6\delta_{4}\big)+\frac{7K \|\xi_{A}\|_{4/3} }{375}\delta_{3}\\
		&\leq \frac{2K \|\xi_{A}\|_{4/3} }{75}\cdot \big(\delta_{1}+\delta_{2}+\delta_{3}+\delta_{4}\big).
	\end{align*}
	Similarly, we have
	\begin{align*}
		H_{212}\leq \frac{4K \|\xi_{A}\|_{4/3} }{75}\cdot \big(\delta_{1}+\delta_{2}+\delta_{3}+\delta_{4}\big),
	\end{align*}
	and hence
	\begin{align}\label{eq-yy-20}
		H_{21}\leq \frac{2K \|\xi_{A}\|_{4/3} }{25}\cdot \big(\delta_{1}+\delta_{2}+\delta_{3}+\delta_{4}\big).
	\end{align}
	With similar argument as that leading to (\ref{eq-yy-20}), we can obtain the same upper bound for $H_{22}$. Then,
	\begin{align}\label{eq-yy-21}
		H_{2}\leq \frac{4K \|\xi_{A}\|_{4/3} }{25}\cdot \big(\delta_{1}+\delta_{2}+\delta_{3}+\delta_{4}\big).
	\end{align}
	It follows from (\ref{eq-pp-08}), (\ref{eq-pp-10}), (\ref{eq-pp-11}), (\ref{eq-kk-32}) and (\ref{eq-yy-21}) that
	\begin{align*}
      		&0.38\mathbb{E}\left\{\xi_A \mathbf{1}\left(\eta_{A,B} \leq V \leq \zeta_{A,B}\right)\right\}\\
		&\quad\quad\leq \|\xi_{A}\|_{4/3}\cdot\big(1.22(b-a+\alpha)+4\delta_{1}+96\delta_{2}+23\delta_{3}+28\delta_{4}\big)\\
		&\quad\quad\quad+ 0.16K \|\xi_{A}\|_{4/3} \cdot \big(\delta_{1}+\delta_{2}+\delta_{3}+\delta_{4}\big)\\
		&\quad\quad\leq 1926\|\xi_{A}\|_{4/3} \cdot\big((b-a)/1500+\delta_{1}+\delta_{2}+\delta_{3}+\delta_{4}\big)\\
&\quad\quad\quad+0.16K\|\xi_{A}\|_{4/3}\cdot\big((b-a)/1500+\delta_{1}+\delta_{2}+\delta_{3}+\delta_{4}\big),  
	\end{align*}
	which implies
	\begin{align*}
		0.38K\leq 0.16K+1926.
	\end{align*}
   By solving the above inequality, we obtain
	\[
	K\leq 8755.
	\]
	This completes the proof.
\end{proof}

\section{Proofs of Main Results}\label{section-proof-of-main-results}
In this section, we present the proofs of main results. 
\subsection{Proof of Theorem \textup{\ref{thm-main1-1}}} 
In this subsection, we provide the proof of Theorem  \textup{\ref{thm-main1-1}}.   
We begin by establishing a general Berry–Esseen bound for 
$W_{1}.$
\begin{theo}\label{thm-main1}  Under \textup{(LD1)} and \textup{(LD2)}, we have
	\begin{align}\label{eq-thm-01}
		\sup _{z \in \mathbb{R}}|\mathbb{P}(W_{1} \leq z)-\Phi(z)|\leq C(\beta_{1}+\beta_{2}+\beta_{3}),
	\end{align}
	where
	\begin{align*}
		\beta_{1}&= \frac{1}{\sigma^{3}}\sum_{i\in [n]} |A_{i}|^2\left\|X_i\right\|_{4}^3+\frac{1}{\sigma^{3}}\sum_{i\in [n]}\sum_{j\in A_{i}} |A_{i}|\left\|X_j\right\|_{4}^3,\nonumber\\
		\beta_{2}^{2}&=\frac{1}{\sigma^{4}}\sum_{i \in [n]}\sum_{j\in A_{i}}\sum_{k\in A_{ij}}\sum_{l\in A_{k}\cup N_{k}}\|X_i\|_{4}\|X_j\|_{4}\|X_{k}\|_{4}\|X_{l}\|_{4}\nonumber\\
		&\quad+\frac{1}{\sigma^{4}}\sum_{i\in [n]}\sum_{j\in A_{i}\cup N_{i}} |A_{i}|^2\left\|X_i\right\|_{4}^3\|X_{j}\|_{4}+\frac{1}{\sigma^{4}}\sum_{i\in [n]}\sum_{j\in A_{i}}\sum_{k\in A_{i}\cup N_{j}\cup A_{j}} |A_{i}|\left\|X_j\right\|_{4}^3\|X_{k}\|_{4},\nonumber\\
		\beta_{3}^{2}&= \frac{1}{\sigma^{5}}\sum_{i\in [n]} |A_{i}|^2\sum_{j\in A_{i}\cup N_{i}}\sum_{k\in N_{j}}\left\|X_i\right\|_{4}^3\|X_{j}\|_{4}\|X_{k}\|_{4}\nonumber\\
        &\quad+\frac{1}{\sigma^{5}}\sum_{i\in [n]}\sum_{j\in A_{i}}\sum_{k\in A_{i}\cup N_{j}}\sum_{l\in N_{k}} |A_{i}|\left\|X_j\right\|_{4}^3\|X_{k}\|_{4}\|X_{l}\|_{4}\nonumber\\
		&\quad+\frac{1}{\sigma^{5}}\sum_{i\in [n]}\sum_{(j,k)\in D_{i}} |A_{i}|^2\left\|X_i\right\|_{4}^3\|X_{j}\|_{4}\|X_{k}\|_{4}\nonumber\\
        &\quad+\frac{1}{\sigma^{5}}\sum_{i\in [n]}\sum_{j\in A_{i}}\sum_{(k,l)\in D_{j}} |A_{i}|\|X_j\|_{4}^3\|X_{k}\|_{4}\|X_{l}\|_{4}.
	\end{align*}
\end{theo}
\begin{proof}
  Without loss of generality, we  assume that 
\begin{align}\label{eq-y-0.1}
	\beta_{2}\leq 1/8,
\end{align}
Otherwise (\ref{eq-thm-01}) naturally holds.  
For any $z\in \mathbb{R}$ and $\varepsilon>0$, let
\begin{align}\label{eq-defininition-of-h}
	h_{z, \varepsilon}(w)= \begin{cases}1, & \text { if } w \leq z, \\  1+\varepsilon^{-1}(z-w), & \text { if } z<w, \leq z+\varepsilon\\ 0, & \text { if } w>z+\varepsilon. \end{cases}
\end{align}
Note that \begin{align}\label{eq-ppa-01}
  \sup_{z\in \R}|\p(W_{1}\leq z)-\Phi(z)|  \leq 0.4\varepsilon+\sup_{z\in \R}|\e h_{z, \varepsilon}(W_{1})-\e h_{z,\varepsilon}(Z)|,
\end{align}
we focus on $|\e h_{z, \varepsilon}(W_{1})-\e h_{z,\varepsilon}(Z)|.$
Let $f_{z,\varepsilon}$ be the solution of  the following Stein equation
\begin{align}\label{eq-y-001}
	f^{\prime}(w)-w f(w)=h_{z, \varepsilon}(w)-\e \{h_{z,\varepsilon}(Z)\},
\end{align}
where $Z\sim N(0,1).$
It can be shown that (see, e.g. \cite{ChenandShao})
\begin{align}\label{eq-y-01}
	0\leq f_{z,\varepsilon}(w)\leq 1,  \quad|f_{z, \varepsilon}'(w)|\leq 1
\end{align}
and
\begin{align}\label{eq-y-02}
	\left|f_{z, \varepsilon}^{\prime}(w+t)-f_{z, \varepsilon}^{\prime}(w)\right| \leq(|w|+1)|t|+\frac{1}{\varepsilon} \int_{t \wedge 0}^{t \vee 0} \mathbf{1}[z \leq w+u \leq z+\varepsilon] d u.
\end{align}
For any $i\in [n]$,  let $Y_i=\sum_{j \in A_i} X_i$, $S^{(i)}=S-Y_i$ and 
\begin{align}\label{eq-definition-of-K}
\begin{aligned}
        \hat{K}_{i}(t)&= X_{i}(\mathbf{1}(-Y_{i}\leq t\leq 0)-\mathbf{1}(0\leq t\leq -Y_{i})),\\
        \hat{K}_{i}&=\int_{-\infty}^{+\infty}\hat{K}_{i}(t)dt,\quad \hat{K}=\sum_{i\in [n]}\hat{K}_{i}.
\end{aligned}
\end{align}
By (LD1), $X_i$ is independent of $S^{(i)}$, it follows that
\begin{align*}
    &\mathbb{E} W_{1} f_{z, \varepsilon}(W_{1})\\
	&\quad =\frac{1}{\sigma}\sum_{i \in[n]} \mathbb{E}\big\{X_i\big(f_{z, \varepsilon}(S/\sigma)-f_{z, \varepsilon}(S^{(i)}/\sigma)\big)\big\}\\
	&\quad=\frac{1}{\sigma^{2}}\sum_{i \in[n]} \mathbb{E}\Big\{X_i \int_{-Y_i}^0 f_{z, \varepsilon}^{\prime}((S+t)/\sigma) dt\Big\} \\
	&\quad=\frac{1}{\sigma^{2}}\sum_{i \in[n]} \mathbb{E}\left\{X_i Y_i f_{z, \varepsilon}^{\prime}(W_{1})\right\}+\frac{1}{\sigma^{2}}\sum_{i \in[n]} \mathbb{E}\Big\{X_i \int_{-Y_i}^0\left[f_{z, \varepsilon}^{\prime}((S+t)/\sigma)-f_{z, \varepsilon}^{\prime}(S/\sigma)\right] d t\Big\}\\
	&\quad=\frac{1}{\sigma^{2}} \mathbb{E}\big\{\hat{K} f_{z, \varepsilon}^{\prime}(W_{1})\big\}+\frac{1}{\sigma^{2}} \mathbb{E}\Big\{\int_{-\infty}^{+\infty}\left[f_{z, \varepsilon}^{\prime}((S+t)/\sigma)-f_{z, \varepsilon}^{\prime}(S/\sigma)\right] \hat{K}(t) d t\Big\}.
\end{align*}
Therefore,
\begin{align}\label{eq-y-04}
	&|\e h_{z, \varepsilon}(W_{1})-\e h_{z,\varepsilon}(Z)|
	\leq I_{1}+I_{2}+I_{3},
\end{align}
where 
\begin{align*}
  I_{1}&=\frac{1}{\sigma^{2}}\e|\hat{K}-\sigma^{2}|,\qquad I_{2}=\frac{1}{\sigma^{3}}\sum_{i\in [n]}\e\{| X_{i}Y_{i}^{2}|(1+|W_{1}|)\},\\
I_{3}&=\frac{1}{\varepsilon\sigma^{3}}\sum_{i\in [n]}\e\Big\{\int_{-\infty}^{+\infty}\int_{t \wedge 0}^{t \vee 0} \mathbf{1}(z \leq (S+u)/\sigma \leq z+\varepsilon)\cdot |\hat{K}_{i}(t)| d u d t\Big\}.
\end{align*}
In what follows,  we provide upper bounds for $I_{1}$, $I_{2}$ and $I_{3}$ respectively. For $I_{1}$, by Lemma \ref{lem-XiYi-moment}, we have
\begin{align}\label{eq-y-05}
	I_{1}=\frac{1}{\sigma^{2}}\e|\hat{K}-\sigma^{2}|&=\frac{1}{\sigma^{2}}\e \Big|\sum_{i\in [n]}\sum_{j\in  A_{i}}(X_{i}X_{j}-\e X_{i}X_{j}) \Big|\leq C\beta_{2}.
\end{align}
For $I_{2}$, we have
\begin{align}\label{eq-y-06}
	I_{2}&\leq \frac{1}{\sigma^{3}}\sum_{i\in [n]}|A_{i}|\sum_{j\in A_{i}}\e\{| X_{i}X_{j}^{2}|(1+|W_{1}|)\}
	\leq I_{21}+I_{22},
\end{align}
where 
\begin{align*}
    I_{21}&=\frac{1}{3\sigma^{3}}\sum_{i\in [n]}|A_{i}|^{2}\e\{| X_{i}^{3}|(1+|W_{1}|)\},\quad I_{22}=\frac{2}{3\sigma^{3}}\sum_{i\in [n]}|A_{i}|\sum_{j\in A_{i}}\e\{| X_{j}^{3}|(1+|W_{1}|)\}.
\end{align*}
For $I_{21}$, let $W_{1}^{(i)}=S^{(i)}/\sigma$, then 
\begin{align*}
\begin{aligned}
        	I_{21}&\leq \frac{1}{3\sigma^{3}}\sum_{i\in [n]}|A_{i}|^{2}\e\big\{| X_{i}^{3}|(1+|W_{1}^{(i)}|)\big\}+\frac{1}{3\sigma^{4}}\sum_{i\in [n]}|A_{i}|^{2}\e| X_{i}^{3}Y_{i}|\\
	&\leq \frac{1}{3\sigma^{3}}\sum_{i\in [n]}|A_{i}|^{2}\e| X_{i}|^{3}\cdot(1+\e|W_{1}^{(i)}|)+\frac{1}{3\sigma^{4}}\sum_{i\in [n]}\sum_{j\in A_{i}}|A_{i}|^{2}\| X_{i}\|_{4}^{3}\|X_{j}\|_{4}.
\end{aligned}
\end{align*}
Noting that by Minkowski's inequality, we have
\begin{align*}
	\|S^{(i)}\|_{2}\leq \|S\|_{2}+\|S-S^{(i)}\|_{2}\leq \sigma+\sum_{j\in A_{i}}\|X_{j}\|_{2},
\end{align*}
which further implies 
\begin{align}\label{eq-y-09}
	I_{21}\leq \frac{C}{\sigma^{3}}\sum_{i\in [n]}|A_{i}|^{2}\e|X_{i}|^{3}+ \frac{C}{\sigma^{4}}\sum_{i\in [n]}\sum_{j\in A_{i}}|A_{i}|^{2}\| X_{i}\|_{4}^{3}\|X_{j}\|_{4}.
\end{align}
With a similar argument, we can obtain the following upper bound for  $I_{22}$:
\begin{align}\label{eq-y-09.1}
	I_{22}\leq \frac{C}{\sigma^{3}}\sum_{i\in [n]}\sum_{j\in A_{i}} |A_{i}|\left\|X_j\right\|_{4}^3+ \frac{C}{\sigma^{4}}\sum_{i\in [n]}\sum_{j\in A_{i}}\sum_{k\in A_{j}} |A_{i}|\|X_j\|_{4}^3\|X_{k}\|_{4}.
\end{align} 
Combining  (\ref{eq-y-0.1}), (\ref{eq-y-09})  and (\ref{eq-y-09.1}) yields
\begin{align}\label{eq-y-10}
	I_{2}\leq C\beta_{1}+C\beta_{2}^{2}\leq C\beta_{1}+C\beta_{2}.
\end{align}
For $I_{3}$, as $|s|\leq |Y_{i}|$, then
\begin{align}\label{eq-y-11}
\begin{aligned}
    	I_{3}&\leq \frac{1}{2\varepsilon\sigma^{3}}\sum_{i\in [n]} \e\big\{ |X_{i}Y_{i}^{2}|\mathbf{1}(z-|Y_{i}|/\sigma \leq S/\sigma \leq z+|Y_{i}|/\sigma+\varepsilon)\big\}\\
	&\leq \frac{1}{2\varepsilon\sigma^{3}}\sum_{i\in [n]}\sum_{j\in A_{i}} |A_{i}|\e\big\{ |X_{i}X_{j}^{2}|\mathbf{1}(z-|Y_{i}|/\sigma \leq S/\sigma \leq z+|Y_{i}|/\sigma+\varepsilon)\big\}\\
	&\leq I_{31}+I_{32},
\end{aligned}
\end{align}
where 
\begin{align*}
    I_{31}&=\frac{1}{6\varepsilon\sigma^{3}}\sum_{i\in [n]} |A_{i}|^{2}\e\big\{ |X_{i}|^{3}\mathbf{1}(z-|Y_{i}|/\sigma \leq S/\sigma \leq z+|Y_{i}|/\sigma+\varepsilon)\big\},\nonumber\\
    I_{32}&=\frac{1}{3\varepsilon\sigma^{3}}\sum_{i\in [n]}\sum_{j\in A_{i}} |A_{i}|\e\big\{ |X_{j}|^{3}\mathbf{1}(z-|Y_{i}|/\sigma \leq S/\sigma \leq z+|Y_{i}|/\sigma+\varepsilon)\big\}.
\end{align*}
For any $i\in [n]$, let $$S_{i}=S-\sum_{k\in N_{i}}X_{k}.$$
For $I_{31}$, choosing 
\[
\varepsilon=\beta_{2}+\beta_{3},
\] then by Proposition \ref{prop:1}, we have
\begin{align*}
    I_{31}&\leq \frac{1}{6\varepsilon\sigma^{3}}\sum_{i\in [n]}|A_{i}|^{2}\e\Big\{ |X_{i}|^{3}\mathbf{1}\Big(z-2\sum_{m\in A_{i}\cup N_{i}}\frac{|X_{m}|}{\sigma} \leq \frac{S_{i}}{\sigma} \leq z+2\sum_{m\in A_{i}\cup N_{i}}\frac{|X_{m}|}{\sigma}+\varepsilon\Big)\Big\}\\
	& \leq\frac{C}{\varepsilon\sigma^{3}}\sum_{i\in [n]}|A_{i}|^{2}\|X_{i}\|_{4}^{3}\cdot \Big(\varepsilon+\frac{1}{\sigma}\sum_{j\in N_{i}}\|X_{j}\|_{4}+\frac{1}{\sigma}\sum_{m\in A_{i}\cup N_{i}}\|X_{m}\|_{4}\\
    &\qquad\qquad+\frac{1}{\sigma^{2}}\sum_{m\in A_{i}\cup N_{i}}\sum_{k\in N_{m}}\|X_{k}\|_{4}\|X_{m}\|_{4}+\sum_{(k,l)\in D_{i}}\|X_{k}\|_{4}\|X_{l}\|_{4}+\beta_{1}+\beta_{2}+\beta_{3}\Big)\\
	&\leq C\beta_{1}+C(\beta_{2}^{2}+\beta_{3}^{2})/\varepsilon\\
	&\leq C(\beta_{1}+\beta_{2}+\beta_{3}).
\end{align*}
Similarly, for $I_{32}$, we have
\begin{align*}
	I_{32}&\leq \frac{1}{3\varepsilon\sigma^{3}}\sum_{i\in [n]}\sum_{j\in A_{i}}|A_{i}|\e\Big\{ |X_{j}|^{3}\mathbf{1}\Big(z-2\sum_{m\in A_{i}\cup N_{j}}|X_{m}| \leq\frac{S_{j}}{\sigma}\leq z+2\sum_{m\in A_{i}\cup N_{j}}|X_{m}|+\varepsilon\Big)\Big\}\nonumber\\
	&\leq C(\beta_{1}+\beta_{2}+\beta_{3}),
\end{align*}
and hence 
\begin{align}\label{eq-y-13}
    I_{3}\leq C(\beta_{1}+\beta_{2}+\beta_{3}).
\end{align}
We complete the proof of Theorem \ref{thm-main1-1}
by (\ref{eq-ppa-01}) (\ref{eq-y-04}), (\ref{eq-y-05}), (\ref{eq-y-11}), (\ref{eq-y-10}) and (\ref{eq-y-13}).
\end{proof}
 
 \begin{proof}[Proof of Theorem \ref{thm-main1-1}.]
    Observe that 
	\begin{align*}
		\beta_{1}\leq \frac{2\kappa^{2}}{\sigma^{3}}\sum_{i\in [n]}\|X_{i}\|_{4}^{3}\quad \text{and}\quad  \beta_{2}\leq \frac{C\kappa^{1/2}(\kappa+\tau^{1/2})}{\sigma^{2}}\Big(\sum_{i\in [n]}\|X_{i}\|_{4}^{4}\Big)^{1/2}.
	\end{align*}
	As for $\beta_{3}$, note that for any $i\in [n]$, by Cauchy's inequality, we have 
	\begin{align}\label{eq-pr-03}
    \begin{aligned}
        		\sum_{j\in A_{i}}\sum_{k\in N_{j}}\|X_{j}\|_{4}\|X_{k}\|_{4}&\leq C \Big(\sum_{j\in A_{i}}\sum_{k\in N_{j}}(\|X_{j}\|_{4}^{2}+\|X_{j}\|_{4}^{2})^{2}\Big)^{1/2}\cdot \Big(\sum_{j\in A_{i}}\sum_{k\in N_{j}}1\Big)^{1/2}\\
		& \leq C\Big(\kappa^{3}\sum_{j\in [n]}\|X_{j}\|_{4}^{4}\Big)^{1/2}.
    \end{aligned}
	\end{align}
	Similarly, we have
	\begin{align}\label{eq-pr-03.1}
		\sum_{j\in N_{i}}\sum_{k\in N_{j}}\|X_{j}\|_{4}\|X_{k}\|_{4}
		& \leq C\Big(\kappa^{3}\sum_{j\in [n]}\|X_{j}\|_{4}^{4}\Big)^{1/2}.
	\end{align}
	Moreover, by Cauchy's inequality, we have
	\begin{align}\label{eq-pr-04}
    \begin{aligned}
        		\sum_{(k,l)\in D_{i}}\|X_{k}\|_{4}\|X_{l}\|_{4}&\leq \frac{1}{2}\sum_{(k,l)\in D_{i}}(\|X_{k}\|_{4}^{2}+\|X_{l}\|_{4}^{2})\\
		&\leq \frac{1}{2}\Big(\sum_{(k,l)\in D_{i}}(\|X_{k}\|_{4}^{2}+\|X_{l}\|_{4}^{2})^{2}\Big)^{1/2}\Big(\sum_{(k,l)\in D_{i}}1\Big)^{1/2}\\
		&\leq C\Big(\kappa\tau\sum_{k\in [n]}\|X_{k}\|_{4}^{4}\Big)^{1/2}.
    \end{aligned}
	\end{align}
	Combining (\ref{eq-pr-03})--(\ref{eq-pr-04}) yields that
	\begin{align*}
		\beta_{3}\leq \frac{C\kappa^{2}}{\sigma^{3}}\sum_{i\in [n]}\|X_{i}\|_{4}^{3}+\frac{C\kappa^{1/2}(\kappa+\tau^{1/2})}{\sigma^{2}}\Big(\sum_{i\in [n]}\|X_{i}\|_{4}^{4}\Big)^{1/2}.
	\end{align*}
	This proves Theorem \ref{thm-main1-1}.
 \end{proof}

\subsection{Proof of Theorem \textup{\ref{thm-main2}}}
 In this subsection, we prove  Theorem \textup{\ref{thm-main2}} by Stein's method and the concentration inequalities approach.  Let
 \begin{align*}
 	\bar{V}=\psi\big(\sum_{i\in [n]}X_{i}Y_{i}\big)\quad \text{and}\quad \bar{W}_{2}=\sum_{i\in [n]}X_{i}/\bar{V},
 \end{align*}
 where $\psi(x)=((x\vee (\sigma^{2}/4))\wedge (2\sigma^{2}))^{1/2}.$ We begin by introducing a proposition which provides a Berry--Esseen bound for $\bar{W}_{2}.$
 \begin{prop}\label{prop-BE-bar-W2} Under \textup{(LD1)} and \textup{(LD2)}, we have
 	\begin{align*}
 		\sup _{z \in \mathbb{R}}|\mathbb{P}(\bar{W}_{2} \leq z)-\Phi(z)| \leq C \lambda\Big\{\frac{\kappa^{2}}{\sigma^{3}}\sum_{i \in[n]}\|X_{i}\|_{4}^{3}+\frac{\kappa^{1 / 2}(\kappa+\tau^{1/2})}{\sigma^{2}}\big(\sum_{i \in[n]} \|X_i\|_{4}^4\big)^{1 / 2}\Big\},
 	\end{align*}
 	where
 	\[
 	\lambda=\frac{\kappa}{\sigma^{2}}\sum_{i\in [n]}\|X_{i}\|_{2}^{2}.
 	\]
 \end{prop}
 \begin{proof} Without loss of generality, assume that \[\kappa^{2}\sum_{i\in [n]}\e|X_{i}|^{3}\leq \sigma^{3}/500\quad \text{and} \quad \kappa^{1 / 2}(\kappa+\tau^{1/2})\big(\sum_{i \in[n]} \|X_i\|_{4}^4\big)^{1 / 2}\leq \sigma^{2}/500,\] otherwise the inequality is trivial. Let $f_{z,\varepsilon}$ be the solution of Stein equation (\ref{eq-y-001}) with
 	\[
 	\varepsilon=\frac{\lambda\kappa^{2}}{\sigma^3} \sum_{i \in [n]} \|X_{i}\|_{4}^{3}+\frac{\lambda\kappa^{1/2}(\kappa+\tau^{1/2})}{\sigma^{2}}\big(\sum_{i\in [n]}\|X_{i}\|_{4}^{4}  \big)^{1/2}.
 	\]
 Observe that
 	\begin{align*}
         		&\mathbb{E}\{\bar{W}_{2} f_{z,\varepsilon}(\bar{W}_{2})\}\\
 		&\quad=\sum_{i \in [n]} \mathbb{E}\Big\{\frac{X_{i}}{\bar{V}}\Big(f_{z,\varepsilon}(\bar{W}_{2})-f_{z,\varepsilon}\Big(\frac{S-Y_{i}}{\bar{V}}\Big)\Big)\Big\}+\sum_{i \in [n]} \mathbb{E}\Big\{\frac{X_{i}}{\bar{V}} f_{z,\varepsilon}\Big(\frac{S-Y_{i}}{\bar{V}}\Big)\Big\} \\
 		&\quad=\mathbb{E}\Big\{\frac{1}{\bar{V}^2} \int_{-\infty}^{\infty} f_{z,\varepsilon}^{\prime}\Big(\frac{S+u}{\bar{V}}\Big) \hat{K}(u) d u\Big\}+\sum_{i \in [n]} \mathbb{E}\Big\{\frac{X_{i}}{\bar{V}} f_{z,\varepsilon}\Big(\frac{S-Y_{i}}{\bar{V}}\Big)\Big\},
 	\end{align*}
 	where $\hat{K}(u)$ is defined in (\ref{eq-definition-of-K}).
 	With similar arguments as that leading to (\ref{eq-y-04}), we have
 	\begin{align}\label{eq-l-03}
 	 |\e h_{z, \varepsilon}(\bar{W}_{2})-\e h_{z,\varepsilon}(Z)|\leq  |R_1|+|R_2|+|R_3|+|R_4|,
 	\end{align}
 	where
 	\begin{align*}
 		R_1&=\mathbb{E}\Big|1-\frac{1}{\bar{V}^2} \sum_{i \in [n]} X_{i} Y_{i}\Big|, \\
 		R_2&=\frac{8}{\sigma^3} \mathbb{E}\left\{(|\bar{W}_{2}|+1) \int_{-\infty}^{\infty}|t \hat{K}(t)| d t\right\}, \\
 		R_3&=\frac{8}{\varepsilon\sigma^3} \sum_{i \in [n]} \mathbb{E}\left\{\left|X_{i} Y_{i}^{2}\right| \mathbf{1}\left(z-\left|Y_{i}\right| / \bar{V} \leq \bar{W}_{2} \leq z+\varepsilon+\left|Y_{i}\right| / \bar{V}\right)\right\},\\
 		R_4&=\sum_{i \in [n]} \Big|\mathbb{E}\Big\{\frac{X_{i}}{\bar{V}} f_{z,\varepsilon}\Big(\frac{S-Y_{i}}{\bar{V}}\Big)\Big\}\Big|.
 	\end{align*}
 	Below we give upper bounds for $R_{1},R_{2}, R_{2}$ and $R_{4}$ in turn.
 	For $R_1$, we have
 	\begin{align*}
 		R_1 \leq & R_{11}+R_{12},
 	\end{align*}
    where 
    \begin{align*}
        R_{11}&=\mathbb{E}\Big\{\Big|1-\frac{4}{\sigma^2} \sum_{i \in [n]} X_{i} Y_{i}\Big| \mathbf{1}\Big(\sum_{i \in [n]} X_i Y_i\leq \sigma^2 / 4\Big)\Big\},\\
        R_{12}&=\mathbb{E}\Big\{\Big|1-\frac{1}{2\sigma^2} \sum_{i \in [n]} X_{i} Y_{i}\Big| \mathbf{1}\Big(\sum_{i \in [n]} X_i Y_i\geq 2\sigma^2 \Big)\Big\}.
    \end{align*}
    For $R_{11}$, by Lemma \ref{lem-XiYi-moment} and Chebyshev's inequality, we have
    \begin{align*}
                R_{11}&\leq \frac{4}{\sigma^{2}}\e\Big\{\Big(\Big|\sum_{\in [n]}X_{i}Y_{i}-\sigma^{2}\Big|+\frac{3\sigma^{2}}{4}\Big)\mathbf{1}\Big(\Big|\sum_{i \in [n]} X_i Y_i-\sigma^{2}\Big|\geq 3\sigma^2 / 4\Big)\Big\}\\
       &\leq \frac{C}{\sigma^{2}}\e\Big|\sum_{i \in [n]} (X_{i}Y_{i}-\e X_{i}Y_{i})\Big|\\
       &\leq \frac{C\kappa^{1/2}(\kappa+\tau)}{\sigma^{2}}\Big(\sum_{i\in [n]}\|X_{i}\|_{4}^{4}\Big)^{1/2}.
    \end{align*}
   A similar argument holds for $R_{12}$, and hence 
   \begin{align}\label{eq-l-04}
       R_{1}\leq \frac{C\kappa^{1/2}(\kappa+\tau)}{\sigma^{2}}\Big(\sum_{i\in [n]}\|X_{i}\|_{4}^{4}\Big)^{1/2}.
   \end{align}
 	For $R_{2}$, note that $\bar{V}\geq \sigma/2$. Combining this with (\ref{eq-y-06}), (\ref{eq-y-09}) and (\ref{eq-y-09.1}) implies that
 	\begin{align}\label{eq-l-05}
    \begin{aligned}
        R_2 &\leq \frac{8}{\sigma^3} \sum_{i \in [n]} \mathbb{E}\left\{(|\bar{W}_{2}|+1)\left|X_{i} Y_{i}^2\right|\right\}\\
        & \leq \frac{8}{\sigma^3} \sum_{i \in [n]} \mathbb{E}\Big\{\big(\frac{2}{\sigma}|S|+1\big)\left|X_{i} Y_{i}^2\right|\Big\}\\
 		&\leq \frac{C\kappa^{2}}{\sigma^{3}}\sum_{i\in [n]}\e|X_{i}|^{3}+ \frac{C\kappa^{3}}{\sigma^{4}}\sum_{i\in [n]}\e|X_{i}|^{4}.
    \end{aligned}	
 	\end{align}
 	For $R_{3}$, by Young's inequality, we have
 	\begin{align}\label{eq-l-06}
 		R_{3}&\leq R_{31}+R_{32},
 	\end{align}
    where 
    \begin{align*}
       R_{31}&=\frac{8\kappa^{2}}{3\varepsilon\sigma^3} \sum_{i \in [n]} \mathbb{E}\left\{|X_{i}|^{3} \mathbf{1}\left(z-\left|Y_{i}\right| / \bar{V} \leq \bar{W}_{2} \leq z+\varepsilon+\left|Y_{i}\right| / \bar{V}\right)\right\},\\
       R_{32}&=\frac{16\kappa}{3\varepsilon\sigma^3} \sum_{i \in [n]}\sum_{j\in A_{i}} \mathbb{E}\left\{|X_{j}|^{3} \mathbf{1}\left(z-\left|Y_{i}\right| / \bar{V} \leq \bar{W}_{2} \leq z+\varepsilon+\left|Y_{i}\right| / \bar{V}\right)\right\}.
    \end{align*}
 	For $R_{31}$, define
 	\[
 	\bar{V}_{i}=\psi\Big(\sum_{k\in N_{i}^{c}}\sum_{l\in A_{k}\cap N_{i}^{c}} X_{k}X_{l}\Big)\quad\text{and}\quad  S_{i}=\sum_{k\in N_{i}^{c}}X_{k},
 	\]
 	then
 	\begin{align}\label{eq-l-07}
 		\Big|\frac{S}{\bar{V}}-\frac{S_{i}}{\bar{V}_{i}}\Big|\leq \Big|\frac{S}{\bar{V}}-\frac{S_{i}}{\bar{V}}\Big|+\Big|\frac{S_{i}}{\bar{V}}-\frac{S_{i}}{\bar{V}_{i}}\Big|
 		\leq \frac{2}{\sigma}\sum_{m\in N_{i}}|X_{m}|+|S_{i}|\cdot \Big|\frac{1}{\bar{V}}- \frac{1}{\bar{V}_{i}} \Big|.
 	\end{align}
  It follows from the definition of $\bar{V}$, $\bar{V}_{i}$ and (\ref{eq-l-08}) that
 	\begin{align}\label{eq-l-09}
         		|\bar{V}-\bar{V}_{i}|&=\Big|\psi\Big(\sum_{k\in [n]}\sum_{l\in A_{k}}X_{k}X_{l}\Big)-\psi\Big(\sum_{k\in N_{i}^{c}}\sum_{l\in A_{k}\cap N_{i}^{c}}X_{k}X_{l}\Big)\Big|\nonumber\\
 		&\leq \sqrt{\Big|\sum_{k\in [n]}\sum_{l\in A_{k}}X_{k}X_{l}- \sum_{k\in N_{i}^{c}}\sum_{l\in A_{k}\cap N_{i}^{c}}X_{k}X_{l} \Big|}\\
 		&\leq \sqrt{\sum_{k\in N_{i}}\sum_{l\in A_{k}}|X_{k}X_{l}|+\sum_{k\in N_{i}}\sum_{l\in N_{l}}|X_{k}X_{l}|}=\sigma T_{i}.\nonumber
 	\end{align}
 	Moreover, recall that $\bar{V}\geq \sigma/2$ and $\bar{V}_{i}\geq \sigma/2$, then
 	\begin{align}\label{eq-l-10}
 		|S_{i}|\cdot \Big|\frac{1}{\bar{V}}- \frac{1}{\bar{V}_{i}} \Big|\leq \frac{4|S_{i}|}{\sigma}.
 	\end{align}
 	Combining (\ref{eq-l-07}), (\ref{eq-l-09}) and (\ref{eq-l-10}) yields
 	\begin{align}\label{eq-l-11}
 		\Big|\frac{S}{\bar{V}}-\frac{S_{i}}{\bar{V}_{i}}\Big|\leq \sum_{m\in N_{i}}\frac{2|X_{m}|}{\sigma}+\frac{4|S_{i}|}{\sigma}\cdot \min\{1,T_{i}\}=\sum_{m\in N_{i}}\frac{2|X_{m}|}{\sigma}+\frac{4|S_{i}|Q_{i}}{\sigma},
 	\end{align}
 	where $Q_{i}=\min\{1,T_{i}\}.$ By (\ref{eq-l-11}), we have
 	\begin{align*}
 		& \mathbf{1}\left(z-\left|Y_{i}\right| / \bar{V} \leq \bar{W}_{2} \leq z+\left|Y_{i}\right| / \bar{V}\right)\nonumber\\
 		&\quad\leq \mathbf{1}\Big(z-\sum_{m\in A_{i}\cup N_{i}}\frac{4|X_{m}|}{\sigma}-\frac{4|S_{i}|Q_{i}}{\sigma}  \leq \frac{S_{i}}{\bar{V}_{i}} \leq z+\varepsilon+\sum_{m\in A_{i}\cup N_{i}}\frac{4|X_{m}|}{\sigma}+\frac{4|S_{i}|Q_{i}}{\sigma} \Big).
 	\end{align*}
 	Combining the foregoing inequality,  
  Proposition \ref{prop:2} and (\ref{eq-yy-14})--(\ref{eq-yy-16}), we obtain
 	\begin{align}\label{eq-l-13}
        |R_{31}|&\leq \frac{C\kappa^{2}}{\varepsilon\sigma^3} \sum_{i \in [n]} \mathbb{E}\Big\{|X_{i}|^{3}\mathbf{1}\Big(z-\sum_{m\in A_{i}\cup N_{i}}\frac{4|X_{m}|}{\sigma}-\frac{4|S_{i}|Q_{i}}{\sigma}  \leq \frac{S_{i}}{\bar{V}_{i}}\nonumber\\
&\qquad\qquad\qquad\qquad\qquad\qquad\leq z+\varepsilon+\sum_{m\in A_{i}\cup N_{i}}\frac{4|X_{m}|}{\sigma}+\frac{4|S_{i}|Q_{i}}{\sigma} \Big)\Big\}\nonumber\\
 		&\leq \frac{C\lambda\kappa^{2}}{\varepsilon\sigma^4} \sum_{i \in [n]}\sum_{m\in A_{i}\cup N_{i}}\|X_{i}\|_{4}^{3}\|X_{m}\|_{4}+\frac{C\lambda\kappa^{2}}{\varepsilon\sigma^3} \sum_{i \in [n]} \|X_{i}\|_{4}^{3}\cdot\frac{\kappa^{2}}{\sigma^{3}}\sum_{i\in [n]}\|X_{i}\|_{4}^{3}\nonumber\\
 	&\quad+\frac{C\lambda\kappa^{2}}{\varepsilon\sigma^3} \sum_{i \in [n]} \|X_{i}\|_{4}^{3}\cdot  \frac{\kappa^{1/2}(\kappa+\tau^{1/2})}{\sigma^{2}}\big(\sum_{i\in [n]}\|X_{i}\|_{4}^{4}  \big)^{1/2} \\
 	&\quad+\frac{C\lambda\kappa^{2}}{\varepsilon\sigma^3} \sum_{i \in [n]} \|X_{i}\|_{4}^{3}\cdot \Big( \sum_{k\in N_{i}}\sum_{l\in A_{k}\cup N_{k}}\frac{\|X_{k}\|_{4}\|X_{l}\|_{4}}{\sigma^{2}}  \Big)^{1/2}\nonumber\\
       &\quad+\frac{C\lambda\kappa^{2}}{\sigma^3} \sum_{i \in [n]} \|X_{i}\|_{4}^{3}\nonumber\\
 		&\leq \frac{C\lambda\kappa^{2}}{\sigma^3} \sum_{i \in [n]} \|X_{i}\|_{4}^{3}+\frac{C\lambda\kappa^{1/2}(\kappa+\tau^{1/2})}{\sigma^{2}}\big(\sum_{i\in [n]}\|X_{i}\|_{4}^{4}  \big)^{1/2},	\nonumber
 	\end{align}
 	where
 	\[
 	\lambda=\frac{\kappa}{\sigma^{2}}\sum_{i\in [n]}\|X_{i}\|_{2}^{2}.
 	\]
   For $R_{32}$,	with similar arguments as that leading to (\ref{eq-l-13}),  we have
 	\begin{align}\label{eq-l-14}
 		&R_{32}\leq \frac{C\lambda\kappa^{2}}{\varepsilon\sigma^3} \sum_{i \in [n]} \|X_{i}\|_{4}^{3}+\frac{C\lambda\kappa^{1/2}(\kappa+\tau^{1/2})}{\sigma^{2}}\big(\sum_{i\in [n]}\|X_{i}\|_{4}^{4}  \big)^{1/2}.
 	\end{align}
 	Combining (\ref{eq-l-06}), (\ref{eq-l-13})  and (\ref{eq-l-14}) yields
 	\begin{align}\label{eq-l-15}
 		|R_{3}|\leq C\lambda \Big\{\frac{\kappa^{2}}{\sigma^{3}}\sum_{i \in[n]}\|X_{i}\|_{4}^{3}+\frac{\kappa^{1 / 2}(\kappa+\tau^{1/2})}{\sigma^{2}}\big(\sum_{i \in[n]} \|X_i\|_{4}^4\big)^{1 / 2}\Big\}.
 	\end{align}
 	As for $R_{4}$, by Lemma \ref{lem-R4}, we have
 	\begin{align}\label{eq-l-16}
 		|R_{4}|\leq \frac{C\kappa^{2}}{\sigma^{3}}   \sum_{i\in [n]}\e|X_{i}|^{3}+\frac{C\kappa^{3}}{\sigma^{4}}   \sum_{i\in [n]}\e|X_{i}|^{4}.
 	\end{align}
 	We complete the proof of Proposition \ref{prop-BE-bar-W2} by  (\ref{eq-l-03}), (\ref{eq-l-04}), (\ref{eq-l-05}), (\ref{eq-l-15}) and (\ref{eq-l-16}).
 \end{proof}
Now, we are ready to prove Theorem \ref{thm-main2}.
 \begin{proof}[Proof of Theorem \ref{thm-main2}]
      Assume that 
 \begin{align*}
 	\frac{\kappa^{2}}{\sigma^{3}}\sum_{i\in [n]}\|X_{i}\|_{4}^{3}\leq \frac{1}{500},
 \end{align*}
 otherwise (\ref{eq-thm-03}) holds  trivially.   Define
 \[
 \tilde{V}=\psi\big(\sum_{i\in [n]}(X_{i}Y_{i}-\bar{X}\bar{Y})\big)\quad\text{and}\quad \tilde{W}_{2}=S/\tilde{V},
 \]
 then by the triangle inequality, we have
 \begin{align}\label{eq-m-01}
 \begin{aligned}
      	&\sup_{z\in \R}|\p(W_{2}\leq z)-\Phi(z)|\\
        &\quad\leq \sup_{z\in \R}| \p(W_{2}\leq z)-\p(\tilde{W}_{2}\leq z)|+\sup_{z\in \R}| \p(\tilde{W}_{2}\leq z)-\p(\bar{W}_{2}\leq z)|\\
 	&\quad\quad+\sup_{z\in \R}|\p(\bar{W}_{2}\leq z)-\Phi(z)|.
 \end{aligned}
 \end{align}
 For the first term of (\ref{eq-m-01}), we have
 \begin{align}\label{eq-m-02}
 \begin{aligned}
     &\sup_{z\in \R}| \p(W_{2}\leq z)-\p(\tilde{W}_{2}\leq z)|\\
 	&\quad\leq \p\big(\sum_{i\in [n]}(X_{i}Y_{i}-\bar{X}\bar{Y})\leq \sigma^{2}/4\big)+\p\big(\sum_{i\in [n]}(X_{i}Y_{i}-\bar{X}\bar{Y})\geq 2\sigma^{2}\big)\\
 	&\quad\leq \p\big(\sum_{i\in [n]}X_{i}Y_{i}\leq n|\bar{X}\bar{Y}|+\sigma^{2}/4\big)+\p(\sum_{i\in [n]}X_{i}Y_{i}\geq -n|\bar{X}\bar{Y}|+2\sigma^{2})\\
 	&\quad\leq \p\big(\sum_{i\in [n]}X_{i}Y_{i}\leq \sigma^{2}/2\big)+\p\big(\sum_{i\in [n]}X_{i}Y_{i}\geq 7\sigma^{2}/4\big)+2\p(n|\bar{X}\bar{Y}|\geq \sigma^{2}/4).  
 \end{aligned}	
 \end{align}
 Note that $\sigma^{2}=\e S^{2}=\sum_{i\in [n]} X_{i}Y_{i}$, then by Lemma \ref{lem-XiYi-moment}, we have
 \begin{align}\label{eq-m-03}
 \begin{aligned}
      	\p\Big(\sum_{i\in [n]}X_{i}Y_{i}\leq \sigma^{2}/2\Big) &\leq \p\Big( \Big|\sum_{i\in [n]}(X_{i}Y_{i}-\e\{X_{i}Y_{i}\})\Big|\geq \frac{\sigma^{2}}{2}   \Big)\\
 	&\leq \frac{C(\kappa^{3/2}+\kappa^{1/2}\tau^{1/2})}{\sigma^{2}}\cdot \big(\sum_{i\in [n]}\|X_{i}\|_{4}^{4}\big)^{1/2}.
 \end{aligned}
 \end{align}
 Similarly,
 \begin{align}\label{eq-m-04}
 	\p\Big(\sum_{i\in [n]}X_{i}Y_{i}\geq 7\sigma^{2}/4\Big)\leq \frac{C(\kappa^{3/2}+\kappa^{1/2}\tau^{1/2})}{\sigma^{2}}\cdot \Big(\sum_{i\in [n]}\|X_{i}\|_{4}^{4}\Big)^{1/2}.
 \end{align}
 As for $\p(n|\bar{X}\bar{Y}|\geq \sigma^{2}/4)$, by Lemma \ref{lem-fourth-moment-for-S},  we have
 \begin{align}\label{eq-m-05}
 	\p(n|\bar{X}\bar{Y}|\geq \sigma^{2}/4)\leq \frac{4}{\sigma^{2}}\cdot \e|S\bar{Y}|\leq \frac{4\|\bar{Y}\|_{4}}{\sigma}\leq 8\lambda\kappa n^{-1}.
 \end{align}
 On the other hand, note that by H\"{o}lder's inequality, we have 
 \begin{align*}
   \sigma^{3}= \Big(\sum_{i\in [n]}\sum_{j\in A_{i}} \e \{X_{i}X_{j}\}\Big)^{3/2}&\leq  \kappa^{1/2}n^{1/2}\sum_{i\in [n]}\sum_{j\in A_{i}} \e\{|X_{i}|^{3/2}|X_{j}|^{3/2}\}\nonumber\\
   &\leq \kappa^{3/2}n^{1/2}\sum_{i\in [n]} \e|X_{i}|^{3},
 \end{align*}
 which further implies 
 \begin{align}\label{eq-m-07}
 \kappa^{1/2}n^{-1/2}\leq \frac{\kappa^{2}}{\sigma^{3}}\sum_{i\in [n]} \e|X_{i}|^{3}.
 \end{align}
 Combining (\ref{eq-m-02})--(\ref{eq-m-07}), we have
 \begin{align}\label{eq-m-08}
 \begin{aligned}
     &\sup_{z\in \R}| \p(W\leq z)-\p(\tilde{W}_{2}\leq z)|\\
&\quad \leq \frac{C\kappa^{1/2}(\kappa^{1/2}+\tau^{1/2})}{\sigma^{2}}\cdot \big(\sum_{i\in [n]}\|X_{i}\|_{4}^{4}\big)^{1/2}+\frac{C\lambda\kappa^{2}}{\sigma^{3}}\sum_{i\in [n]} \|X_{i}\|_{4}^{3}.
 \end{aligned} 	
 \end{align}
 By (\ref{eq-m-08}) and Proposition \ref{prop-BE-bar-W2}, in order to prove Theorem \ref{thm-main2}, it suffices to provide an upper bound for $\sup_{z\in \R}| \p(\tilde{W}_{2}\leq z)-\p(\bar{W}_{2}\leq z)|.$   Let $\varepsilon_{1}=\kappa^{1/2}\lambda n^{-1/2}$, then 
 \begin{align}\label{eq-m-09}
 \begin{aligned}
      	\p(\tilde{W}_{2}\leq z)-\p(\bar{W}_{2}\leq z)&\leq \e  \{h_{z, \varepsilon_{1}}(\tilde{W}_{2})\}-\e  \{h_{z, \varepsilon_{1}}(\bar{W}_{2})\}+\p(z\leq \bar{W}_{2}\leq z+\varepsilon_{1})\\
 	&:=Q_{1}+Q_{2},
 \end{aligned}
 \end{align}
 where $h_{z,\varepsilon_{1}}$ is defined by (\ref{eq-defininition-of-h}) and 
 \begin{align*}
Q_{1}&=\e \{ h_{z, \varepsilon_{1}}(\tilde{W}_{2})\}-\e \{ h_{z, \varepsilon_{1}}(\bar{W}_{2})\},\quad Q_{2}=\p(z\leq \bar{W}_{2}\leq z+\varepsilon_{1}).
 \end{align*}
 It follows from the definition of $\tilde{V}$ and $\bar{V}$ that
 \begin{align}\label{eq-m-10}
 	|Q_{1}|\leq \|h'\|_{\infty}\e|\tilde{W}_{2}-\bar{W}_{2}|\leq \frac{1}{\varepsilon_{1}}\e\Big\{|S|\cdot\Big|\frac{1}{\tilde{V}}-\frac{1}{\bar{V}}\Big|\Big\}\leq \frac{4}{\varepsilon_{1}\sigma^{3}}\e|S^{2}\bar{Y}|.
 \end{align}
 By Lemma \ref{lem-fourth-moment-for-S} and the inequality $ab\leq ca^{2}/2+b^{2}/(2c)$ with $a=S, b=S\cdot\sum_{i\in [n]}Y_{i}, c=\kappa\lambda\sigma$, we have
 \begin{align}\label{eq-m-11}
 	\e\Big|S^{2}\big(\sum_{i\in [n]}Y_{i}\big)\Big|\leq \frac{\kappa\lambda\sigma}{2}\cdot \e|S|^{2}+\frac{1}{2\kappa\lambda\sigma}\e\Big|S^{2}\big(\sum_{i\in [n]}Y_{i}\big)^{2}\Big|\leq 13\kappa\lambda\sigma^{3}.
 \end{align}
  Combining (\ref{eq-m-10}) and (\ref{eq-m-11}) yields
 \begin{align}\label{eq-m-12}
 	|Q_{1}|\leq C\kappa^{1/2}\lambda n^{-1/2}\leq \frac{C\lambda\kappa^{2}}{\sigma^{3}}\sum_{i\in [n]} \|X_{i}\|_{4}^{3}.
 \end{align}
 As for $Q_{2}$, observe that
 \begin{align}\label{eq-m-13}
 \begin{aligned}
      	|Q_{2}|&\leq |\Phi(z+\varepsilon_{1})-\Phi(z)|+2\sup_{z\in\R}|\p(\bar{W}_{2}\leq z)-\Phi(z)|\\
 	& \leq \frac{C\lambda\kappa^{2}}{\sigma^{3}}\sum_{i\in [n]} \|X_{i}\|_{4}^{3}+2\sup_{z\in\R}|\p(\bar{W}_{2}\leq z)-\Phi(z)|.
 \end{aligned}
 \end{align}
 By  (\ref{eq-m-09}), (\ref{eq-m-12}) and (\ref{eq-m-13}), we obtain an upper bound for $\p(\tilde{W}_{2}\leq z)-\p(\bar{W}_{2}\leq z)$, i.e.,
 \begin{align*}
 	\p(\tilde{W}_{2}\leq z)-\p(\bar{W}_{2}\leq z)\leq \frac{C\lambda\kappa^{2}}{\sigma^{3}}\sum_{i\in [n]} \|X_{i}\|_{4}^{3}+2\sup_{z\in\R}|\p(\bar{W}_{2}\leq z)-\Phi(z)|.
 \end{align*}
 The lower bound can be derived by a similar argument. Together with Proposition \ref{prop-BE-bar-W2}, we complete the proof of Theorem \ref{thm-main2}.
 \end{proof}

\subsection{Proofs of Theorems \textup{\ref{thm-graph-dependency}}, \textup{\ref{thm-distributed-U-statistics}}, \textup{\ref{thm-constrained-U-statistics}} and \textup{\ref{thm-app4}}} 
In this subsection, we present the proofs of  Theorems \textup{\ref{thm-graph-dependency}}, \textup{\ref{thm-distributed-U-statistics}}, \textup{\ref{thm-constrained-U-statistics}} and \textup{\ref{thm-app4}}.
\subsubsection{Proofs of Theorems  \textup{\ref{thm-graph-dependency}} and  \textup{\ref{thm-constrained-U-statistics}}}
\begin{proof}[Proof of Theorem \ref{thm-graph-dependency}]
   Let $A_{ij}=A_{i}\cup A_{j}$, then  $A_{ij}$ satisfies (LD2). Observe that $N_{i}=\{k\in \mathcal{V}: i\in A_{k}\}=A_{i}$ and for any $k\in [n]$,
	\begin{align}\label{eq-m1a-01}
	    \sum_{i\in [n]}\sum_{j\in A_{i}}\mathbf{1}(k\in A_{ij})\leq \sum_{i\in [n]}\sum_{j\in A_{i}}\mathbf{1}(k\in A_{i})+\sum_{i\in [n]}\sum_{j\in A_{i}}\mathbf{1}(k\in A_{i})\leq 2\kappa^{2},
	\end{align}
	which implies $\kappa\leq 2d$ and $\tau\leq 4d^{2}$. Applying  Theorems \ref{thm-main1-1} and \ref{thm-main2}, respectively, we have (\ref{eq-app2-01}) and (\ref{eq-app2-02}). 
\end{proof}
\begin{proof}[Proof of Theorem \textup{\ref{thm-constrained-U-statistics}}]
    By assumption $\lim_{n\to \infty}\text{Var}(U_{n}(f;\DD)/n^{2b-1}=\sigma_{f,\DD}^{2}>0$, there exists $N$ such that for any $n\geq N$, we have $\sigma_{n}^{2}:=\text{Var}(U_{n}(f; \DD))\geq n^{2b-1}\sigma_{f,\DD}^{2}/2.$
So to apply Theorems \ref{thm-main1-1} and  \ref{thm-main2}, we only need to estimate the order of $\kappa$ and $\tau.$ For any $i,j\in\II$, define
\begin{align*}
  A_{ij1}&=\{k\in \II,  s(k)\cap s(i)=\varnothing\ \&\  s(k)\cap  s(j)=\varnothing \ \text{and}\ |k-i|\leq m\},\nonumber\\
  A_{ij2}&=\{k\in \II,  s(k)\cap s(i)=\varnothing\ \&\  s(k)\cap  s(j)=\varnothing \ \text{and}\ |k-j|\leq m\}.
\end{align*}
Now for $i\in\II$ and $j\in A_{i}$, let $A_{ij}=A_{i1}\cup A_{j1}\cup A_{ij1}\cup A_{ij2}.$ It can be verified that $A_{i}$ and $A_{ij}$ satisfy (LD1) and (LD2). Note that $A_{ij1}\subset A_{i2}$ and $A_{ij2}\subset A_{j2}$,  to estimate the order of $A_{i}$ and $A_{ij}$, it suffices to estimate the number of $A_{i1}$ and $A_{i2}$. In fact, note that $l$ and $m$ are fixed, then
\begin{align}\label{eq-p-app2-2}
\begin{aligned}
      |A_{i1}|&=\sum_{p=1}^{l} |\{k\in \II, | s(k)\cap s(i)|=p\}|\\
  &\leq l\cdot |\{k\in \II, | s(k)\cap s(i)|=1\}|= O(n^{b-1}).
\end{aligned}
\end{align}
By the definition of $A_{i2}$, for any $j\in A_{i2}$, there exist $p\in s(i)$ and $q\in  s(j)$ such that $|p-q|\leq m$, which implies that
\begin{align}\label{eq-p-app2-3}
|A_{i2}|=O(n^{b-1}).
\end{align}
Combine (\ref{eq-p-app2-2}) and (\ref{eq-p-app2-3}), we have
\begin{align}\label{eq-p-app2-4}
  |A_{i}|\leq |A_{i1}|+|A_{i2}|=O(n^{b-1})\quad \text{and}\quad |A_{ij}|\leq |A_{i}|+|A_{j}|=O(n^{b-1}).
\end{align}
As for $N_{i}$, by the definition of $A_{i}$, we have
\begin{align}\label{eq-p-app2-5}
\begin{aligned}
     N_{i}&=\{k\in\II, i\in A_{k}\}\\
     &=\{k\in\II, i\in A_{k1}\}\cup \{k\in\II, i\in A_{k2}\}\\
&=\{k\in\II,  s(k)\cap s(i)\neq \varnothing\}\cup \{k\in\II,  s(k)\cap s(i)=\varnothing\ \text{and}\ |i-k|\leq m \}\\
&=A_{i1}\cup A_{i2}=A_{i}.
\end{aligned}
\end{align}
Combining (\ref{eq-p-app2-4}) and (\ref{eq-p-app2-5}) yields $\kappa=O(n^{b-1}).$ For $\tau$, note that $A_{ij}\subset A_{i}\cup A_{j}$, then
\begin{align*}
\sum_{i\in \II}\sum_{j\in A_{i}} \mathbf{1}(k\in A_{ij})\leq \sum_{i\in \II}\sum_{j\in A_{i}} \mathbf{1}(k\in A_{i})+\mathbf{1}(k\in A_{j})\leq 2\kappa^{2}.
\end{align*}
Then $\tau=O(n^{2b-2}).$ This proves Theorem \ref{thm-constrained-U-statistics} by Theorems \ref{thm-main1-1} and \ref{thm-main2}. 
\end{proof}
\subsubsection{Proof of Theorem \textup{\ref{thm-distributed-U-statistics}}}
Before proving Theorem \ref{thm-distributed-U-statistics}, we first apply Theorem \ref{thm-main1} to obtain a more general Berry--Esseen bound for distributed statistical inference. For any $1\leq i\leq k$, let $\JJ_{i}$ be a subset of $\mathbb{N}_{+}$. Moreover, let 
\[
\II_{i}=\{i\}\times \JJ_{i}\quad \text{and}\quad \II=\cup_{i=1}^{k}\II_{i}.
\]
We consider a filed of locally dependent random variables $\{X_{ij}: (i,j)\in \II\}$. Suppose that 
\begin{itemize}
   \item [(LD-D1)]  For any $(i,j)\in \II$, there exists $A_{i,j} \subset \JJ_{i}$ satisfying that $X_{ij}$ is independent of $\left\{X_{kl}\in \II: l\notin A_{i,j}\right\}$.
 \item [(LD-D2)]  For any $(i,j)\in \II$ and $(i,k)\in \{i\}\times A_{i,j}$, there exists $ A_{i,j}\subset A_{i, jk}\subset \JJ_{i}$ satisfying that $\left\{X_{ij}, X_{ik}\right\}$ is independent of $\{X_{pq}\in \II: q \notin A_{i,jk}\}$.
\end{itemize}
Assume that $\e X_{ij}=0$ for any $(i,j)\in \II.$ Let $S=\sum_{i=1}^{k}\sum_{j\in \JJ_{i}}X_{ij}$ and $\sigma^{2}=\mathrm{Var}(S)$. Moreover, for any $(i,j)\in\II$, let $N_{i,j}=\{(p,q)\in \II: (i,j)\in \{p\}\times A_{p,q}\}=\{(i,q): q\in \JJ_{i}, j\in A_{i,q}\}$. In addition, for any $1\leq i\leq k$, let $\kappa_{i}$ and $\tau_{i}$ be two positive constants such that 
\begin{align*}
\kappa_{i}&=\max\Big\{\sup_{j\in \JJ_{i}} |N_{i,j}|, \sup_{(j,k):j\in \JJ_{i}, k\in A_{i,j}} |A_{i,jk}| \Big\},\\\tau_{i}&=\max_{j\in \JJ_{i}} \big|(p,q)\in \JJ_{i}\times \JJ_{i};\ q\in A_{i,p},\  j\in A_{i,pq}\big|.
\end{align*}  We have the following proposition.
\begin{prop}\label{prop-distributed-inference}
If $\e|X_{ij}|^{4}<\infty$ for any $1\leq i\leq k$, $j\in \JJ_{i}$, then
\begin{align*}
	\sup_{z\in\R}|\p(S/\sigma\leq z)-\Phi(z)|\leq \frac{C}{\sigma^{3}}\sum_{i=1}^{k}\sum_{j\in \JJ_{i}}\kappa_{i}^{2}\|X_{ij}\|_{4}^{3}+\frac{C}{\sigma^{2}} \Big(\sum_{i=1}^{k}\sum_{j\in \JJ_{i}}(\kappa_{i}^{3}+\kappa\tau_{i})\|X_{ij}\|_{4}^{4} \Big)^{1/2}.
\end{align*}
\end{prop}
\begin{proof}
 Note that $\{X_{ij}, (i,j)\in \II\}$ can be regarded as a field of locally dependent random variables satisfying conditions (LD1) and (LD2). We apply Theorem \ref{thm-main1} to prove Proposition \ref{prop-distributed-inference}. To this end, we only need to provide  upper bounds for $\beta_{1}$, $\beta_{2}$ and $\beta_{3}.$ Some simple algebraic calculations give that
 \begin{align}\label{eq-kl1-001}
 	\beta_{1}\leq \frac{C}{\sigma^{3}}\sum_{i=1}^{k}\sum_{j\in \JJ_{i}}\kappa_{i}^{2}\|X_{ij}\|_{4}^{3}\quad \text{and}\quad 	\beta_{2}^{2}\leq \frac{C}{\sigma^{4}} \sum_{i=1}^{k}\sum_{j\in \JJ_{i}}(\kappa_{i}^{3}+\kappa_{i}\tau_{i})\|X_{ij}\|_{4}^{4}.
 \end{align}
 As for $\beta_{3}$, note that for any $(i,j),(i,l)\in \II$, by Cauchy inequality, we have 
 \begin{align}\label{eq-kl1-02}
 \begin{aligned}
     &\sum_{p\in A_{i,j}\cup N_{i,j}\cup N_{i,l}}\sum_{q\in N_{i,p}}\|X_{ip}\|_{4}\|X_{iq}\|_{4}\\
 	&\leq C\Big(\sum_{p\in A_{i,j}\cup N_{i,j}\cup N_{i,l}}\sum_{q\in N_{i,p}}(\|X_{ip}\|_{4}^{2}+\|X_{iq}\|_{4}^{2})^{2}\Big)^{1/2}\Big(\sum_{p\in A_{i,j}\cup N_{i,j}\cup N_{i,l}}\sum_{q\in N_{i,p}}1\Big)^{1/2}\\
 	&\leq C\Big(\kappa_{i}^{3}\sum_{p\in \JJ_{i}}\|X_{ip}\|_{4}^{4}\Big)^{1/2}\leq C \Big(\sum_{i=1}^{k}\sum_{j\in \JJ_{i}}(\kappa_{i}^{3}+\kappa\tau_{i})\|X_{ij}\|_{4}^{4}\Big)^{1/2}.
 \end{aligned}	
 \end{align}
 Similarly,  for any $(i,j)\in \II$, let 
 \begin{align*}
  D_{(i,j)}&=\{(p,q_{1},q_{2}); q_{2}\in A_{p,q_{1}}, \{i,j\}\in \{p\}\times A_{p,q_{1}q_{2}}  \}\nonumber\\
  &=\{(q_{1},q_{2})\in \JJ_{i}\times \JJ_{i}; q_{2}\in A_{i,q_{1}}, j\in  A_{i,q_{1}q_{2}}\},
 \end{align*}
 then
 \begin{align}\label{eq-kl1-03}
 \begin{aligned}
     &\sum_{(p,q)\in D_{(i,j)}} \|X_{ip}\|_{4}\|X_{iq}\|_{4}\\
     &\quad
     \leq 
    C\Big( \sum_{(p,q)\in D_{(i,j)}}(\|X_{ip}\|_{4}^{2}+\|X_{iq}\|_{4}^{2})^{2}\Big)^{1/2} \Big( \sum_{(p,q)\in D_{(i,j)}}1\Big)^{1/2}\\
    &\quad\leq C\Big( \tau_{i}\sum_{p\in \JJ_{i}}\sum_{q\in A_{i,p}} \big(\|X_{ip}\|_{4}^{4}+\|X_{iq}\|_{4}^{4}\big)\Big)^{1/2}\\
    &\quad\leq \frac{C}{\sigma^{2}} \Big(\sum_{i=1}^{k}\sum_{j\in \JJ_{i}}(\kappa_{i}^{3}+\kappa\tau_{i})\|X_{ij}\|_{4}^{4}\Big)^{1/2}.
 \end{aligned}    
 \end{align}
 Combining  (\ref{eq-kl1-001}), (\ref{eq-kl1-02}) and (\ref{eq-kl1-03}), we complete the proof of Proposition \ref{prop-distributed-inference}.
\end{proof}
Now, we are ready to prove Theorem \ref{thm-distributed-U-statistics}.
\begin{proof}[Proof of Theorem \ref{thm-distributed-U-statistics}]
    We apply Proposition \ref{prop-distributed-inference} to prove Theorem \ref{thm-distributed-U-statistics}.
 For each $1\leq i\leq k$, we denote $\mathcal{J}_{i}$ as the set consisting of all summation indicators in $i$-th subset, i.e.  $$\mathcal{J}_{i}:=\big\{j=(j_{1},j_{2},\cdots, j_{m})\in \R^{ m},\ 1\leq j_{1}<j_{2}<\cdots<j_{m}\leq n_{i}\ \text{and} \ X_{ij_{k}}\in \mathbb{S}_{i}, \forall k\in [m]\big\},$$
 then $|\JJ_{i}|=\binom{n_{i}}{m}.$
 Let $\mathcal{I}_{i}=\{i\}\times \JJ_{i}$ and $\mathcal{I}=\cup_{i=1}^{k}\mathcal{I}_{i}.$
Moreover, for any $j\in \mathcal{J}_{i}$, define $ s(j)=\{j_{1},j_{2},\cdots, j_{m}\}.$ To apply Proposition \ref{prop-distributed-inference}, we need to define $A_{i,j}$ and $A_{i,jk}$ so that the collection  $\{X_{ij}, (i,j)\in \II\}$ satisfies (LD-D1) and (LD-D2).  For any $(i,j)\in \mathcal{I}$, define
\[
A_{i,j}=\{p\in \mathcal{J}_{i}:  s(p)\cap s(j)\neq \varnothing\}.
\]
For each $(i,j)\in \mathcal{I}$ and $p\in A_{i,j}$, define
\[
A_{i,jk}=\{q\in \mathcal{J}_{i}:  s(q)\cap( s(j)\cup  s(p))\neq \varnothing\}.
\]
It is easy to check that, with $A_{i,j}$ and $A_{i,jk}$ defined above, the collection  $\{X_{ij}, (i,j)\in \II\}$ satisfies (LD-D1) and (LD-D2). We now derive upper bounds for $\kappa_{i}$ and $\tau_{i}.$
By the definition of $A_{i,j}$, we have
\begin{align*}
      |A_{i,j}|&=\binom{n_{i}}{m}-\binom{n_{i}-m}{m}=\sum_{t=1}^{m} \binom{n_{i}-m}{m-t}\binom{m}{t}\\
  &=\sum_{u=0}^{m-1} \frac{m}{u+1} \binom{n_{i}-m}{m-1-u}\binom{m-1}{u}\\
  &\leq m\binom{n_{i}-1}{m-1},
\end{align*}
where the second equality and the last inequality make use of the following equality: 
\begin{align*}
  \binom{n}{m}=\sum_{t=0}^{m} \binom{n-m}{m-t}\binom{m}{t},\quad \text{for all } n\geq m.
\end{align*}
 By the definition of $N_{i,j}$, we can check that $N_{i,j}=A_{i,j}$. In addition, note that
$A_{i,jk}\subset A_{i,j}\cup A_{i,k}$, then for any $(i,j)\in \mathcal{I}$, we have 
\begin{align*}
	&\sum_{p\in \mathcal{J}_{i}}\sum_{q\in A_{i,p}}\mathbf{1}\big(j\in  A_{i,pq}\big)
	\leq \sum_{p\in \mathcal{J}_{i}}\sum_{q\in A_{i,p}}\mathbf{1}\big(j\in  A_{i,p}\big)+\sum_{p\in \mathcal{J}_{i}}\sum_{q\in A_{i,p}}\mathbf{1}\big(j\in  A_{i,p}\big)
	\leq 2\kappa_{i}^{2}.
\end{align*}
  Combining above results, we have \begin{align}\label{eq-kl-04}
|\mathcal{J}_{i}|=\frac{n_{i}}{m}\binom{n_{i}-1}{m-1},\quad\kappa_{i}\leq Cm\binom{n_{i}-1}{m-1} \quad \text{and}\quad \tau_{i}\leq Cm^{2}\binom{n_{i}-1}{m-1}^{2}.
\end{align} Below we give a lower bound for  $\mathrm{Var}(U_{i})$.  Let $$
 W_{i}=\frac{m}{n_{i}}\sum_{j=1}^{n_{i}} g(X_{ij}),$$
then by Hoeffding decomposition (see \cite{Hoeffding19481}), we have 
\begin{align*}
	\mathrm{Var}(U_{N,i})\geq \mathrm{Var}(W_{i})=\frac{m^{2}}{n_{i}}\sigma_{1}^{2},
\end{align*}
which further implies
\begin{align}\label{eq-kl-01}
	\mathrm{Var}(W)=\frac{1}{m^{2}\sigma_{1}^{2}N}\sum_{i=1}^{k}n_{i}^{2}\mathrm{Var}(U_{N,i})\geq 1.
\end{align} 
  Applying Proposition \ref{prop-distributed-inference} and combining (\ref{eq-kl-04}) and (\ref{eq-kl-01}), we obtain Theorem \ref{thm-distributed-U-statistics}.
\end{proof}

\subsubsection{Proof of Theorem \textup{\ref{thm-app4}}}\label{sec-decorated} 
For any $\psi\in A_{\varphi}$, define 
\[
A_{\varphi,\psi}=\left\{\phi \in \mathcal{J}: e\left(G_{\phi} \cap (G_{\varphi}\cup G_{\psi})\right) \geq 1\right\}.
\]
Then $A_{\varphi,\psi}$ satisfies (LD2). Moreover, since $A_{\varphi,\psi}\subseteq A_{\varphi}\cup A_{\psi}$, it follows from \eqref{eq-m1a-01} that $\tau\le 2\kappa^{2}$. Finally, noting that $N_{\varphi}=A_{\varphi}$ and $|A_{\varphi}|=O(n^{v-2})$, we obtain $\kappa=O(n^{v-2})$. Consequently, Theorem \ref{thm-app4} follows from
Theorems \ref{thm-main1-1} and \ref{thm-main2}.

\begin{funding}
The research of Q.-M. Shao is partially supported by National Nature Science Foundation
of China NSFC 12031005 and Shenzhen Outstanding Talents Training Fund, China, and
the research of Z.-S. Zhang is partially supported by National Nature Science Foundation of China NSFC 12301183 and National Nature Science Found for Excellent Young
Scientists Fund.
\end{funding}

\bibliographystyle{imsart-nameyear} 
\bibliography{refs}       

\begin{thebibliography}{39}

\bibitem[\protect\citeauthoryear{Abraham, Delmas and
  Weibel}{2025}]{abraham2025probability}
\begin{barticle}[author]
\bauthor{\bsnm{Abraham},~\bfnm{Romain}\binits{R.}},
  \bauthor{\bsnm{Delmas},~\bfnm{Jean-Fran{\c{c}}ois}\binits{J.-F.}} \AND
  \bauthor{\bsnm{Weibel},~\bfnm{Julien}\binits{J.}}
(\byear{2025}).
\btitle{Probability-graphons: Limits of large dense weighted graphs}.
\bjournal{Innovations in Graph Theory}
\bvolume{2}
\bpages{25--117}.
\bdoi{10.5802/igt.7}
\end{barticle}
\endbibitem

\bibitem[\protect\citeauthoryear{Baldi and Rinott}{1989}]{BaldiandRinott1989}
\begin{barticle}[author]
\bauthor{\bsnm{Baldi},~\bfnm{Pierre}\binits{P.}} \AND
  \bauthor{\bsnm{Rinott},~\bfnm{Yosef}\binits{Y.}}
(\byear{1989}).
\btitle{{On Normal Approximations of Distributions in Terms of Dependency
  Graphs}}.
\bjournal{The Annals of Probability}
\bvolume{17}
\bpages{1646--1650}.
\bdoi{10.1214/aop/1176991178}
\end{barticle}
\endbibitem

\bibitem[\protect\citeauthoryear{Barbour, Karoński and
  Ruciński}{1989}]{BARBOUR1989125}
\begin{barticle}[author]
\bauthor{\bsnm{Barbour},~\bfnm{A.~D}\binits{A.~D.}},
  \bauthor{\bsnm{Karoński},~\bfnm{Michal}\binits{M.}} \AND
  \bauthor{\bsnm{Ruciński},~\bfnm{Andrzej}\binits{A.}}
(\byear{1989}).
\btitle{{A central limit theorem for decomposable random variables with
  applications to random graphs}}.
\bjournal{Journal of Combinatorial Theory, Series B}
\bvolume{47}
\bpages{125-145}.
\bdoi{https://doi.org/10.1016/0095-8956(89)90014-2}
\end{barticle}
\endbibitem

\bibitem[\protect\citeauthoryear{Bolthausen}{1982}]{Bolthausen1982}
\begin{barticle}[author]
\bauthor{\bsnm{Bolthausen},~\bfnm{E.}\binits{E.}}
(\byear{1982}).
\btitle{{On the Central Limit Theorem for Stationary Mixing Random Fields}}.
\bjournal{The Annals of Probability}
\bvolume{10}
\bpages{1047--1050}.
\bdoi{10.1214/aop/1176993726}
\end{barticle}
\endbibitem

\bibitem[\protect\citeauthoryear{B{\'o}na}{2007}]{bona2007copies}
\begin{bmisc}[author]
\bauthor{\bsnm{B{\'o}na},~\bfnm{Mikl{\'o}s}\binits{M.}}
(\byear{2007}).
\btitle{The copies of any permutation pattern are asymptotically normal}.
\bnote{Preprint. Available at \url{http://arxiv.org/abs/0712.2792}}.
\end{bmisc}
\endbibitem

\bibitem[\protect\citeauthoryear{Chen, Goldstein and
  Shao}{2011}]{chen2011normal}
\begin{bbook}[author]
\bauthor{\bsnm{Chen},~\bfnm{Louis H.~Y.}\binits{L.~H.~Y.}},
  \bauthor{\bsnm{Goldstein},~\bfnm{Larry}\binits{L.}} \AND
  \bauthor{\bsnm{Shao},~\bfnm{Qi-Man}\binits{Q.-M.}}
(\byear{2011}).
\btitle{Normal Approximation by Stein’s Method}.
\bseries{Probability and Its Applications}.
\bpublisher{Springer Berlin Heidelberg}.
\bdoi{10.1007/978-3-642-15007-4}
\end{bbook}
\endbibitem

\bibitem[\protect\citeauthoryear{Chen and Peng}{2021}]{chensongxi2021}
\begin{barticle}[author]
\bauthor{\bsnm{Chen},~\bfnm{Song~Xi}\binits{S.~X.}} \AND
  \bauthor{\bsnm{Peng},~\bfnm{Liuhua}\binits{L.}}
(\byear{2021}).
\btitle{{Distributed statistical inference for massive data}}.
\bjournal{The Annals of Statistics}
\bvolume{49}
\bpages{2851--2869}.
\bdoi{10.1214/21-AOS2062}
\end{barticle}
\endbibitem

\bibitem[\protect\citeauthoryear{Chen, R{\"o}llin and
  Xia}{2021}]{Chen-recursive}
\begin{barticle}[author]
\bauthor{\bsnm{Chen},~\bfnm{Louis H.~Y.}\binits{L.~H.~Y.}},
  \bauthor{\bsnm{R{\"o}llin},~\bfnm{Adrian}\binits{A.}} \AND
  \bauthor{\bsnm{Xia},~\bfnm{Aihua}\binits{A.}}
(\byear{2021}).
\btitle{{Palm theory, random measures and Stein couplings}}.
\bjournal{The Annals of Applied Probability}
\bvolume{31}
\bpages{2881--2923}.
\bdoi{10.1214/21-AAP1666}
\end{barticle}
\endbibitem

\bibitem[\protect\citeauthoryear{Chen and
  Shao}{2001}]{chenandshao2001non-uniform}
\begin{barticle}[author]
\bauthor{\bsnm{Chen},~\bfnm{Louis H.~Y.}\binits{L.~H.~Y.}} \AND
  \bauthor{\bsnm{Shao},~\bfnm{Qi-Man}\binits{Q.-M.}}
(\byear{2001}).
\btitle{A non-uniform Berry–Esseen bound via Stein’s method}.
\bjournal{Probability Theory and Related Fields}
\bvolume{120}
\bpages{236--254}.
\bdoi{10.1007/PL00008782}
\end{barticle}
\endbibitem

\bibitem[\protect\citeauthoryear{Chen and Shao}{2004}]{ChenandShao}
\begin{barticle}[author]
\bauthor{\bsnm{Chen},~\bfnm{Louis H.~Y.}\binits{L.~H.~Y.}} \AND
  \bauthor{\bsnm{Shao},~\bfnm{Qi-Man}\binits{Q.-M.}}
(\byear{2004}).
\btitle{{Normal approximation under local dependence}}.
\bjournal{The Annals of Probability}
\bvolume{32}
\bpages{1985--2028}.
\bdoi{10.1214/009117904000000450}
\end{barticle}
\endbibitem

\bibitem[\protect\citeauthoryear{Chen and Shao}{2007}]{chenandshao2007normal}
\begin{barticle}[author]
\bauthor{\bsnm{Chen},~\bfnm{Louis H.~Y.}\binits{L.~H.~Y.}} \AND
  \bauthor{\bsnm{Shao},~\bfnm{Qi-Man}\binits{Q.-M.}}
(\byear{2007}).
\btitle{{Normal approximation for nonlinear statistics using a concentration
  inequality approach}}.
\bjournal{Bernoulli}
\bvolume{13}
\bpages{581--599}.
\bdoi{10.3150/07-BEJ5164}
\end{barticle}
\endbibitem

\bibitem[\protect\citeauthoryear{Dufour and Olhede}{2024}]{dufour2024inference}
\begin{bmisc}[author]
\bauthor{\bsnm{Dufour},~\bfnm{Charles}\binits{C.}} \AND
  \bauthor{\bsnm{Olhede},~\bfnm{Sofia~C.}\binits{S.~C.}}
(\byear{2024}).
\btitle{Inference for decorated graphs and application to multiplex networks}.
\bnote{Preprint. Available at \url{http://arxiv.org/abs/2408.12339}}.
\end{bmisc}
\endbibitem

\bibitem[\protect\citeauthoryear{Eichelsbacher and
  Redno{\ss}}{2023}]{Peter-Kolmogorov-bounds-2023}
\begin{barticle}[author]
\bauthor{\bsnm{Eichelsbacher},~\bfnm{Peter}\binits{P.}} \AND
  \bauthor{\bsnm{Redno{\ss}},~\bfnm{Benedikt}\binits{B.}}
(\byear{2023}).
\btitle{{Kolmogorov bounds for decomposable random variables and subgraph
  counting by the Stein–Tikhomirov method}}.
\bjournal{Bernoulli}
\bvolume{29}
\bpages{1821--1848}.
\bdoi{10.3150/22-BEJ1522}
\end{barticle}
\endbibitem

\bibitem[\protect\citeauthoryear{Fang}{2016}]{Fang2016Multivariate}
\begin{barticle}[author]
\bauthor{\bsnm{Fang},~\bfnm{Xiao}\binits{X.}}
(\byear{2016}).
\btitle{A Multivariate CLT for Bounded Decomposable Random Vectors with the
  Best Known Rate}.
\bjournal{Journal of Theoretical Probability}
\bvolume{29}
\bpages{1510--1523}.
\bdoi{10.1007/s10959-015-0619-7}
\end{barticle}
\endbibitem

\bibitem[\protect\citeauthoryear{Fang}{2019}]{10.1214/19-EJP301}
\begin{barticle}[author]
\bauthor{\bsnm{Fang},~\bfnm{Xiao}\binits{X.}}
(\byear{2019}).
\btitle{{Wasserstein-2 bounds in normal approximation under local dependence}}.
\bjournal{Electronic Journal of Probability}
\bvolume{24}
\bpages{1--14}.
\bdoi{10.1214/19-EJP301}
\end{barticle}
\endbibitem

\bibitem[\protect\citeauthoryear{Flajolet, Szpankowski and
  Vall\'{e}e}{2006}]{flajolet2006hidden}
\begin{barticle}[author]
\bauthor{\bsnm{Flajolet},~\bfnm{Philippe}\binits{P.}},
  \bauthor{\bsnm{Szpankowski},~\bfnm{Wojciech}\binits{W.}} \AND
  \bauthor{\bsnm{Vall\'{e}e},~\bfnm{Brigitte}\binits{B.}}
(\byear{2006}).
\btitle{Hidden word statistics}.
\bjournal{J. ACM}
\bvolume{53}
\bpages{147–-183}.
\bdoi{10.1145/1120582.1120586}
\end{barticle}
\endbibitem

\bibitem[\protect\citeauthoryear{Hoeffding}{1948}]{Hoeffding19481}
\begin{barticle}[author]
\bauthor{\bsnm{Hoeffding},~\bfnm{Wassily}\binits{W.}}
(\byear{1948}).
\btitle{{A Class of Statistics with Asymptotically Normal Distribution}}.
\bjournal{The Annals of Mathematical Statistics}
\bvolume{19}
\bpages{293--325}.
\bdoi{10.1214/aoms/1177730196}
\end{barticle}
\endbibitem

\bibitem[\protect\citeauthoryear{Hoeffding and Robbins}{1948}]{Hoeffding1948}
\begin{barticle}[author]
\bauthor{\bsnm{Hoeffding},~\bfnm{Wassily}\binits{W.}} \AND
  \bauthor{\bsnm{Robbins},~\bfnm{Herbert}\binits{H.}}
(\byear{1948}).
\btitle{{The central limit theorem for dependent random variables}}.
\bjournal{Duke Mathematical Journal}
\bvolume{15}
\bpages{773--780}.
\bdoi{10.1215/S0012-7094-48-01568-3}
\end{barticle}
\endbibitem

\bibitem[\protect\citeauthoryear{Hofer}{2018}]{Hofer2017ACL}
\begin{barticle}[author]
\bauthor{\bsnm{Hofer},~\bfnm{Lisa}\binits{L.}}
(\byear{2018}).
\btitle{A Central Limit Theorem for Vincular Permutation Patterns}.
\bjournal{Discrete Mathematics \& Theoretical Computer Science}
\bvolume{Vol. 19 no. 2, Permutation Patterns 2016}.
\bdoi{10.23638/DMTCS-19-2-9}
\end{barticle}
\endbibitem

\bibitem[\protect\citeauthoryear{Jacquet and
  Szpankowski}{2015}]{jacquet2015analytic}
\begin{bbook}[author]
\bauthor{\bsnm{Jacquet},~\bfnm{Philippe}\binits{P.}} \AND
  \bauthor{\bsnm{Szpankowski},~\bfnm{Wojciech}\binits{W.}}
(\byear{2015}).
\btitle{Analytic Pattern Matching: From DNA to Twitter}.
\bpublisher{Cambridge University Press}, \baddress{Cambridge}.
\bdoi{10.1017/CBO9780511843204}
\end{bbook}
\endbibitem

\bibitem[\protect\citeauthoryear{Janisch and
  Leh{\'e}ricy}{2024}]{JanischLehericy2024}
\begin{barticle}[author]
\bauthor{\bsnm{Janisch},~\bfnm{Maximilian}\binits{M.}} \AND
  \bauthor{\bsnm{Leh{\'e}ricy},~\bfnm{Thomas}\binits{T.}}
(\byear{2024}).
\btitle{Berry--Esseen-Type Estimates for Random Variables with a Sparse
  Dependency Graph}.
\bjournal{Journal of Theoretical Probability}
\bvolume{37}
\bpages{3627--3653}.
\bdoi{10.1007/s10959-024-01363-z}
\end{barticle}
\endbibitem

\bibitem[\protect\citeauthoryear{Janson}{2023}]{janson_2023}
\begin{barticle}[author]
\bauthor{\bsnm{Janson},~\bfnm{Svante}\binits{S.}}
(\byear{2023}).
\btitle{{Asymptotic normality for $m$-dependent and constrained $U$-statistics,
  with applications to pattern matching in random strings and permutations}}.
\bjournal{Advances in Applied Probability}
\bvolume{55}
\bpages{841–894}.
\bdoi{10.1017/apr.2022.51}
\end{barticle}
\endbibitem

\bibitem[\protect\citeauthoryear{Kunszenti-Kov{\'a}cs, Lov{\'a}sz and
  Szegedy}{2022}]{kunszentikovacs2022multigraph}
\begin{barticle}[author]
\bauthor{\bsnm{Kunszenti-Kov{\'a}cs},~\bfnm{D{\'a}vid}\binits{D.}},
  \bauthor{\bsnm{Lov{\'a}sz},~\bfnm{L{\'a}szl{\'o}}\binits{L.}} \AND
  \bauthor{\bsnm{Szegedy},~\bfnm{Bal{\'a}zs}\binits{B.}}
(\byear{2022}).
\btitle{Multigraph limits, unbounded kernels, and Banach space decorated
  graphs}.
\bjournal{Journal of Functional Analysis}
\bvolume{282}
\bpages{109284}.
\bdoi{10.1016/j.jfa.2021.109284}
\end{barticle}
\endbibitem

\bibitem[\protect\citeauthoryear{L{\'e}vy}{1935}]{levy1935proprietes}
\begin{barticle}[author]
\bauthor{\bsnm{L{\'e}vy},~\bfnm{Paul}\binits{P.}}
(\byear{1935}).
\btitle{Propri{\'e}t{\'e}s asymptotiques des sommes de variables al{\'e}atoires
  ind{\'e}pendantes ou encha{\^\i}n{\'e}es}.
\bjournal{Journal de Math{\'e}matiques Pures et Appliqu{\'e}es}
\bvolume{14}
\bpages{347--402}.
\bnote{Available at
  \url{https://www.numdam.org/item/JMPA_1935_9_14_1-4_347_0/}}.
\end{barticle}
\endbibitem

\bibitem[\protect\citeauthoryear{Lin and Xi}{2010}]{LIN201016}
\begin{barticle}[author]
\bauthor{\bsnm{Lin},~\bfnm{N.}\binits{N.}} \AND
  \bauthor{\bsnm{Xi},~\bfnm{R.}\binits{R.}}
(\byear{2010}).
\btitle{{Fast surrogates of U-statistics}}.
\bjournal{{Computational Statistics $\&$ Data Analysis}}
\bvolume{54}
\bpages{16-24}.
\bdoi{https://doi.org/10.1016/j.csda.2009.08.009}
\end{barticle}
\endbibitem

\bibitem[\protect\citeauthoryear{Liu and Austern}{2023}]{Wasserstein-p(2023)}
\begin{barticle}[author]
\bauthor{\bsnm{Liu},~\bfnm{Tianle}\binits{T.}} \AND
  \bauthor{\bsnm{Austern},~\bfnm{Morgane}\binits{M.}}
(\byear{2023}).
\btitle{{Wasserstein-p bounds in the central limit theorem under local
  dependence}}.
\bjournal{Electronic Journal of Probability}
\bvolume{28}
\bpages{1--47}.
\bdoi{10.1214/23-EJP1009}
\end{barticle}
\endbibitem

\bibitem[\protect\citeauthoryear{Liu and
  Zhang}{2023}]{liuandzhang2023cramer-typemoderate}
\begin{barticle}[author]
\bauthor{\bsnm{Liu},~\bfnm{Song-Hao}\binits{S.-H.}} \AND
  \bauthor{\bsnm{Zhang},~\bfnm{Zhuo-Song}\binits{Z.-S.}}
(\byear{2023}).
\btitle{{Cramér-type moderate deviations under local dependence}}.
\bjournal{The Annals of Applied Probability}
\bvolume{33}
\bpages{4747--4797}.
\bdoi{10.1214/23-AAP1931}
\end{barticle}
\endbibitem

\bibitem[\protect\citeauthoryear{Lov{\'a}sz and
  Szegedy}{2010}]{lovasz2010limits}
\begin{bmisc}[author]
\bauthor{\bsnm{Lov{\'a}sz},~\bfnm{L{\'a}szl{\'o}}\binits{L.}} \AND
  \bauthor{\bsnm{Szegedy},~\bfnm{Bal{\'a}zs}\binits{B.}}
(\byear{2010}).
\btitle{Limits of compact decorated graphs}.
\bnote{Preprint. Available at \url{http://arxiv.org/abs/1010.5155}}.
\end{bmisc}
\endbibitem

\bibitem[\protect\citeauthoryear{Nicodème, Salvy and
  Flajolet}{2002}]{nicodeme2002motif}
\begin{barticle}[author]
\bauthor{\bsnm{Nicodème},~\bfnm{Pierre}\binits{P.}},
  \bauthor{\bsnm{Salvy},~\bfnm{Bruno}\binits{B.}} \AND
  \bauthor{\bsnm{Flajolet},~\bfnm{Philippe}\binits{P.}}
(\byear{2002}).
\btitle{Motif statistics}.
\bjournal{Theoretical Computer Science}
\bvolume{287}
\bpages{593--617}.
\bnote{Algorthims}.
\bdoi{https://doi.org/10.1016/S0304-3975(01)00264-X}
\end{barticle}
\endbibitem

\bibitem[\protect\citeauthoryear{Rai\v{c}}{2004}]{Raic2004}
\begin{barticle}[author]
\bauthor{\bsnm{Rai\v{c}},~\bfnm{M.}\binits{M.}}
(\byear{2004}).
\btitle{A Multivariate CLT for Decomposable Random Vectors with Finite Second
  Moments}.
\bjournal{Journal of Theoretical Probability}
\bvolume{17}
\bpages{573--603}.
\bdoi{10.1023/B:JOTP.0000040290.44087.68}
\end{barticle}
\endbibitem

\bibitem[\protect\citeauthoryear{{Rosenblatt}}{1956}]{PNAS...42...43R}
\begin{barticle}[author]
\bauthor{\bsnm{{Rosenblatt}},~\bfnm{M.}\binits{M.}}
(\byear{1956}).
\btitle{{a Central Limit Theorem and a Strong Mixing Condition}}.
\bjournal{Proceedings of the National Academy of Science}
\bvolume{42}
\bpages{43--47}.
\bdoi{10.1073/pnas.42.1.43}
\end{barticle}
\endbibitem

\bibitem[\protect\citeauthoryear{Ross}{2011}]{ross2011}
\begin{barticle}[author]
\bauthor{\bsnm{Ross},~\bfnm{Nathan}\binits{N.}}
(\byear{2011}).
\btitle{{Fundamentals of Stein’s method}}.
\bjournal{Probability Surveys}
\bvolume{8}
\bpages{210--293}.
\bdoi{10.1214/11-PS182}
\end{barticle}
\endbibitem

\bibitem[\protect\citeauthoryear{Ruci{\'n}ski}{1988}]{Rucinski1988}
\begin{barticle}[author]
\bauthor{\bsnm{Ruci{\'n}ski},~\bfnm{Andrzej}\binits{A.}}
(\byear{1988}).
\btitle{When are small subgraphs of a random graph normally distributed?}
\bjournal{Probability Theory and Related Fields}
\bvolume{78}
\bpages{1--10}.
\bdoi{10.1007/BF00718031}
\end{barticle}
\endbibitem

\bibitem[\protect\citeauthoryear{Shao and Zhang}{2022}]{shaoandzhang2022berry}
\begin{barticle}[author]
\bauthor{\bsnm{Shao},~\bfnm{Qi-Man}\binits{Q.-M.}} \AND
  \bauthor{\bsnm{Zhang},~\bfnm{Zhuo-Song}\binits{Z.-S.}}
(\byear{2022}).
\btitle{{Berry–Esseen bounds for multivariate nonlinear statistics with
  applications to M-estimators and stochastic gradient descent algorithms}}.
\bjournal{Bernoulli}
\bvolume{28}
\bpages{1548--1576}.
\bdoi{10.3150/21-BEJ1336}
\end{barticle}
\endbibitem

\bibitem[\protect\citeauthoryear{Shao and Zhang}{2025}]{ShaoZhang2025}
\begin{barticle}[author]
\bauthor{\bsnm{Shao},~\bfnm{Qi-Man}\binits{Q.-M.}} \AND
  \bauthor{\bsnm{Zhang},~\bfnm{Zhuo-Song}\binits{Z.-S.}}
(\byear{2025}).
\btitle{Berry-Esseen bounds for functionals of independent random variables}.
\bjournal{Stochastic Processes and their Applications}
\bvolume{184}
\bpages{104574}.
\bdoi{10.1016/j.spa.2025.104574}
\end{barticle}
\endbibitem

\bibitem[\protect\citeauthoryear{Shao and Zhou}{2016}]{shanandzhou2016}
\begin{barticle}[author]
\bauthor{\bsnm{Shao},~\bfnm{Qi-Man}\binits{Q.-M.}} \AND
  \bauthor{\bsnm{Zhou},~\bfnm{Wen-Xin}\binits{W.-X.}}
(\byear{2016}).
\btitle{{Cramér type moderate deviation theorems for self-normalized
  processes}}.
\bjournal{Bernoulli}
\bvolume{22}
\bpages{2029--2079}.
\bdoi{10.3150/15-BEJ719}
\end{barticle}
\endbibitem

\bibitem[\protect\citeauthoryear{Su, Ulyanov and
  Wang}{2025}]{SuUlyanovWang2025PerturbStein}
\begin{barticle}[author]
\bauthor{\bsnm{Su},~\bfnm{Z.}\binits{Z.}},
  \bauthor{\bsnm{Ulyanov},~\bfnm{V.~V.}\binits{V.~V.}} \AND
  \bauthor{\bsnm{Wang},~\bfnm{X.}\binits{X.}}
(\byear{2025}).
\btitle{On Approximation of Sums of Locally Dependent Random Variables via
  Perturbations of Stein Operator}.
\bjournal{Theory of Probability and Its Applications}
\bvolume{70}
\bpages{24--36}.
\bdoi{10.1137/S0040585X97T992215}
\end{barticle}
\endbibitem

\bibitem[\protect\citeauthoryear{Tem{\v{c}}inas, Nanda and
  Reinert}{2024}]{TemcinasNandaReinert2024}
\begin{barticle}[author]
\bauthor{\bsnm{Tem{\v{c}}inas},~\bfnm{Tadas}\binits{T.}},
  \bauthor{\bsnm{Nanda},~\bfnm{Vidit}\binits{V.}} \AND
  \bauthor{\bsnm{Reinert},~\bfnm{Gesine}\binits{G.}}
(\byear{2024}).
\btitle{Multivariate central limit theorems for random clique complexes}.
\bjournal{Journal of Applied and Computational Topology}
\bvolume{8}
\bpages{1837--1880}.
\bdoi{10.1007/s41468-023-00146-5}
\end{barticle}
\endbibitem

\bibitem[\protect\citeauthoryear{Zhang}{2024}]{zhang2021berryesseen}
\begin{barticle}[author]
\bauthor{\bsnm{Zhang},~\bfnm{Zhuo-Song}\binits{Z.-S.}}
(\byear{2024}).
\btitle{Berry-Esseen bounds for self-normalized sums of locally dependent
  random variables}.
\bjournal{Science China Mathematics}
\bvolume{67}
\bpages{2629--2652}.
\bdoi{10.1007/s11425-023-2189-9}
\end{barticle}
\endbibitem

\end{thebibliography}





\end{document}